%% file: div-free-webs.tex
\pdfoutput=1
\documentclass[11pt]{article}

\usepackage[a4paper,%
  outer=25mm,%
  inner=25mm,%
  top=20mm,%
  bottom=20mm%
]{geometry}

\usepackage[T1]{fontenc}
\usepackage[utf8]{inputenc}
\usepackage{lmodern}
\usepackage[%
  backend=biber,%
  language=english,%
  autolang=other,%
  abbreviate=true,%
	date=year,%
  isbn=false,%
  doi=false,%
	url=false,%
  giveninits=true%
]{biblatex}
\renewbibmacro{in:}{}

\usepackage{amssymb}
\usepackage{amsfonts}
\usepackage{amsthm}
\usepackage{stmaryrd}
\usepackage{mathtools}
\usepackage{cancel}
\usepackage{tikz-cd}

\usepackage{enumitem}

\usepackage[bookmarks]{hyperref}

\title{Local invariants of divergence-free webs}
\author{Wojciech Domitrz, Marcin Zubilewicz}
\date{\today}

\theoremstyle{plain}
\newtheorem{thm}{Theorem}
\newtheorem{lem}[thm]{Lemma}
\newtheorem{prop}[thm]{Proposition}
\newtheorem{propdef}[thm]{Proposition/Definition}
\newtheorem{cor}[thm]{Corollary}
\theoremstyle{definition}
\newtheorem{defn}[thm]{Definition}
\newtheorem{exmp}{Example}
\theoremstyle{remark}
\newtheorem*{rem}{Remark}

\numberwithin{equation}{section}

\let\oldsection\section%
\renewcommand{\section}{%
  \renewcommand{\theequation}{\thesection.\arabic{equation}}%
  \oldsection}%

\graphicspath{{./graphics/}}

\newcommand{\ifstr}[2]{\def\cmpstrone{#1}\def\cmpstrtwo{#2}%
\ifx\cmpstrone\cmpstrtwo}

{
  \par\medskip
}

\newcommand{\Rb}[1]{\ifstr{#1}{1}\mathbb{R}\else\mathbb{R}^{#1}\fi}
\newcommand{\maps}[3]{#1:#2\rightarrow\,#3}
\newcommand{\set}[1]{\left\{#1\right\}}
\newcommand{\smset}[1]{\textstyle\{#1\}}
\newcommand{\ival}[1]{[ #1 ]}
\newcommand{\smfrac}[2]{{\textstyle\frac{#1}{#2}}}
\newcommand{\basis}[1]{\textstyle{\smash{\frac{\partial}{\partial#1}}}}
\newcommand{\smsum}[1]{\smash{\textstyle\sum_{#1}}\,}

\DeclarePairedDelimiter\abs{\lvert}{\rvert}
\DeclarePairedDelimiter\norm{\lVert}{\rVert}

\newcommand{\Fol}{\mathfrak{F}}
\newcommand{\fol}{\mathcal{F}}
\newcommand{\gol}{\mathcal{G}}
\newcommand{\hol}{\mathcal{H}}
\newcommand{\web}{\mathcal{W}}

\newcommand{\csep}{m}

\newcommand{\kcurv}{\mathcal{K}}
\newcommand{\Vol}[1]{\operatorname{Vol}^{#1}}

\newcommand*{\Rc}{\mathrm{Rc}}

\newcommand{\smint}[1]{{\textstyle\int_{#1}}}
\newcommand{\Int}{\int\limits}
\newcommand{\IterInt}[4]{%
  \int\limits_{#1}^{#2}\hskip-.6ex\cdots\hskip-.6ex\int\limits_{#3}^{#4}}

\newcommand{\Fols}{\operatorname{Fol}}
\newcommand{\codim}{\operatorname{codim}}

\newcommand{\iprod}{\mathbin{\lrcorner}}

\begin{document}

  \maketitle
  \begin{abstract}
    The objects of our study are webs in the geometry of volume-preserving
    diffeomorphisms. We introduce two local invariants of divergence-free webs:
    a differential one, directly related to the curvature of the natural
    connection of a divergence-free $2$-web introduced by S. Tabachnikov
    (1992), and a geometric one, inspired by the classical notion of planar
    $3$-web holonomy defined by W. Blaschke and G. Thomsen (1928). We show that
    triviality of either of these invariants characterizes trivial
    divergence-free web-germs up to equivalence. We also establish some
    preliminary results regarding the full classification problem, which
    jointly generalize the theorem of S. Tabachnikov on normal forms of
    divergence-free $2$-webs. They are used to provide a canonical form and a
    complete set of invariants of a generic divergence-free web in the planar
    case. Lastly, the relevance of local triviality conditions and their
    potential applications in numerical relativity are discussed.
  \end{abstract}

\section{Introduction}

Many lines of current research in differential geometry will inevitably lead to
at least a passing consideration of almost-product structures. These objects,
which take their name from a~splitting of the tangent bundle into a direct sum
of two complementary distributions they induce, lie at the foundation of
several geometries of considerable, continuing interest. Examples include the
geometry of symplectic pairs, in which one subbundle is the skew-orthogonal
complement of the other with respect to the ambient symplectic form
\cite{bandespairgeom}; bi-Lagrangian manifolds, which come equipped with two
complementary foliations with Lagrangian leaves \cite{etayosantamaria,
bilagrangian, vaisman} acting as \emph{double polarizations} in the
quantum-mechanical setting \cite{hess}; para-Kähler structures
\cite{anciauxromon, paracomplex}, i.e. integrable almost-product structures
with compatible neutral metrics, which turn out to be equivalent to
bi-Lagrangian structures \cite{etayosantamaria}; and foliated
(pseudo-)Riemannian manifolds \cite{tondeurriem}, widely used in numerical
relativity to fix a privileged ``time'' axis by means of families of
simultaneity hypersurfaces \cite{andersson1, adm, 3p1, lefloch}.

Over the years, much effort has been made in order to understand each of the
aforementioned structures individually. To complement these studies, it might
be fruitful to adopt a broader point of view and set out to document their
shared traits. For instance: each of these objects allows us to define a
canonical volume form on the underlying manifold. This volume form, taken
together with two complementary foliations generated by the integrable
subbundles induced by the almost-product structures, constitute what S.
Tabachnikov named a \emph{divergence-free $2$-web} in his work on local
properties of bi-Lagrangian and bi-Legendrian manifolds \cite{2-webs}.
This paper aims to expand upon his work by investigating known geometrical
invariants of divergence-free webs, introducing novel ones and providing
potential applications in computational physics.

In physics, divergence-free $2$-webs accompanying the metric structure appear
naturally in the context of numerical relativity. A choice of a spacelike
foliation by hypersurfaces of simultaneity, its timelike 1-dimensional
complement and the associated Lorenzian volume element, lead naturally to an
invariant formulation of differential equations governing the stress-energy
tensor with respect to ``time''. Among them are the energy-momentum
conservation laws expressed using the divergence operator
$\operatorname{div}X$, which is defined in terms of the volume element $dV$ by
the equality $\mathcal{L}_X dV = (\operatorname{div}X)\cdot dV$ for the Lie
derivative $\mathcal{L}_X$ along the vector field $X$. Insight into structures
combining the two foliations and the volume element might simplify handling
divergences in actual computations and pave the way for further improvements in
numerical methods designed with these conservation laws in mind.

In order to motivate further discussion, we will now introduce a generalization
of the notion of divergence-free web which accommodates structures with
larger number of foliations. Let us state the relevant definitions, together
with some elementary properties and examples of such objects.

An \emph{$n$-web} is a smooth manifold $M$ of dimension $m$ equipped with a
collection of $n$ foliations $\fol_1,\ldots,\fol_n$ of $M$. In this paper we
consider only \emph{regular} webs, for which the following \emph{general
position condition} holds: for every $p\in M$
\begin{equation}
  \label{eq:dfw-gen-pos}
  \codim\bigcap_{i=1}^n T_p\fol_i = \sum_{i=1}^n \codim T_p\fol_i.
\end{equation}
In this paper, webs will not be distinguished according to the ordering of the
constituent foliations, although an explicit ordering will prove useful in
actual computations. Hence we notate them as either $\web=(M,\Fol)$ with
$\Fol=\Fols(\web)=\set{\fol_1,\ldots,\fol_n}$ or
$\web=(M,\fol_1,\ldots,\fol_n)$ according to circumstances.

The \emph{codimension} of a regular $n$-web $\web$ is defined as the
nondecreasing sequence of the numbers $\codim\fol$ for $\fol\in\Fols(\web)$. If
$\codim\web = (k,k,\ldots,k)$, the web $\web$ is said to have \emph{codimension
$k$}. A regular $n$-web together with a nowhere-vanishing volume form on $M$
will be called a \emph{divergence-free $n$-web}, following \cite{2-webs}. In
most circumstances, a regular web $\web$ equipped with a volume form $\Omega$
will be notated by the symbol $\web_\Omega$.

Given two webs $\web_M=(M,\fol_1,\ldots,\fol_n)$ and
$\web_N=(N,\gol_1,\ldots,\gol_n)$, a diffeomorphism $\maps{\varphi}{M}{N}$
carrying the leaves of $\fol_j$ onto the leaves of $\gol_{\sigma(j)}$ for some
permutation $\sigma\in S_n$ will be called an \emph{equivalence}
between $\web_M$ and $\web_N$. If, in addition, $\web_M, \web_N$ come equipped
with volume forms $\Omega_M, \Omega_N$ and $\varphi^*(\Omega_N)=\Omega_M$, then
$\varphi$ is called an \emph{equivalence of divergence-free webs}. We will
mostly focus on \emph{local equivalences}, which correspond to germs of the
above maps.

Our variant of the general position condition (\ref{eq:dfw-gen-pos}) guarantees
that every regular $n$-web $\web=(M,\fol_1,\ldots,\fol_n)$ is \emph{locally
trivial} in the sense that for each point $p\in M$ one can find a
\emph{$\web$-adapted local coordinate chart} $(U,\xi)$ with
$\xi(x)=(x_1,x_2,\ldots,x_m)$. It is defined as a chart which is
\begin{enumerate}[label=$(\arabic{enumi})$, nosep]
  \item \emph{centered at $p$}, meaning that $\xi(p)=0$;
  \item \emph{cubical}, i.e.  $\xi(U) = \prod_{i=1}^m(a_i,b_i)$ for some
    $a,b\in\Rb{m}$;
  \item such that each plaque $L\in\fol_i$ has the form
    \begin{equation}
      L=\smset{x\in\Rb{m}:
        \forall_{k=1,\ldots,\codim\fol_i}\ x_{\csep_i+k}=\gamma_k}
    \end{equation}
    for some real constants $\gamma_k\in\Rb{1}$, where $\csep_j =
    \sum_{k=1}^{j-1}\codim\fol_k$ for $j=1,2,\ldots,n+1$.
\end{enumerate}
\vspace*{1ex}
This diffeomorphism $\maps{\xi}{U}{C=\prod_{i}(a_i,b_i)}$ can be
interpreted as a local equivalence between $\web$ and a \emph{standard $n$-web
$\web_0$ of codimension} $\codim\web$ on $C\subseteq\Rb{m}$. One similarly
defines a \emph{standard divergence-free $n$-web $\web_{0;\Lambda}$} as the web
$\web_0$ equipped with the Lebesgue volume form $\Lambda = dx_1\wedge
dx_2\wedge\cdots\wedge dx_m$; a divergence-free web
$\web_\Omega$ is \emph{locally trivial} if it is locally equivalent to
$\web_{0;\Lambda}$.

\begin{exmp}
  A symplectic pair \cite{bandespairgeom} is a smooth manifold $M$ of dimension
  $2n$ equipped with a pair of closed $2$-forms $\omega,\eta$ which satisfy
  $TM=\ker\omega\oplus\ker\eta$. Forming their sum $\Omega_+ = \omega+\eta$ is
  one way to give $M$ a structure of a symplectic manifold. The two kernel
  distributions integrate to a pair of complementary foliations $\fol,\gol$
  which, taken together with $\Omega_+^n$, form a divergence-free $2$-web
  $\web_{\Omega} = (M,\fol,\gol,\Omega_+^n)$.
  In $\web_{\Omega}$-adapted coordinates $(x_1,\ldots,x_{2k},
  y_1,\ldots,y_{2(n-k)})$ one can express $\omega$ and $\eta$ as
  \begin{equation}
    \omega = \sum_{\mathclap{1\leq i<j\leq 2k}}
      \, f_{ij}(x,y)\,dx_i\wedge dx_j,\qquad
    \eta = \sum_{\mathclap{1\leq i<j \leq 2(n-k)}}
      \, g_{ij}(x,y)\,dy_i\wedge dy_j.
  \end{equation}
  The equalities $d\omega=d\eta=0$ yield $f_{ij}=f_{ij}(x)$ and
  $g_{ij}=g_{ij}(y)$, hence by applying the classical Darboux theorem
  independently to $\omega$ and $\eta$ we bring $\Omega_+$ into the standard
  form $\sum_{i=1}^k dx_{2i-1}\wedge dx_{2i} + \sum_{j=1}^{n-k}dy_{2j-1}\wedge
  dy_{2j}$. Since symplectic pairs are locally trivial, their associated
  divergence-free $2$-webs are also locally trivial.
\end{exmp}

As in the case of regular $n$-webs, every divergence-free $n$-web $\web$
with $\sum_{\fol\in\Fols(\web)}\codim\fol < \dim M$ is locally equivalent
to a standard one: by picking a $\web$-adapted coordinate system
$(x_1,x_2,\ldots,x_m)$ on $U\subseteq M$ and stretching the space along the
$x_m$-axis by an appropriately chosen transformation of the form $x\mapsto
y=(x_1,\ldots,x_{m-1},y_m(x))$ one exhibits a local equivalence between
$\web$ on $U$ and the standard divergence-free web $\web_{0;\Lambda}$ of
codimension $\codim\web$. When the number $\sum_{\fol\in\Fols(\web)}
\codim\fol$ reaches $\dim M$, the situation becomes more involved, as indicated
by the following elementary theorem.

\begin{thm}
  \label{thm:dfw-logh}
  Let $\web_0=(\Rb{m},\fol_1,\ldots,\fol_n)$ be a fixed standard $n$-web
  satisfying $\sum_{i=1}^n \codim\fol_i=m$, and let $\Lambda = dx_1\wedge
  \cdots\wedge dx_m\in\Omega^m(\Rb{m})$. Given a volume form
  $\Omega=h(x)\,dx_1\wedge \cdots\wedge dx_m$ for some smooth function-germ
  $h$ at $0$, the divergence-free web-germs at $0$ given by
  $\web_{0;\Omega}=(\Rb{m},\fol_1,\ldots,\fol_n,\Omega)$ and
  $\web_{0;\Lambda}=(\Rb{m},\fol_1,\ldots,\fol_n,\Lambda)$ are equivalent if
  and only if
  \begin{equation}
    \label{eq:dfw-logh}
      \frac{\partial \log h}{\partial x_k\partial x_l} = 0
  \end{equation}
  for each pair of indices $1\leq i<j\leq n$ and for each
  $k=\csep_i+1,\ldots,\csep_{i+1}$, $l=\csep_j+1,\ldots,\csep_{j+1}$, where
  $\csep_s = \sum_{k=1}^{s-1} \codim\fol_k$ for each $s=1,\ldots,n+1$.
\end{thm}
\begin{proof}
  Let $\maps{\varphi}{(\Rb{m},0)}{(\Rb{m},0)}$ be a diffeomorphism-germ
  witnessing the local equivalence between $\web_{0;\Omega}$ and
  $\web_{0;\Lambda}$. By permuting the variables of the codomain chart we can
  assume that
  $\varphi(x_1,\ldots,x_m)
      = \left(\varphi_1(\mathbf{x}_1), \ldots,
      \varphi_n(\mathbf{x}_n)\right)$,
  where $\mathbf{x}_i = (x_{\csep_i\!+1},\ldots,x_{\csep_{i+1}})$ for $\csep_i
  = \sum_{k=1}^{i-1}\codim\fol_i$. Now, by expanding the equality
  $\varphi^*\Lambda = \Omega$ we arrive at an equivalent condition
  \begin{equation}
    \label{eq:dfw-dens-prod}
    h(x_1,\ldots,x_m) =
      \det\frac{d\varphi_1}{d\mathbf{x}_1}(\mathbf{x}_1)
      \,\cdots
      \,\det\frac{d\varphi_n}{d\mathbf{x}_n}(\mathbf{x}_n),
  \end{equation}
  where $\frac{d\varphi_k}{d\mathbf{x}_k}$ denotes the Jacobian matrix of
  $\varphi_k$ with respect to the variables $\mathbf{x}_k$. From this the
  necessity of condition $(\ref{eq:dfw-logh})$ follows immediately.

  For sufficiency, we note that condition (\ref{eq:dfw-logh}) is equivalent to
  the existence of some smooth functions $g_k\in C^\infty(\Rb{c_k},0)$ with
  $c_k=\codim\fol_k$ such that $h(x_1,\ldots,x_m) = g_1(\mathbf{x}_1) \cdots
  g_n(\mathbf{x}_n)$. In this case, the diffeomorphism-germ
  $\varphi(x_1,\ldots,x_m)
      = \left(\varphi_1(\mathbf{x}_1), \ldots, \varphi_n(\mathbf{x}_n)\right)$
  given by
  $\varphi_i(\mathbf{x}_i)
      = (x_{\csep_i+1},\ldots,x_{\csep_{i+1}-1},G_i(\mathbf{x}_i))$,
  where
  \begin{equation}
    G_i(\mathbf{x_i})
      = \int_0^{x_{\csep_{i+1}}}
        g_i(x_{\csep_i+1},\ldots,x_{\csep_{i+1}-1},t)\,dt,
  \end{equation}
  is a valid local equivalence between $\web_{0;\Omega}$ and
  $\web_{0;\Lambda}$.
\end{proof}

The main goal of this paper is to pinpoint and examine several kinds of local
invariants of general divergence-free webs $\web_\Omega$, the existence of
which is hinted at by this observation. One of these invariants can in fact be
defined by naively collecting the quantities from $(\ref{eq:dfw-logh})$ into
a~single geometric object.

\begin{defn}
  \label{def:dfw-k}
  The symmetric covariant $2$-tensor field $\kcurv(\web_\Omega)\in
  \Gamma(S^2(TM))$ defined in any local $\web_\Omega$-adapted coordinates
  $(x_1,\ldots,x_m)$ by
  \begin{equation}
    \kcurv(\web_\Omega) = \sum_{i\neq j}
      \sum_{k=\csep_i\!+1}^{\csep_{i+1}} \sum_{l=\csep_j\!+1}^{\csep_{j+1}}
      \frac{\partial \log h}{\partial x_k\partial x_l}(x)
      \ dx_k dx_l,
  \end{equation}
  where $\Omega=h(x)\,dx_1 \wedge dx_2\wedge\cdots\wedge dx_m$ and $\csep_i =
  \sum_{k=1}^{i-1} \codim\fol_k$ for $i=1,\ldots,n+1$, is called the
  \emph{nonuniformity tensor} of the divergence-free $n$-web
  $\web_\Omega=(M,\fol_1,\ldots,\fol_n,\Omega)$.
\end{defn}

It might be surprising that this definition actually works: a direct
calculation proves that this tensor field is well defined. Moreover, it can be
proven in a~similar way that if $\maps{\varphi}{\web_\Omega}{\web_\eta}$ is an
equivalence of divergence-free webs, then $\varphi^*\kcurv(\web_\eta) =
\kcurv(\web_\Omega)$.

This diffeo-geometric object holds minimal amount of information which
ensures the equivalence between its vanishing and the local triviality of the
web $\web_\Omega$ by means of Theorem \ref{thm:dfw-logh}. While its
construction may seem ad hoc, we will show in Theorems
\ref{thm:dfw-ricci} and \ref{thm:dfw-ricci-general} that it can be derived from
another, more natural invariant, which was introduced by S. Tabachnikov in
\cite{2-webs} in the special case of divergence-free $2$-webs in order to
determine the realizable local normal forms of bi-Legendrian manifolds.

The invariant in question is the curvature of a certain natural connection
associated with $\web_\Omega$ defined on the determinant bundle $\det T\fol$
for a fixed $\fol\in\Fols(\web_\Omega)$. It was proved in \cite{2-webs} that
its vanishing characterizes trivial divergence-free $2$-webs; it also played a
role in the solution to the problem of finding a normal form of $\Omega$ in
$\web_\Omega$-adapted coordinates provided in the same paper. One of our goals
is to define the natural connection for divergence-free $n$-webs with $n>2$ and
use its curvature to extend the results of \cite{2-webs} to this more general
setting.

We will express the connection and its curvature using Cartan's method of
moving frames \cite{chernchen}. With its help, the connection itself can be
succinctly characterized as the unique linear connection $\omega$ on
the Whitney sum of determinant bundles $\det T\fol^c$ for certain foliations
$\fol^c$ complementary to $\fol\in\Fols(\web_\Omega)$ induced from any
principal $G$-connection $\Theta$ (referred to as $\web_\Omega$-connection in
the sequel) equivariant with respect to the Lie group $G$ of differentials of
$\web_{0}$-equivalence germs, where $\web_{0}$ is the trivial divergence-free
web. This representation of $\omega$ yields a clear relationship between the
curvature invariants of a $H$-structure for a given subgroup $H\leq G$ (e.g. a
bi-Lagrangian manifold) and those of its induced divergence-free web, since a
$H$-connection associated with a reduction of the structure group from $G$ to
$H$ is itself a $G$-connection \cite{foundg1}. For example, we will show that
the \emph{nonuniformity tensor} defined above is a certain invariant part of
the Ricci tensor of a $G$-connection (Theorem \ref{thm:dfw-ricci-general});
since the natural connection of a bi-Lagrangian manifold is in particular
a~$G$-connection, we obtain the result of I. Vaisman \cite{vaisman}
characterizing Ricci-free bi-Lagrangian manifolds in terms of the induced
volume form, which was also noted in \cite{2-webs}.

Also, of note is a side remark found in \cite[Fig. 2]{2-webs}, where an
interesting geometric interpretation of the curvature form associated to a
natural connection of a divergence-free $2$-web in the planar case was
mentioned in passing. It was pointed out that its only nonzero coefficient
measures the degree to which the equality $ac=bd$ between the volumes of
adjacent quadrilateral regions $A,B,C,D$ enclosed by the leaves of the web
fails (Figure \ref{fig:dfw-taba}).
Since the notions of divergence-free $2$-web and Lagrangian $2$-web coincide,
a question arises whether a similar interpretation can be found for webs (both
divergence-free and Lagrangian) of higher (co-)dimensions. In the case of
divergence-free webs, we provide a positive answer by introducing a local
invariant inspired by the classical works of Thomsen, Blaschke and Bol on the
\emph{holonomy} of planar webs \cite{gewebe, pereira}.

\begin{figure}
  \centering
  \def\svgwidth{\textwidth}
  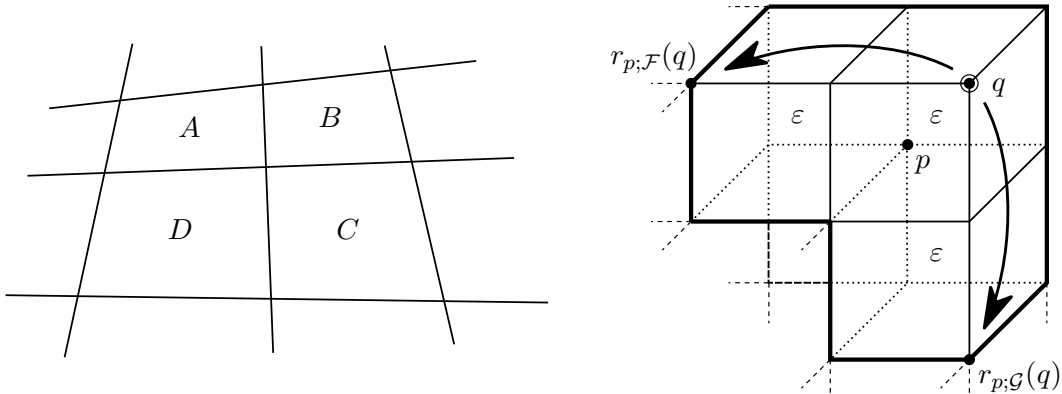
  \caption{\footnotesize Left: Tabachnikov's interpretation of divergence-free
  $2$-web's curvature in the planar case \cite{2-webs}. Right: the action of
  map-germs $r_{p;\fol}, r_{p;\gol}$. The regions bounded by the leaves of
  $\web_\Omega$ have $\Omega$-volumes equal to $\varepsilon$.}
  \label{fig:dfw-taba}
\end{figure}

At its core are certain smooth map-germs $r_{p;\fol}$ associated with each pair
$(\web_\Omega,\fol)$ consisting of a codimension-$1$ divergence-free web-germ
$\web_\Omega$ at $p\in M$ and one of its foliations $\fol\in
\Fols(\web_\Omega)$. The effect of applying $r_{p;\fol}$ to a point $q\in M$
lying on one side of the hypersurface $F_p\in\fol$ crossing $p$ is to transport
it to $q'=r_{p;\fol}(q)$ lying on the other side of $F_p$ along the curve
formed by intersecting the leaves of each $\gol\in\Fols(\web_\Omega)
\setminus\set{\fol}$ crossing $q$ in such a way, that the regions $R_{p,q}$ and
$R_{p,q'}$ enclosed by the leaves of $\web_\Omega$ crossing one of the points
$p,q$ and $p,q'$ respectively have equal $\Omega$-volumes (Figure
\ref{fig:dfw-taba}). If we choose two different foliations
$\fol,\gol\in\Fols(\web_\Omega)$, then the composition
$\ell_{p;\fol,\gol}=r_{p;\gol}\circ r_{p;\fol}\circ r_{p;\gol}\circ r_{p;\fol}$
transports each point through the concatenation of the corresponding curves.
This piecewise-smooth curve is closed whenever $\web_\Omega$ is trivial. We
will prove that nontriviality of $\ell_{p;\fol,\gol}$ is measured by the
coefficient of the nonuniformity tensor $\kcurv(\web_\Omega)$ corresponding to
the directions complementary to $T\fol,T\gol$ (Theorem
\ref{thm:dfw-geom-quantitative}).

Aggregating these maps into a single mathematical entity defines the
aforementioned local geometric invariant, called the \emph{volume-preserving
reflection holonomy} of $\web_\Omega$. In Theorem \ref{thm:dfw-geom-hicodim},
we relate the vanishing of this invariant to a generalization of S.
Tabachnikov's interpretation of the curvature of the natural connection in
terms of volumes of adjacent quadrilaterals as stated above.

Lastly, we note that assessing triviality of the divergence $2$-web induced by
the space-like foliation by hyperplanes of simultaneity in Lorentzian geometry
may prove beneficial to numerical approaches to relativistic fluid dynamics
\cite{andersson1, andersson2}. By working, if possible, in a coordinate system
which not only has privileged space and time directions, as in the classical
3+1 formalism \cite{adm}, but also makes the metric volume density constant, we
can simplify the calculation involving divergences of the fluid velocity field
occuring in e.g. the continuity equation, where the volume density plays an
important role \cite{3p1}. We will elaborate on this remark in the final
part of this paper.

The structure of the paper is as follows. In section 2, we introduce the
connections associated with a given divergnece-free $n$-web, derive their
representation in coordinates, provide some examples and relate the curvature
of the natural connection of a divergence-free $n$-web to the Ricci tensor of
the corresponding principal connection. In section 3, we show that the
nonuniformity tensor determines a divergence-free $n$-web uniquely given fixed
initial data, and use the results surrounding this theorem to give a canonical
forms and moduli space of the webs' volume forms in the planar case. In section
4 we define the geometric invariants mentioned above and establish several
geometric conditions for triviality of divergence-free $n$-webs. In the last
section we give potential applications in numerical relativity.

From now on we shall assume that divergence-free webs
$\web_\Omega=(M,\fol_1,\ldots,\fol_n,\Omega)$ satisfy
$\sum_{\fol\in\Fols(\web_\Omega)}\codim\fol = \dim M$.

\section{Connections associated with divergence-free webs}

First, we construct a set of affine connections compatible with the structure
of a given divergence-free $n$-web
$\web_\Omega=(M,\fol_1,\ldots,\fol_n,\Omega)$. While in general there are many
such connections, they share a number of properties that allow us to extract
well-defined differential invariants that do not depend on a choice of the
connection. We proceed via Cartan's method of moving frames (see e.g.
\cite{chernchen}).

\subsection{Web-adapted coframes and principal connections}

Our preferred coframes $(\xi^1, \xi^2, \ldots,
\xi^m)$ are those which satisfy
\begin{enumerate}[label=$(\alph{enumi})$,ref=\alph{enumi}]
  \item\label{it:dfw-coframe-fol}
    $T\fol_i = \bigcap_{j=\csep_i\!+1}^{\csep_{i+1}} \ker \xi^j$ for
    $i=1,2,\ldots,n$, where $\csep_i = \sum_{j=1}^{i-1}\codim\fol_j$,
  \item\label{it:dfw-coframe-vol}
    $\Omega = \xi^1\wedge\xi^2\wedge\cdots\wedge\xi^m$.
\end{enumerate}
Locally, each choice of such $1$-forms for which the ideals $I_i =
\langle \xi^{\csep_i\!+1},\ldots,\xi^{\csep_{i+1}}\rangle$ with $\csep_i =
\sum_{j=1}^{i-1} \codim\fol_j$ are all closed under the exterior derivative
fully defines a divergence-free $n$-web structure by taking
$(\ref{it:dfw-coframe-fol})$ and $(\ref{it:dfw-coframe-vol})$ as
definitions of $\fol_i$ and $\Omega$. The integrability of $T\fol_i$ follows
from Frobenius's theorem. Such coframes will be called
\emph{$\web_\Omega$-adapted}.

Each choice of a $\web_\Omega$-adapted coframe partitions the index set
$[m]=\smset{1,2,\ldots,m}$ into $n$ subsets
\begin{equation}
  \label{eq:dfw-part}
  \pi_i
    = \smset{k\in[m] : T\fol_i\subseteq\ker\xi^k}
    = \smset{\csep_i\!+1,\ldots,\csep_{i+1}},\qquad
  \csep_i=\sum_{j=1}^{i-1}\codim\fol_j
\end{equation}
for $i=1,2,\ldots,n$. This partition $\pi$ is a recurring theme of nearly
every coordinate computation throughout this paper. For notational convenience,
let us introduce the symbol $\sim$ for the equivalence relation induced by
$\pi$. More explicitly:
\begin{equation}
  \label{eq:dfw-sim}
  i\sim j \iff i,j\in\pi_k\ \text{for some }k=1,2,\ldots,n,
\end{equation}
and $i\not\sim j$ otherwise.

\begin{prop}
  \label{thm:dfw-affine}
  Let $M$ be a $m$-dimensional smooth manifold and let $\web_\Omega = (M,
  \fol_1, \ldots, \fol_n, \Omega)$ be a divergence-free $n$-web. There exists a
  torsionless affine connection $\Theta$ whose connection $1$-forms
  $\theta_{ij}$ with respect to each $\web_\Omega$-adapted coframe satisfy
  \begin{enumerate}[label=$(\arabic{enumi})$, ref=\arabic{enumi}]
    \item \label{thm:dfw-affine:fols}
      $\forall_{j\not\sim k}\ \theta^j_k=0$,\hfill
        ($\fol_i$ are $\Theta$-parallel)
    \item \label{thm:dfw-affine:vol}
      $\sum_{k=1}^m \theta^k_k = 0$. \hfill
        (the volume form $\Omega$ is $\Theta$-parallel)
  \end{enumerate}
\end{prop}
\begin{proof}
  Fix a $\web_\Omega$-adapted coframe $(\xi^1, \xi^2, \ldots, \xi^m)$
  on a neighbourhood $U\subseteq M$ of a given point $p\in M$. Our goal is to
  find a $m\times m$ matrix of $1$-forms $\Theta=[\theta^j_k]$ which for each 
  $j=1,2,\ldots,m$ satisfy
  the structure equation
  \begin{equation}
    \label{eq:dfw-affine:structure}
    d\xi^j + \smsum{k}\theta^j_k\wedge\xi^k = 0
  \end{equation}
  of a torsionless connection and have properties $(\ref{thm:dfw-affine:fols})$
  and $(\ref{thm:dfw-affine:vol})$ given in the statement. The construction of
  $\Theta$ is as follows. Since $T\fol_i$ are involutive, one can find a
  collection of $1$-forms $\alpha^j_k\in\Omega^1(U)$ with $\alpha^j_k=0$ for
  $j\not\sim k$ such that $d\xi^j = \smsum{k}\alpha^j_k\wedge\xi^k$ for each
  $j=1,2,\ldots,m$. Let $\beta = \smsum{k}\alpha^k_k$. Expand this $1$-form
  into components with respect to the coframe: $\beta = \smsum{k}f_k\xi^k$ for
  some $f_k\in C^\infty(M)$. For each $j,k=1,2,\ldots,m$ put
  \begin{equation}
    \theta^j_k = -\alpha^j_k + \delta_{jk}f_k\xi^k
  \end{equation}
  where $\delta_{jk}$ denotes the Kronecker's delta.

  This defines the desired connection locally. To define it globally,
  use a partition of unity $\set{g_U}_{U\in\mathcal{U}}$ corresponding to the
  covering of $M$ by the coframe-domains $U\in\mathcal{U}$ to define $\Theta$
  on each $U\in\mathcal{U}$ with respect to the $\web_\Omega$-adapted coframe
  $((\xi_U)^1,\ldots,(\xi_U)^m)$ as
  \begin{equation}
    \Theta = \sum_{V\in\mathcal{U} : U\cap V\neq\varnothing}
      g_V(x)(Q_{UV}\Theta_VQ_{UV}^{-1} - (dQ_{UV})Q_{UV}^{-1}),
  \end{equation}
  where $\Theta_U$ are the matrices of connection $1$-forms on $U$ obtained
  above, while $Q_{UV}$ are the transition matrices satisfying
  $(\xi_U)^k=\smsum{l}(Q_{UV})^k_l(\xi_V)^l$ on $U\cap V$ \cite{foundg1}. Note
  that $Q_{UV}$ is a product of two matrix-valued functions $RP$, where
  $R^k_l=0$ for $k\not\sim l$ and $P$ is a constant permutation matrix
  corresponding to $\sigma\in S_m$ for which $\sigma(k)\sim\sigma(l)$ if and
  only if $k\sim l$. Moreover, $\det Q_{UV}=1$, which leads to
  $\operatorname{tr} ((dQ_{UV})Q^{-1}_{UV}) = 0$. Using these facts, it can be
  checked directly by means of the structure equation that this indeed yields a
  torsionless connection satisfying $(\ref{thm:dfw-affine:fols})$ and
  $(\ref{thm:dfw-affine:vol})$.
\end{proof}

\begin{defn}
  \label{def:dfw-affine}
  An affine torsionless connection $\Theta$ satisfying
  $(\ref{thm:dfw-affine:fols})$ and $(\ref{thm:dfw-affine:vol})$ of Proposition
  \ref{thm:dfw-affine} will be called a \emph{$\web_\Omega$-connection}.
\end{defn}

\begin{prop}
  \label{thm:dfw-affine-natural}
  Let $\varphi$ be a local equivalence of divergence-free $n$-webs
  between $\web_\Omega$ and $\web_{\tilde\Omega}$. If $\Theta$ is a
  $\web_{\tilde\Omega}$-connection, then $\varphi^*\Theta$ is a
  $\web_\Omega$-connection.
\end{prop}
\begin{proof}
  Choose a coframe $(\xi^1,\ldots,\xi^m)$ on the codomain of $\varphi$
  and express $\Theta$ as $[\theta^j_k]$ in terms of this coframe. Then the
  connection matrix of $\varphi^*\Theta$ in the coframe $(\varphi^*\xi^1,
  \varphi^*\xi^2, \ldots, \varphi^*\xi^m)$ is exactly
  $[\varphi^*\theta^j_k]$ by the structure equation
  $(\ref{eq:dfw-affine:structure})$. The coframe
  $(\varphi^*\xi^k)_{k=1,\ldots,m}$ is $\web_\Omega$-adapted,
  since $\varphi^*\tilde\Omega=\Omega$ and for each $\fol\in\Fols(\web_\Omega)$
  the equality $d\varphi(T\fol)=T\gol$ holds for some
  $\gol\in\web_{\tilde\Omega}$. Moreover, the matrix $[\varphi^*\theta^j_k]$
  satisfies $(\ref{thm:dfw-affine:fols})$ and $(\ref{thm:dfw-affine:vol})$ of
  Proposition \ref{thm:dfw-affine} by linearity of pullback, hence it
  represents a $\web_\Omega$-connection.
\end{proof}

\begin{lem}
  \label{thm:dfw-affine-diff}
  Let $\web_\Omega=(M,\fol_1,\ldots,\fol_n,\Omega)$ be a divergence-free
  $n$-web. Given a fixed $\web_\Omega$-connection $\hat\Theta$, the mapping
  $\Theta\mapsto \Theta-\hat\Theta$ defines a bijection between the space of
  $\web_\Omega$-connections and the space of symmetric $TM$-valued covariant
  $2$-tensor fields $D$ satisfying
  \begin{enumerate}[label=$(\arabic{enumi})$, ref=\arabic{enumi}]
    \item \label{thm:dfw-affine-diff:closed}
      $\forall_{i}\,\forall_{v,w\in T\fol_i}
      \ D(v,w) \in T\fol_i$,
    \item \label{thm:dfw-affine-diff:diag}
      $\forall_{i}\,\forall_{v\in T\fol_i}
        \,\forall_{w\in \cap_{k\neq i}T\fol_k}
        \ D(v,w) = 0$,
    \item \label{thm:dfw-affine-diff:trace}
      $\forall_{v\in TM}\ \operatorname{tr} D(v,\cdot) = 0$.
  \end{enumerate}
\end{lem}
\begin{proof}
  Pick a $\web_\Omega$-adapted coframe $(\xi^1,\xi^2,\ldots,\xi^m)$. Let $(e_1,
  \ldots, e_m)$ be the dual frame. The quantity $D=\smsum{i,j,k}D_{ij}^k
  e_k\otimes \xi^i\otimes\xi^j$ can be thought of as a matrix of $1$-forms
  $[\Delta^k_j]$, where $\Delta^k_j = \smsum{i}D_{ij}^k\xi^i$. The matrix
  $\Theta=\hat\Theta + [\Delta^k_j]$ is a matrix of a torsionless connection if
  and only if $D$ is a symmetric $TM$-valued $2$-tensor. Necessity follows from
  Cartan's lemma on division of $2$-forms \cite[Theorem 3.4]{chernchen} applied
  to the difference of structure equations $(\ref{eq:dfw-affine:structure})$ of
  $\Theta,\hat\Theta$, while sufficiency is a consequence of the way both
  entities transform under a change of frame. Recall the relation $j\sim k$
  defined in $(\ref{eq:dfw-sim})$. Properties
  $(\ref{thm:dfw-affine-diff:closed})$-$(\ref{thm:dfw-affine-diff:trace})$ of
  $D$ can be expressed in terms of the coframe as:
  $(\ref{thm:dfw-affine-diff:closed})$ $D_{ij}^k=0$ for $i\not\sim k$ and
  $j\not\sim k$; $(\ref{thm:dfw-affine-diff:diag})$ $D_{ij}^k=0$ for $i\not\sim
  j$; $(\ref{thm:dfw-affine-diff:trace})$ $\smsum{j}D_{ij}^j=0$; or in terms of
  $\Delta^k_j$ as:
  $(\ref{thm:dfw-affine-diff:closed})$-$(\ref{thm:dfw-affine-diff:diag})$
  $\Delta^k_j=0$ for $j\not\sim k$ (assuming the symmetry of $D$);
  $(\ref{thm:dfw-affine-diff:trace})$ $\smsum{k}\Delta^k_k = 0$. A quick
  comparison with the statement of Proposition \ref{thm:dfw-affine} makes it
  clear that $D$ has these properties if and only if $\Theta$ is a
  $\web_\Omega$-connection.
\end{proof}

In the course of the proof above we have established that for each fixed
$k=1,2,\ldots,m$ the matrices $D_{ij}^k$ formed by the coefficients of a
difference tensor in a fixed $\web_\Omega$-adapted coframe are exactly the
symmetric matrices which are zero outside of a square diagonal block spanned by
entries with $i,j=\csep_k\!+1,\ldots,\csep_{k+1}$, where $\csep_k =
\sum_{l=1}^{k-1}\codim\fol_l$. The dimension of the space of these symmetric
matrices for fixed $k$ is equal to $\frac{1}{2}c_k(c_k+1)$, where $c_k =
\codim\fol_k$. These matrices are further bound by equations $\smsum{j}
D_{kj}^j = 0$ for $k=1,\ldots,m$, which allow us to express each $D_{kk}^k$ in
terms of the their off-diagonal entries. This makes the differentials
$(dD_{ij}^k)$ for $i\leq j$ and $i\sim j \sim k$ without $i=j=k$ a valid
coframe of the bundle of difference tensors $\mathfrak{D}$. Hence, if we put
$c_i=\codim\fol_i$, the dimension
of $\mathfrak{D}$ is
\begin{equation}
  \dim\mathfrak{D} = \sum_{i=1}^n c_i\big( \smfrac{1}{2}c_i(c_i+1) - 1 \big)
    = \sum_{i=1}^n \smfrac{1}{2}c_i(c_i-1)(c_i+2).
\end{equation}
In particular, if all $c_i$ are equal to $1$, the bundle $\mathfrak{D}$ has
dimension $0$.

\begin{cor}
  Let $\web_\Omega = (M,\fol_1,\ldots,\fol_n,\Omega)$ be a codimension-$1$
  divergence-free web. A $\web_\Omega$-connection exists and is unique. \qed
\end{cor}

\subsection{The natural connection of a divergence-free web}

When the codimension of any of the foliations $\fol\in\Fols(\web_\Omega)$ of a
divergence-free $n$-web $\web_\Omega$ exceeds $1$, the uniqueness claim
regarding the $\web_\Omega$-connections clearly breaks down. Note, however,
that there are some invariant quantities independent of the choice of the
$\web_\Omega$-connection.

Define $\pi_i$ as in $(\ref{eq:dfw-part})$ and $j\sim
k$ as in $(\ref{eq:dfw-sim})$. Given a $\web_\Omega$-adapted coframe
$(\xi^1,\ldots,\xi^m)$ with the dual frame $(e_1,\ldots,e_m)$, for each
$\web_\Omega$-connection $\Theta = [\theta^j_k]$ and each fixed
$i=1,2,\ldots,n$ the sums $\sum_{k\in\pi_i} \theta^k_k$ remain the same
irrespective of the choice of $\Theta$, since $\sum_{k\in\pi_i} D_{jk}^k = 0$
for each difference tensor $D$ of Lemma \ref{thm:dfw-affine-diff} by properties
$(\ref{thm:dfw-affine-diff:diag})$ and $(\ref{thm:dfw-affine-diff:trace})$.

The partial traces $\sum_{k\in\pi_i} \theta^k_k$ can be interpreted in terms of
the covariant derivative $\nabla$ of $\Theta$, which acts on vector fields
$X=\sum_k X^ke_k$ by $\nabla X = \smsum{k}(dX^k + \smsum{j} X^j\theta^k_j)
\otimes e_k$. The corresponding action on multivectors of the form
$\mathbf{e_i}=e_{\csep_i\!+1}\wedge\cdots\wedge e_{\csep_{i+1}}$ with $\csep_i
= \sum_{k=1}^{i-1}\codim\fol_k$ for $i=1,2,\ldots,n$ is
\begin{equation}
  \label{eq:dfw-conn-det}
  \begin{aligned}
    &\nabla (e_{\csep_i\!+1}\wedge\cdots\wedge e_{\csep_{i+1}})
      = \sum_{k\in\pi_i}\sum_{j=1}^m
        (-1)^{k-1}\theta^j_k\otimes
          e_j\wedge e_{\csep_i\!+1}\wedge\cdots\wedge
          \widehat{e_k} \wedge\cdots\wedge e_{\csep_{i+1}} \\
      & \hskip 5em = \sum_{k\in\pi_i}\sum_{j\not\sim k}
        (-1)^{k-1}\theta^j_k\otimes
          e_j\wedge e_{\csep_i\!+1}\wedge\cdots\wedge
          \widehat{e_k} \wedge\cdots\wedge e_{\csep_{i+1}} \\
      &\hskip 6em + \big(\sum_{k\in\pi_i} \theta^k_k\big)\otimes
        e_{\csep_i\!+1}\wedge\cdots\wedge e_k\wedge\cdots\wedge
        e_{\csep_{i+1}}, \quad \pi_i = \smset{\csep_i\!+1,\ldots,\csep_{i+1}},
  \end{aligned}
\end{equation}
where a hat denotes omission. Note that the first summand vanishes as a
consequence of property $(\ref{thm:dfw-affine:fols})$ of Proposition
$\ref{thm:dfw-affine}$, yielding $\nabla\mathbf{e}_i =
(\sum_{k\in\pi_i}\theta_k^k) \otimes\mathbf{e}_i$ for $i=1,2,\ldots,n$.

The above two paragraphs allow us to conclude that the Whitney sum of line
bundles generated by the $\mathbf{e}_i$ carries a \emph{unique} linear
connection arising naturally from the structure of the web via associated
connections. We will now define these objects invariantly in terms of
components of the web $\web_\Omega=(M, \fol_1,\ldots,\fol_n,\Omega)$.

First, let $\fol_i^c$ denote the \emph{foliation complementary to $\fol_i$ with
respect to $\web_\Omega$}, which we define as the result of integrating the
(involutive) tangent distribution $\bigcap_{j\neq i}T\fol_j$. It is worthwhile
to note that, since $\bigcap_{i=1}^n T\fol_i = \{0\}$, the tangent bundle
decomposes both into $TM = \bigoplus_{i=1}^n T\fol_i^c$ and $TM = T\fol_i
\oplus T\fol_i^c$ for each $i=1,\ldots,n$. The values of the corresponding
projections at the tangent vector $X\in TM$, which we denote by
$X_{\fol_i^c}\in T\fol_i^c$ and $X_{\fol_i}\in T\fol_i$, are the unique vectors
which satisfy $X=\sum_{k=1}^n X_{\fol^c_k} = X_{\fol_i} + X_{\fol_i^c}$ for
$i=1,2,\ldots,n$.

Now, the line bundles mentioned above are exactly the \emph{determinant
bundles} $\det T\fol_i^c$, i.e. the bundles of top-degree multivectors in
$T\fol_i^c$. The preceding discussion can be summarized in the following
proposition.

\begin{propdef}
  \label{thm:dfw-conn}
  The action of a $\web_\Omega$-adapted connection $\Theta$ on $\mathcal{X}(M)$
  induces a linear connection $\omega$ on the bundle
  $\bigoplus_{i=1}^n \det T\fol_i^c$ which is independent of the choice of
  $\Theta$. This connection will be called the
  \emph{natural $\web_\Omega$-connection}. \qed
\end{propdef}

\begin{rem}
  Assume that $\web_\Omega$ is a codimension-$1$ divergence-free $n$-web. The
  determinant bundles $\det T\fol_i^c$ reduce to $T\fol_i^c$, allowing
  us to identify the (unique) $\web_\Omega$-connection $\Theta$ with the
  natural connection $\omega$ in this case.
\end{rem}

By uniqueness of natural connections and Proposition
\ref{thm:dfw-affine-natural}, the pullback of the natural connection of
$\web_{\tilde\Omega}$ via a local equivalence of divergence-free $n$-webs
between $\web_{\Omega}$ and $\web_{\tilde\Omega}$ coincides with the natural
connection of $\web_\Omega$. To further justify the use of the adjective
\emph{natural} in the above definition, we will show that we can characterize
$\web_\Omega$-connections in terms of the natural connection.

\begin{prop}
  \label{thm:dfw-web-vs-natural}
  Let $(M,\fol_1,\ldots,\fol_n,\Omega)$ be a divergence-free $n$-web. An
  affine torsionless connection $\Theta$ on $M$ is a $\web_\Omega$-connection
  if and only its action on $\bigoplus_{i=1}^n \det T\fol_i^c$ is well-defined
  and coincident with the action of the natural $\web_\Omega$-connection
  $\omega$.
\end{prop}
\begin{proof}
  If $\Theta$ is a $\web_\Omega$-connection, then the action of $\Theta$ on
  $\bigoplus_{i=1}^n \det T\fol_i^c$ is equal to that of $\omega$ by the
  definition of $\omega$. For the converse, suppose that $(\xi^1,\xi^2,
  \ldots,\xi^m)$ is a $\web_\Omega$-adapted coframe with dual frame
  $(e_1,e_2,\ldots,e_m)$. The connection $\Theta=[\theta^j_k]$ gives rise to a
  covariant derivative operator $\nabla$ which acts on multivectors
  $\mathbf{e}_i = e_{\csep_i\!+1} \wedge\cdots\wedge
  e_{\csep_{i+1}}\in\Gamma(\det T\fol_i)$ for $i=1,2,\ldots,n$ by means of
  $(\ref{eq:dfw-conn-det})$, where $\csep_i = \sum_{k=1}^{i-1} \codim\fol_k$.
  This action coincides with the action of the natural connection $\omega$ by
  our assumption, which implies that the right-hand side must be equal to
  $(\sum_{k=\csep_i\!+1}^{\csep_{i+1}} \tilde\theta_k^k)\otimes\mathbf{e}_i$,
  for some $\web_\Omega$-connection $\tilde\Theta=[\tilde\theta^j_k]$. From
  this we deduce: $(1)$~$\theta^j_k = 0$ for $j\not\sim k$, where we have used
  notation from $(\ref{eq:dfw-sim})$, and $(2)$~$\sum_{k=1}^m \theta^k_k =
  \sum_{k=1}^m \tilde\theta^k_k = 0$. Since $(\xi^1,\ldots,\xi^m)$ is
  $\web_\Omega$-adapted, these two conditions correspond to
  $(\ref{thm:dfw-affine:fols})$ and $(\ref{thm:dfw-affine:vol})$ of Proposition
  \ref{thm:dfw-affine}, which together with the lack of torsion make $\Theta$ a
  $\web_\Omega$-adapted connection.
\end{proof}

\begin{rem}
  The equivalence above cannot be inferred for all (not necessarily
  torsionless) affine connections, since there are $1$-forms $\theta^k_j=
  \smsum{i}\Gamma^k_{ij}\xi^i$ with coefficients $\Gamma_{ij}^k$ which are
  non-symmetric in $i,j$, but nevertheless satisfy $\theta^k_j=0$ for
  $j\not\sim k$ and $\smsum{k}\theta^k_k=0$.
\end{rem}

\subsection{Coordinate expressions}

Fix a $\web_\Omega$-adapted system of coordinates $(x_1,x_2,\ldots,x_m)$, the
standard frame $(\basis{x_1},\basis{x_2},\ldots,\basis{x_m})$ and its dual
coframe $(dx_1,dx_2,\ldots,dx_m)$. Express the volume form in terms of the
coframe as $\Omega = h(x)\, dx_1\wedge dx_2\wedge\cdots\wedge dx_m$. By
picking a $\web_\Omega$-adapted coframe, for example, $\xi^i = dx_i$ for
$i=1,2,\ldots,m-1$ and $\xi^m = h(x)\,dx_m$ and carrying out the computations
outlined in the proof of Proposition \ref{thm:dfw-affine}, one finds the matrix
$\Theta=[\theta^j_k]$ of a certain $\web_\Omega$-connection
\begin{equation}
  \theta^j_k = \begin{cases}
    \frac{\partial \log h}{\partial x_j}\cdot\xi^j & \text{if } j=k\neq m,\\
      - d\log h + \frac{1}{h}\frac{\partial \log h}{\partial x_m}\cdot\xi^m
      & \text{if } j=k=m,\\
    0 & \text{otherwise},
  \end{cases}
\end{equation}
which can be used to compute the action of the natural connection on the
determinant bundles $\bigoplus_{i=1}^n\det T\fol^c_i$. It is given by
$\nabla\mathbf{e}_i = (\sum_{k=\csep_i\!+1}^{\csep_{i+1}}\theta^k_k)
\otimes\mathbf{e}_i = \smsum{j}\tilde\omega^{j}_i\otimes\mathbf{e}_j$ with
$\csep_i = \sum_{k=1}^{i-1}\codim\fol_k$ as described in the remarks
surrounding equality $(\ref{eq:dfw-conn-det})$. After expressing the result in
the coordinate coframe $(d\mathbf{x}^1,d\mathbf{x}^2,\ldots,d\mathbf{x}^n)$ on
$\bigoplus_{i=1}^n\det T\fol_i^c$ by means of the gauge transformation $\omega
= Q^{-1}\tilde\omega Q + Q^{-1}dQ$ with the transition matrix $Q$ defined by
the relation $\xi^{\csep_j\!+1}\wedge\cdots \wedge\xi^{\csep_{j+1}} = \smsum{k}
Q^j_k\cdot d\mathbf{x}^k$, where $d\mathbf{x}^i =
dx_{\csep_i\!+1}\wedge\cdots\wedge dx_{\csep_{i+1}}\in\det T\fol_i^c$ for
$i=1,\ldots,n$, we get the final expression for the connection $1$-form
$\omega$ of the natural connection in the coordinate coframe. It takes the form
\begin{equation}
  \label{eq:dfw-coords-connection}
  \omega = \begin{pmatrix}
    \omega^1_{1} & 0 & \cdots & 0 \\
    0 & \omega^2_{2} & \cdots & 0 \\
    \vdots & \vdots & \ddots & \vdots \\
    0 & 0 & \cdots & \omega^n_{n}
  \end{pmatrix},
\end{equation}
where
\begin{equation}
  \omega^i_{i} = \sum_{k=\csep_i\!+1}^{\csep_{i+1}}
    \frac{\partial\log h}{\partial x_k}(x)\,dx_k.
\end{equation}
With this in hand, it is straightforward to compute the curvature $2$-form
(denoted by $\Xi$) via Cartan's formula, which yields
\begin{equation}
  \Xi = \begin{pmatrix}
    \Xi^1_1 & 0 & \cdots & 0 \\
    0 & \Xi^2_2 & \cdots & 0 \\
    \vdots & \vdots & \ddots & \vdots \\
    0 & 0 & \cdots & \Xi^n_n
  \end{pmatrix},
\end{equation}
where
\begin{equation}
  \label{eq:dfw-coords-curv-comp}
  \Xi^i_i = d\omega^i_i
    = \sum_{j\neq i}\left(\sum_{k=\csep_i\!+1}^{\csep_{i+1}}
      \sum_{l=\csep_j+1}^{\csep_{j+1}}
        \,\frac{\partial\log h}{\partial x_k\,\partial x_l}(x)
        \,dx_l\wedge dx_k\right).
\end{equation}
As a corollary we obtain the following result.

\begin{thm}
  A divergence-free $n$-web $\web_\Omega$ is locally trivial if and only if its
  natural connection $\omega$ is flat.
\end{thm}
\begin{proof}
  Note that the curvature $2$-forms $\Xi^i_i$
  $(\ref{eq:dfw-coords-curv-comp})$ vanish everywhere if and only if the
  assumptions of Theorem \ref{thm:dfw-logh} are met.
\end{proof}

\begin{exmp}
  \label{ex:dfw-polar}
  Let $M=\Rb{2}\setminus\smset{0}$, $\Omega=dx\wedge dy$ and suppose that
  $\fol, \gol$ are foliations of $M$ by open half-lines emanating from $0$ and
  by concentric circles cenetered at $0$ respectively. Then the divergence-free
  $2$-web $\web_\Omega=(M,\fol,\gol,\Omega)$ is locally trivial. A local
  web-equivalence $\varphi$ witnessing its triviality at a point $p\in M$ is
  given by a restriction of the covering map $\maps{\pi}{\Rb{1}_+
  \times\Rb{1}}{M}; (\rho,\phi)\overset{\pi}{\mapsto} (\sqrt{2\rho}\cos\phi,
  \sqrt{2\rho}\sin\phi)$ to a neighbourhood of $p$.

  The map $\pi$ carries lines into geodesics of the natural connection $\omega$
  of $\web_\Omega$ by Proposition \ref{thm:dfw-affine-natural}. In polar
  coordinates $(\rho, \phi)$ they are given by equations $ar^2+b\phi=c$ for
  some $a,b,c\in\Rb{1}$; in other words, they are \emph{Fermat spirals} (Figure
  \ref{fig:dfw-spirals}). As a parametrized geodesic $\gamma(t)$, a generic
  Fermat spiral can be characterized using $\omega$ as the unique curve with
  given initial velocity $\dot\gamma(0)$ whose angular velocity
  $\dot\gamma_\gol$ with respect to the origin is constant, and such that at
  every instant $t\in\Rb{1}$ the area $\Omega(\dot\gamma_{\fol},
  \dot\gamma_{\gol})$ of an infinitesimal rectangle spanned by the projections
  $\dot\gamma_\fol\in T\fol, \dot\gamma_\gol\in T\gol$ with
  $\dot\gamma=\dot\gamma_\fol+\dot\gamma_\gol$ remains the same.
\end{exmp}

\begin{figure}
  \centering
  \def\svgwidth{\textwidth}
  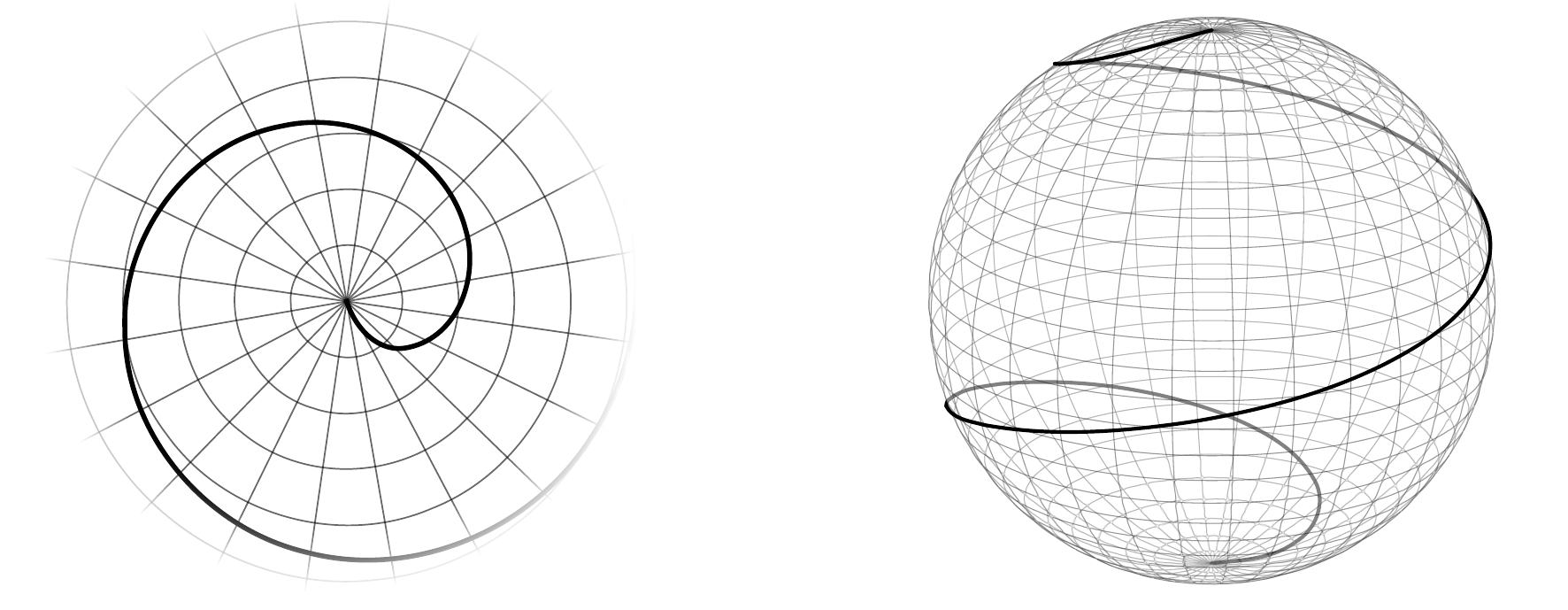
  \caption{\footnotesize Two-dimensional divergence-free webs $\web_\Omega$
  with standard volume forms $\Omega$ and geodesics of their
  $\web_\Omega$-connections. Left: a planar web formed by level sets of polar
  coordinates. The geodesics are the Fermat spirals $ar^2+b\phi=c$ where
  $r,\phi$ are the radial and angular coordinates respectively. Right: the
  web given by lines of constant latitude and longitude on $S^2$. The geodesics
  are the spherical spirals given by $az+b\phi=c$ in cylindrical coordinates
  $(z,\phi)$.}
  \label{fig:dfw-spirals}
\end{figure}

\begin{exmp}
  Let $M$ be a unit $2$-sphere $S^2\subseteq\Rb{3}$ without the north and south
  poles $(x,y,z)=(0,0,\pm 1)\in S^2$. Let $\fol$, $\gol$ be foliations of $M$
  by circles of constant latitude $z=c\in\Rb{1}$ and semicircles of constant
  longitude $ax+by=0$ for $a,b\in\mathbb{R}$, $(a,b)\neq 0$ respectively.
  Denote by $\Omega$ the Riemannian volume form $x\,dy\wedge dz + y\,dz\wedge
  dx + z\,dx\wedge dy$ induced from the ambient space. The divergence-free
  $2$-web $\web_\Omega=(M,\fol,\gol,\Omega)$ is locally trivial. A global
  coordinate system witnessing the triviality of $\web_\Omega$ is the
  cylindrical coordinate system $(\phi,z)$, where $z$ is the linear
  coordinate along the north-south axis and $\phi$ is the longitude of a given
  point. The geodesics of the natural connection $\omega$ can be computed as in
  Example \ref{ex:dfw-polar}; they are the spherical spirals satisfying $az +
  b\phi = c$ for some $a,b,c\in\Rb{1}$ (Figure \ref{fig:dfw-spirals}).

  These curves possess a property similar to the one exhibited by Fermat
  spirals. Note that a~single trasverse geodesic together with a semicircle of
  constant longitude subdivide the sphere into several quasi-rectangular
  strips. The property in question is: the strips which are not adjacent to a
  pole have equal areas. This can easily be seen in cylindrical coordinates.
\end{exmp}

\subsection{Construction in terms of Bott's connection}

That the coordinate expression for the curvature $2$-form essentially agrees
with the one derived in the article of Tabachnikov \cite{2-webs} for
divergence-free $2$-webs is not a coincidence; the natural connection $\omega$
(Definition \ref{thm:dfw-conn}) and Tabachnikov's connection $\nabla$, while
differing in the way they arise, represent the same connection.

To recall the construction of $\nabla$, let us first introduce a partial
connection acting on the quotient bundle $\nu\fol=TM/T\fol$ for some foliation
$\fol$ of $M$. Given a vector field $X\in\Gamma(T\fol)$ and a~section
$\bar{Y}\in\Gamma(\nu\fol)$ with a representative $Y\in\mathfrak{X}(M)$, define
the $D^\fol$-covariant derivative of $\bar{Y}$ along $X$ by
\begin{equation}
  \label{eq:dfw-botts-conn}
  (D^\fol)_X \bar{Y} = [X,Y] \bmod T\fol.
\end{equation}
The \emph{Bott's connection} is the corresponding \emph{partial} connection
$\maps{D^\fol}{\Gamma(T\fol)\times\Gamma(\nu\fol)}{\Gamma(\nu\fol)}$. It is
always flat. This can be verified by picking a local basis of $D^\fol$-parallel
sections in coordinates $(x_1,\ldots,x_m)$ adapted to $\fol$ in which $T\fol =
\bigcap_{i=k+1}^m \ker dx_i$, for example $(\basis{x_1}, \basis{x_2}, \ldots,
\basis{x_k})\bmod T\fol$.

Now, in the $2$-web case, the presence of a complementary foliation $\gol$
allows us to decompose the tangent bundle into $TM=T\fol\oplus T\gol$, leading
to an identification of $\nu\fol$ with $T\gol$ via a~projection onto the second
summand. Thus, $D^\fol$ acts naturally on $T\gol$ along $T\fol$, while $D^\gol$
acts on $T\fol$ along $T\gol$; both of these extend naturally to the action on
the corresponding determinant bundles $\det T\fol$, $\det T\gol$ and their
duals. A volume form $\Omega$ on $M$ defines a bundle isomorphism
$\maps{\mu}{\det T\fol}{\det T^*\gol}$ by the formula $\eta\mapsto
\eta\iprod\Omega$. The aforementioned \emph{Tabachnikov's connection} of a
divergence-free $2$-web $(M,\fol,\gol,\Omega)$ on $\det T\fol$ is the linear
connection $\maps{\nabla}{\mathfrak{X}(M) \times \Gamma(\det
T\fol)}{\Gamma(\det T\fol)}$ whose action is given by
\begin{equation}
  \label{eq:dwf-conj-conn}
  \nabla_X \eta = \mu^{-1}D^\fol_{X_\fol} (\mu\eta) + D^\gol_{X_\gol}\eta,
\end{equation}
where $X=X_\fol+X_\gol$ for $X_\fol\in\Gamma(T\fol),
X_\gol\in\Gamma(T\gol)$.

\begin{prop}
  Let $\web_\Omega=(M,\fol_1,\ldots,\fol_n,\Omega)$ be a divergence-free
  $n$-web and, for each $i=1,2,\ldots,n$, let $\fol^c_i$ be a~foliation with
  tangent distribution $T\fol^c_i = \bigcap_{j\neq i}T\fol_j$. The covariant
  derivative $\maps{D^i}{\mathfrak{X}(M) \times \Gamma(\det
  T\fol^c_i)}{\Gamma(\det T\fol^c_i)}$ associated with the restriction of its
  natural connection $\omega$ to $\det T\fol_i^c$ is exactly the Tabachnikov's
  connection $\nabla^i$ on $\det T\fol_i^c$ of the divergence-free $2$-web $(M,
  \fol_i^c, \fol_i,\Omega)$.
\end{prop}
\begin{proof}
  Let $(\xi^1,\ldots,\xi^m)$ be a $\web_\Omega$-adapted coframe on $M$
  with the dual coframe $(e_1,\ldots,e_m)$ and define the relation $j\sim k$
  as in $(\ref{eq:dfw-sim})$. For a $\web_\Omega$-connection
  $\Theta = [\theta^j_k]$ and the associated covariant derivative $D$, property
  $(\ref{thm:dfw-affine:fols})$ of Proposition~\ref{thm:dfw-affine} means that
  for each $i=1,2,\ldots,n$ and $X=\smsum{k\not\sim i}X^ke_k\in
  \Gamma(T\fol_i)$, the vector field $D_VX = \smsum{j\not\sim i}(VX^j)e_j +
  \smsum{k\not\sim i}\smsum{j} \theta^j_k(V)X^ke_j$ is in $\Gamma(T\fol_i)$.
  Since $T\fol^c_i = \bigcap_{j\neq i}T\fol_i$, the same is true with $T\fol_i$
  replaced by its complementary foliation $T\fol^c_i$. By the lack of torsion
  of $\Theta$, for $X\in\Gamma(T\fol_i)$ and $Y\in\Gamma(T\fol_i^c)$ we obtain
  \begin{equation}
    \label{eq:dfw-affine-taba:bott}
    D_XY = (D_XY)_{\fol_i^c}
      = (D_XY)_{\fol_i^c} - (D_YX)_{\fol_i^c}
      = [X,Y]_{\fol_i^c}
      = D^{\fol_i^c}_XY,
  \end{equation}
  where the mapping $v\mapsto v_{\fol_i^c}$ is the projection from
  $TM=T\fol_i\oplus T\fol_i^c$ to $T\fol_i^c$ along $T\fol_i$.
  Let $\maps{\mu_i}{\det T\fol_i^c}{\det T^*\fol_i}$ be the map $\mu_i(\eta)
  = \eta\iprod\Omega$. Property $(\ref{thm:dfw-affine:vol})$ is equivalent to
  $D\Omega=0$, which together with
  $D\Gamma(T\fol_i^c)\subseteq\Omega^1(M)\otimes \Gamma(T\fol_i^c)$ for
  $i=1,2,\ldots,n$ gives $D_Y(\eta\iprod\Omega) = (D_Y\eta)\iprod\Omega$ for
  each $Y\in\mathfrak{X}(M)$ and $\eta\in\det T\fol^c_i$, or, in other words,
  $D_Y\circ\mu_i = \mu_i\circ D_Y$ for each $Y\in\mathfrak{X}(M)$. Let
  $\gol=\fol_i^c$ and $\hol=\fol_i$. Using the fact that $D$ commutes with
  $\mu_i$ and the equality $(\ref{eq:dfw-affine-taba:bott})$ applied to a fixed
  $\eta\in\Gamma(\det T\fol_i^c)$ via the Leibniz rule we arrive at
  \begin{equation}
    \begin{aligned}
      D_X\eta &= D_{X_{\gol}}\eta + D_{X_{\hol}}\eta \\
      &= \mu_i^{-1}D_{X_{\gol}}(\mu\eta) + D_{X_{\hol}}\eta \\
      &= \mu_i^{-1}D^{\gol}_{X_{\gol}}(\mu\eta) + D^{\hol}_{X_{\hol}}\eta =
      \nabla^i_X\eta,
    \end{aligned}
  \end{equation}
  where $X=X_\gol+X_\hol$ for $X_\gol\in\Gamma(T\gol), X_\hol\in\Gamma(T\hol)$.
  Since the action of $\omega$ on $\bigoplus_{i=1}^n \det T\fol_i^c$ comes from
  the action of $D$ regardless of the choice of $\Theta$ by Proposition
  \ref{thm:dfw-web-vs-natural}, this concludes the proof.
\end{proof}

\subsection{Nonuniformity tensor}
\label{sec:dfw-nonuniformity}

The notion of $\web_\Omega$-connection of a divergence-free $n$-web
$\web_\Omega$ allows us to give an interpretation of the nonuniformity tensor
$\kcurv(\web_\Omega)$ (Definition \ref{def:dfw-k}) in terms of more familiar
differential-geometric objects.

\begin{thm}
  \label{thm:dfw-ricci}
  The Ricci tensor $\Rc$ of the $\web_\Omega$-connection $\Theta$ of a
  divergence-free $n$-web $\web_\Omega$ of codimension $1$ is equal to the
  nonuniformity tensor $\kcurv(\web_\Omega)$.
\end{thm}
\begin{proof}
  \newcommand{\plogh}[2]{\smfrac{\partial^2\log h}{\partial x_{#1}\partial%
    x_{#2}}}%
  Pick a $\web_\Omega$-adapted coordinate system $(x_1,\ldots,x_n)$.
  Coefficients of the curvature tensor $R_{ijk}^j$ in the coordinate basis
  $(\basis{x_1},\ldots,\basis{x_n})$ can be computed from the curvature
  $2$-form $\Xi$ $(\ref{eq:dfw-coords-curv-comp})$ and are equal to
  \begin{equation}
    \begin{aligned}
      R_{ijk}^l &= \Xi^k_l(\basis{x_i},\basis{x_j})
        = \delta_{kl}\,\smsum{m\neq k}\plogh{k}{m}
          \,(\delta_{mi}\delta_{kj}-\delta_{mj}\delta_{ki}) \\
        &= \delta_{kl}\,\big(\,(1-\delta_{ki})\delta_{kj}\plogh{k}{i}
          - (1-\delta_{kj})\delta_{ki}\plogh{k}{j}\,\big),
    \end{aligned}
  \end{equation}
  where $\delta_{ij}$ is the Kronecker's delta.
  To compute the Ricci tensor, we contract the second lower index with the
  upper one. As a result we get
  \begin{equation}
    \Rc_{ik} = \smsum{j} R_{ijk}^j
      = (1-\delta_{ki})\,\plogh{k}{i}.
  \end{equation}
  The expression above is symmetric in its indices, and is $0$ if the indices
  are equal. We can rewrite this tensor using the symmetric product
  to obtain
  \begin{equation}
    \Rc = \sum_{i\neq j}
      \frac{\partial^2\log h}{\partial x_j\,\partial x_i}\,dx_idx_j
        = \kcurv(\web_\Omega).\qedhere
  \end{equation}
\end{proof}

\begin{cor}
  For a codimension-$1$ divergence-free $n$-web $\web_\Omega$ with the
  $\web_\Omega$-connection $\Theta$ the following properties are equivalent.
  \begin{enumerate}[label*=$(\arabic*)$, nosep]
    \item $\web_\Omega$ is locally trivial.
    \item $\Theta$ is flat.
    \item $\Theta$ is Ricci-flat.\qed
  \end{enumerate}
\end{cor}

The fact that the results above concern only the codimension-$1$ case might be
somewhat disappointing. Nevertheless, it happens that we can use them
to recover the nonuniformity tensor $\kcurv(\web_\Omega)$ from a
$\web_\Omega$-connection $\Theta$ for a divergence-free $n$-web $\web_\Omega =
(M, \fol_1,\ldots,\fol_n,\Omega)$ of arbitrary, possibly non-constant
codimension.

The method is based on the following observation, which is more easily seen in
$\web_\Omega$-adapted coordinates: the Ricci tensor $\Rc$ contains all the data
already present in the curvature $\kcurv(\web_\Omega)$ of not only the
codimension-$1$ web $\web_\Omega$ itself, but also of all webs
$\web_{\pi;\Omega}= (M, \fol_{\pi_1}, \ldots, \fol_{\pi_n}, \Omega)$ obtained
by forming a partition $\pi=\smset{\pi_1,\ldots,\pi_n}$ of the index set $[m]$
and letting the foliations $\fol_{\pi_j}$ be generated by the integrable
distributions $\bigcap_{k\in\pi_j} T\fol_k$. Moreover, the coordinate
expressions suggest that the Ricci tensor $\Rc$ of the $\web_\Omega$-connection
can be written as a sum of $\kcurv(\web_{\pi;\Omega})$ and a certain
block-diagonal term, the form of which is invariant with respect to local
equivalences of $\web_{\pi;\Omega}$. It turns out that, given any
divergence-free web $\web_\Omega$ and its $\web_\Omega$-connection, one can use
an auxiliary codimension-$1$ web to obtain such decomposition of the
corresponding Ricci tensor $\Rc$ and the relation between $\kcurv(\web_\Omega)$
and one of the invariant components of $\Rc$. The process will be detailed
below.

With the help of Lemma \ref{thm:dfw-affine-diff} it is possible to compute the
set of Ricci tensors $\Rc$ of $\web_\Omega$-connections $\Theta$ given a single
one. In the course of the calculations one will eventually notice that certain
coefficients of $\Rc$ inside a fixed $\web_{\Omega}$-adapted coordinate chart
do not change when replacing $\Theta$ with its perturbed variant.

\begin{lem}
  \label{thm:dfw-ricci-ext}
  Let $\web_\Omega=(M,\fol_1,\ldots,\fol_n,\Omega)$ be a divergence-free
  $n$-web. In any $\web_\Omega$-adapted coordinate system $(x_1,\ldots,x_m)$
  the Ricci tensor $\Rc$ of any $\web_\Omega$-connection $\Theta$ satisfies
  \begin{equation}
    \Rc_{kl}(x) = \frac{\partial^2\log h}{\partial x_k\,\partial x_l}(x),
  \end{equation}
  for all $i\neq j$, $\csep_i<k\leq \csep_{i+1}$ and
  $\csep_j<l\leq \csep_{j+1}$, where $\csep_s=\sum_{k=1}^{s-1} \codim \fol_k$
  for $s = 1,\ldots,n+1$.
\end{lem}
\begin{proof}
  Let $\Theta$ be a fixed $\web_\Omega$-connection and let $D$ be a
  difference tensor described by Lemma~\ref{thm:dfw-affine-diff}. These two
  objects correspond to matrices of $1$-forms $\Theta=[\theta^j_k]$ and $\Delta
  = [\Delta^j_k]$ with $\Delta^j_k = \smsum{i}D^k_{ij}\,dx_i$ in the standard
  coframe $(dx_1,dx_2,\ldots,dx_m)$. If we denote the Riemann curvature tensors
  corresponding to $\Theta$ and $\Theta+\Delta$ in the standard (co)frame by
  $R$ and $\smash{\widetilde{R}}$ respectively, and use the usual Christoffel
  symbols $\Gamma_{ij}^k$ to denote the coefficients $\theta^k_j(\basis{x_i})$
  of the connection $\Theta$, the following identity, expressed using the
  Einstein summation convention for brevity, is obtained via direct calculation
  for each quadruple of indices $i,j,k,l=1,\ldots,m$.
  \begin{equation}
    \begin{aligned}
      \widetilde{R}_{ikj}^l = R_{ikj}^l
        & + \basis{x_i} D_{kj}^l - \basis{x_k}D_{ij}^l \\
        & + D^l_{im}\,\Gamma_{kj}^m - D_{km}^l\,\Gamma_{ij}^m \\
        & + \Gamma^l_{im}\,D_{kj}^m - \Gamma_{km}^l\,D_{ij}^m \\
        & + D^l_{im}\,D_{kj}^m - D_{km}^l\,D_{ij}^m.
    \end{aligned}
  \end{equation}
  A contraction of the $j$-index with the $l$-index will yield an equality
  relating the corresponding Ricci tensors $\Rc$ and $\smash{\widetilde{\Rc}}$.
  By $\operatorname{tr} \iota_v D = 0$ for each $v\in TM$, i.e. condition
  $(\ref{thm:dfw-affine-diff:trace})$ of Lemma \ref{thm:dfw-affine-diff}, and
  the symmetry of $D$ in the lower indices, the contractions of the second,
  fifth and last summands vanish, leaving only
  \begin{equation}
    \label{eq:dfw-ricci-ext-almost}
    \widetilde{\Rc}_{ij} = \Rc_{ij} - \basis{x_k}D_{ij}^k
      + D_{im}^k\,\Gamma_{kj}^m + \Gamma_{im}^k\,D_{kj}^m
      + D_{im}^k\,D_{kj}^m - \Gamma_{km}^k\,D_{ij}^m.
  \end{equation}

  To show that this expression reduces to
  $\smash{\widetilde{\Rc}_{ij}=\Rc_{ij}}$, define $j\sim k$ as in
  $(\ref{eq:dfw-sim})$ and note that $\Gamma_{ij}^k=0$ if:
  $(1)$  $j\not\sim k$ by condition $(\ref{thm:dfw-affine:fols})$ of
    Proposition $\ref{thm:dfw-affine}$;
  $(2)$  $i\not\sim k$ by the symmetry in lower indices, which follows
    from the lack of torsion;
  $(3)$  $i\not\sim j$ by the preceding two cases.
  The same holds for $D_{ij}^k$, since any difference tensor $D$
  is a difference of two such Christoffel symbols $\Gamma$ and
  $\smash{\tilde{\Gamma}} = \Gamma+D$. Thus, the only nonzero coefficients of
  both $D_{i j}^k$ and $\Gamma_{i j}^k$ are those
  with $i \sim j \sim k$. Applying this to
  $(\ref{eq:dfw-ricci-ext-almost})$ yields the desired equality
  $\smash{\widetilde{\Rc}}_{ij} = \Rc_{ij}$ for each $i\not\sim j$.

  The value of $\Rc_{ij}$ at any given point $p$ inside a $\web_\Omega$-adapted
  chart can be determined by choosing $\Theta$ to be any
  $\web_\Omega$-connection which coincides with the natural connection of a
  codimension-$1$ divergence-free $m$-web $(U,\gol_1,\ldots,\gol_m,\Omega)$
  with $T\gol_i = \ker dx_i$ on a small neighbourhood $U$ of $p$. Such
  a connection exists by a partition of unity argument similar to the one
  employed in Proposition $\ref{thm:dfw-affine}$. Inside the set $U$, the
  components $\Rc_{ij}$ with $i\not\sim j$ are given in Theorem
  \ref{thm:dfw-ricci}, which describes $\Rc$ in the codimension-$1$ case.
\end{proof}

Let us provide an invariant description of the \emph{off-diagonal part} of
$\Rc$ computed above. A~decomposition of the tangent bundle $TM =
\bigoplus_{i=1}^n T\fol_i^c$ for $T\fol_i^c=\bigcap_{j\neq i} T\fol_j$
allows us write the bundle of symmetric covariant $2$-tensors $S^2(M)$ as a
Whitney sum of two subbundles invariant with respect to local equivalences of
$n$-webs:
\begin{equation}
  \begin{aligned}
    S^2_{D}(M)
      &= \set{
        A\in S^2(M) : \forall_{i\neq j}
        \,\forall_{v\in T\fol_i^c}\,\forall_{v\in T\fol_j^c}
        \ A(v,w) = 0
      }, \\
    S^2_{O}(M)
      &= \set{
        A\in S^2(M) : \forall_{j}\,\forall_{v,w\in T\fol_j^c}
        \ A(v,w) = 0
      }.
  \end{aligned}
\end{equation}
The decomposition $S^2(M)=S^2_D(M)\oplus S^2_O(M)$ defines the corresponding
projections
\begin{equation}
  \maps{\mathrm{pr}_D}{S^2(M)}{S^2_D(M)},\qquad
  \maps{\mathrm{pr}_O}{S^2(M)}{S^2_O(M)}.
\end{equation}
Since the Ricci tensor of any $\web_\Omega$-connection $\Theta$ is symmetric
due to its lack of torsion, $\Omega$ being $\Theta$-parallel by condition
$(\ref{thm:dfw-affine:vol})$ of Definition $\ref{def:dfw-affine}$ and the
algebraic Bianchi identity $R_{ijk}^l + R_{jki}^l + R_{kij}^l = 0$
\cite[Chapter I, Proposition 3.1]{nomizusasaki}, we obtain the characterization
of the nonuniformity tensor in arbitrary codimension.
\begin{thm}
  \label{thm:dfw-ricci-general}
  Let $\web_\Omega$ be a divergence-free $n$-web, and let $\Rc$ denote the
  Ricci tensor of any of its $\web_\Omega$-connections. Then
  \begin{equation}
    \kcurv(\web_\Omega) = \mathrm{pr}_O(\Rc).
  \end{equation}
\end{thm}
\begin{proof}
  By comparison in coordinates of Definition \ref{def:dfw-k} and Lemma
  \ref{thm:dfw-ricci-ext}.
\end{proof}

\section{Invariants and the classification problem}

The leading question of this part is: does the nonuniformity tensor
$\kcurv(\web_\Omega)$ of a divergence-free $n$-web
$\web_\Omega=(M,\fol_1,\ldots,\fol_n,\Omega)$ of arbitrary codimension
(Definition \ref{def:dfw-k}) determine the local structure of the web up to
equivalence?

\begin{exmp}
  Denote by $\web_{\Omega_i}$ the germs at $0$ of standard divergence-free
  $2$-webs $(\Rb{2},\fol,\gol,\Omega_i)$ with $T\fol=\ker dx$,
  $T\gol=\ker dy$ and volume forms $\Omega_i = h_i(x,y)\, dx\wedge dy$ with
  \begin{equation}
    h_1(x,y) = e^{\frac{1}{4}x^2y^2},\quad h_2(x,y) =
    (1+x)(1+y)\,e^{\frac{1}{4}x^2y^2},
  \end{equation}
  The nonuniformity tensors of both of these webs are $\kcurv(\web_{\Omega_i})
  = xy\,dxdy$, yet the webs $\web_{\Omega_1},\web_{\Omega_2}$ are not locally
  equivalent. Their local equivalence $\maps{\varphi}{(\Rb{2},0)}{(\Rb{2},0)}$
  would assume the form $\varphi(x,y)=(\hat{x}(x),\hat{y}(y))$ (up to
  permutation of variables) and satisfy $\varphi^*\Omega_1 = \Omega_2$ or, more
  explicitly,
  \begin{equation}
    \label{eq:dfw-nonuniformity-not-enough}
    e^{\frac14 \hat{x}(x)^2\hat{y}(y)^2}\hat{x}'(x)\hat{y}'(y)
      = (1+x)(1+y)e^{\frac14 x^2y^2}.
  \end{equation}
  Setting $x=y=0$ would lead to $\hat{x}'(0)\hat{y}'(0)=1$. By letting one of
  the coordinates vary while keeping the other at $0$ one would obtain
  $\hat{x}(x) = \hat{x}'(0)(x+\smfrac{1}{2}x^2)$ and $\hat{y}(y) =
  \hat{y}'(0)(y+\smfrac{1}{2}y^2)$, contradicting
  $(\ref{eq:dfw-nonuniformity-not-enough})$.
\end{exmp}

To spot the missing ingredient in the recipe for reconstruction of the web from
its nonuniformity tensor, let us view the problem through the lens of
$\web_\Omega$-adapted coordinates $(x_1,\ldots,x_m)$. Bearing in mind the
coordinate expression for $\kcurv(\web_\Omega)$ in terms of $\Omega =
h(x)\,dx_1\wedge dx_2\wedge\cdots\wedge dx_m$ found in the introduction, the
reconstruction problem can be restated as follows: given several mixed parital
derivatives $\smash{\frac{\partial^2 \log h}{\partial x_i\partial x_j}}$ of
a~function-germ $\log h\in C^\infty(\Rb{d})$ at $0$, find $h$. The solution is
determined by $\kcurv(\web_\Omega)$ only up to a multiplication by smooth
functions $f_i(x)=f_i(x_{\csep_i\!+1}, \ldots, x_{\csep_{i+1}})$ with
$\csep_i=\sum_{k=1}^{i-1}\codim\fol_k$ for $i=1,2,\ldots,n$. There is a way to
resolve this ambiguity using certain ``initial conditions'' for the function
$h$, which determine uniquely the correction factors $f_i$. Moreover, any such
initial condition can be brought into a trivial normal form via a change of
$\web_\Omega$-adapted coordinates. We will now give more details on these
results.

In the following theorem we use the concept of the \emph{pullback} $\iota^*E$
of a smooth vector bundle $E\to M$ along an embedded submanifold
$S\overset{\iota}{\hookrightarrow} M$. Its sections, called \emph{sections of
$E$ along $S$}, are exactly the precompositions $\sigma\circ\iota$ of sections
$\sigma\in\Gamma(E)$ with the submanifold inclusion $\iota$.

\begin{thm}
  \label{thm:dfw-normal-recovery}
  Let $\web=(M,\fol_1,\ldots,\fol_n)$ be the germ at $p\in M$ of some
  regular $n$-web on an $m$-dimensional manifold $M$ and, for each
  $i=1,2,\ldots,n$, let $\fol^c_i$ be the foliations complementary to $\fol_i$
  generated by the tangent distributions $\bigcap_{j\neq i}T\fol_j$. Let $F_i$
  be the leaf-germ of $\fol_i^c$ crossing $p$ relative to some $\web$-adapted
  chart with the corresponding
  inclusion $\iota_i:F_i\hookrightarrow M$. Given
  \begin{enumerate}[label=$(\roman{enumi})$, ref=\roman{enumi}]
    \item \label{thm:dfw-normal-recovery:cross}
      nonvanishing smooth section-germs
      $\Omega_{0,i}\in\Gamma(\iota_i^*(\bigwedge^mTM))$ at $p$ for
      $i=1,2,\ldots,n$ such that $(\Omega_{0,i})_{|p}=(\Omega_{0,j})_{|p}$ for
      each pair of different indices $i,j$,
  \end{enumerate}
  and
  \begin{enumerate}[label=$(\roman{enumi})$, ref=\roman{enumi}, start=2]
    \item \label{thm:dfw-normal-recovery:a}
      a covariant $2$-tensor field-germ $A$ at $p$
      satisfying the identities of nonuniformity tensors
      \begin{enumerate}[label=$(\arabic{enumii})$, ref=\arabic{enumii}]
        \item\label{thm:dfw-normal-recovery:a:sym}
          $A(X,Y)=A(Y,X)$ for every $X,Y\in \mathfrak{X}(M)$,
        \item\label{thm:dfw-normal-recovery:a:diag}
          $A(X,Y)=0$ for every $X,Y\in \Gamma(T\fol^c_i)$, where
          $i=1,2,\ldots,n$, and
        \item\label{thm:dfw-normal-recovery:a:compat}
          $X\,A(Y,Z) = Y\,A(X,Z)$ for
          $X\in\Gamma(T\fol^c_i)$, $Y\in\Gamma(T\fol^c_j)$ and
          $Z\in\Gamma(T\fol^c_k)$, where $i,j,k=1,2,\ldots,n$ are such that
          $i\neq k$ and $j\neq k$,
    \end{enumerate}
  \end{enumerate}
  there exists a unique volume form-germ $\Omega\in\Omega^m(M)$ at $p$
  satisfying $\Omega\circ\iota_i = \Omega_{0,i}$ for $i=1,\ldots,n$ such that
  $A=\kcurv(\web_\Omega)$, where $\web_\Omega$ is the induced divergence-free
  $n$-web-germ $(M, \fol_1, \ldots,
  \fol_n, \Omega)$.
\end{thm}
\begin{proof}
  Fix a $\web$-adapted coordinate system $(x_1,\ldots,x_m)$ centered at $p$.
  In these coordinates $F_i=\smset{x\in\Rb{m} : \forall_{k\not\in\pi_i} \
  x_k=0}$, where $\pi_i=\smset{\csep_i+1,\ldots,\csep_{i+1}}$ with
  $\csep_i=\sum_{k=1}^{i-1}\codim\fol_k$ for $i=1,\ldots,n$. Write
  $\Omega_{0,i} = h_i(x_{\csep_i\!+1},\ldots,x_{\csep_{i+1}})\
  dx_1\wedge\cdots\wedge dx_m$ for some function-germs $h_i\in
  C^\infty(F_i)$ at $0$, and let $\tilde{h}\in C(\bigcup_{i=1}^n F_i)$ be the
  function-germ defined by $\tilde{h}(x) = h_i(x_{\csep_i\!+1},\ldots,
  x_{\csep_{i+1}})$ whenever $x\in F_i$. This function is well defined since
  $h_i(0)=h_j(0)$ for $i\neq j$ by assumption.
  
  Our goal, expressed using the notation introduced in $(\ref{eq:dfw-sim})$, is
  to produce a smooth extension $h\in C^\infty(U)$ of $\tilde{h}$ to a
  neighbourhood $U$ of $0$ such that $\smash{\frac{\partial^2 \log h}{\partial
  x_j \partial x_k} = A_{jk}}$ for each pair of indices $j\not\sim k$, and to
  show that it is unique; taking $\Omega = h(x)\,dx_1\wedge\cdots\wedge dx_m$
  will end the proof. The above system of partial differential equations is
  equivalent to its integral counterpart
  \begin{equation}
    \label{eq:dfw-normal-recovery-int}
    \begin{aligned}
      h(x)
        &= \left(\frac{
          h(x_1,\ldots,x_{j-1},0,x_{j+1},\ldots,x_m)
          \,h(x_1,\ldots,x_{k-1},0,x_{k+1},\ldots,x_m)
        }{
          h(x_1,\ldots,x_{j-1},0,x_{j+1},\ldots,x_{k-1},0,x_{k+1},\ldots,x_m)
        }\right)\\
        &\hskip2em\cdot\exp\left(
          \int_0^{x_{j}}\!\!\!\int_0^{x_k}
          \,A_{jk}(x_1,\ldots,t_j,\ldots,t_k,\ldots,x_m)
          \ dt_kdt_j
        \right)\quad\text{for each }j\not\sim k.
    \end{aligned}
  \end{equation}
  Using these equations we can express the value of $h$ at $x=(x_1,\ldots,x_m)$
  in terms of components of $A$ and values of $h$ at points with strictly
  smaller number of nonzero coordinates. While there are several such
  expressions, one for each pair of indices $j\not\sim k$, they are in fact
  equal to each other. To prove this, denote the respective right-hand sides of
  $(\ref{eq:dfw-normal-recovery-int})$ by $\rho_{jk}$ and rewrite the double
  integral of $A_{jk}$ as
  \begin{equation}
    \label{eq:dfw-normal-recovery-compat}
    \begin{aligned}
      &\int_0^{x_{j}}\!\!\!\int_0^{x_k}\left(\int_0^{x_l}
        \,\basis{x_l}A_{jk}(x_1,\ldots,t_j,\ldots,t_k,\ldots,t_l,\ldots,x_m)
        \ dt_l\right)\,dt_kdt_j \\
      &\hskip4em + \int_0^{x_{j}}\!\!\!\int_0^{x_k}
          \,A_{jk}(x_1,\ldots,t_j,\ldots,t_k,\ldots,0,\ldots,x_m)
          \ dt_kdt_j
    \end{aligned}
  \end{equation}
  for some index $l\not\sim k$. Then, expand the second summand using
  $(\ref{eq:dfw-normal-recovery-int})$ for indices $j,k$ at
  $(x_1,\ldots,x_{l-1},0,x_{l+1},\ldots,x_m)$ and insert the
  outcome back into $(\ref{eq:dfw-normal-recovery-int})$. By property
  $(\ref{thm:dfw-normal-recovery:a:compat})$ of the tensor field $A$ and
  Fubini's theorem, the expression for $h(x)$ obtained in this way is symmetric
  with respect to indices $j,l$, hence is equal to both
  $\rho_{jk}$ and $\rho_{lk}$. Since $\rho_{jk}=\rho_{kj}$, the right-hand
  side of $(\ref{eq:dfw-normal-recovery-int})$ does not depend on the choice of
  the indices $j,k$.

  This lets us prove by induction that the desired extension of $\tilde{h}$
  exists, is well-defined and unique. One way to set up the induction is to
  consider subsets of the set of indices $I\subseteq [m]$ and
  linear subspaces $F_I$ of points $x=(x_1,\ldots,x_m)$ satisfying $x_k=0$ for
  $k\not\in I$. Denote by $\pi_i$ the partition of $[m]$ given by
  $(\ref{eq:dfw-part})$. The base case corresponds to the fact that the
  function $h$ is defined on each $F_I$ for $I\subseteq\pi_i$ with
  $i=1,2,\ldots,n$ and is equal to $\tilde{h}$. Now, equality
  $(\ref{eq:dfw-normal-recovery-int})$ allows us to extend $h$ from $F_{I'}\cup
  F_{I''}$ to $F_I$ in a smooth and unique way if $I = I'\cup\smset{k} =
  I''\cup\smset{l}$ for some $k\not\sim l$. If we consider any $I$ not covered
  by the base case such that $h$ extends smoothly and uniquely to each $F_J$
  for $J\subsetneq I$, then $I\cap\pi_i\neq\varnothing$ and
  $I\cap\pi_j\neq\varnothing$ for some $i\neq j$; hence we can take
  $I'=I\setminus\smset{k}$ and $I''=I\setminus\smset{l}$ for some $k\in
  I\cap\pi_i$ and $l\in I\cap\pi_j$ to extend $h$ to $F_I$ using
  $(\ref{eq:dfw-normal-recovery-int})$, finishing the induction step and ending
  the proof.
\end{proof}

\begin{cor}
  \label{thm:dfw-normal-invariants}
  A divergence-free $n$-web-germ $\web_\Omega=(M,\fol_1,\ldots,\fol_n,\Omega)$
  at $p\in M$ is uniquely determined by its nonuniformity tensor-germ
  $\kcurv(\web_\Omega)$ at $p$ and values of $\Omega$ along the union
  $\bigcup_{k=1}^n F_k$ of the leaf-germs $F_k\in\fol_k^c$ passing
  through $p$.
\end{cor}
\begin{proof}
  The nonuniformity tensor $A:=\kcurv(\web_\Omega)$ satisfies the conditions
  listed in $(\ref{thm:dfw-normal-recovery:a})$ of Theorem
  \ref{thm:dfw-normal-recovery}. The volume form $\Omega$ is a valid extension
  of the forms $\Omega\circ\iota_i$ along the inclusions
  $\iota_i:F_i\hookrightarrow\Rb{d}$. Moreover, $\kcurv(\web_\Omega)=A$ holds,
  hence the claim follows from uniqueness in Theorem
  \ref{thm:dfw-normal-recovery}.
\end{proof}

While the set of compatible tuples of volume form-germs along the leaves
crossing $p$ is quite large, there is a way to put each such initial condition
into a trivial normal form with an appropriate choice of a coordinate system.

\begin{lem}
  \label{thm:dfw-normal-coords}
  Let $\web_\Omega=(M,\fol_1,\ldots,\fol_n,\Omega)$ be a germ at $p\in M$ of a
  divergence-free $n$-web, and let $F_i$ be the leaf-germ of $\fol_i^c$
  crossing $p$, where $\fol_i^c$ is given by $T\fol_i^c=\bigcap_{j\neq i}
  T\fol_j$, for $i=1,2,\ldots,n$.
  \begin{enumerate}[label=$(\arabic{enumi})$, ref=\arabic{enumi}]
    \item \label{thm:dfw-normal-coords:exists}
      There exists a $\web_\Omega$-adapted coordinate system $(x_1,\ldots,x_m)$
      in which $\Omega_{|q} = \Lambda_{|q}$ for all $q\in\bigcup_{i=1}^n F_i$,
      where $\Lambda=dx_1\wedge\cdots\wedge dx_m$ is the unit volume form.
    \item \label{thm:dfw-normal-coords:transitions}
      Let $\csep_i=\sum_{k=1}^{i-1}\codim\fol_k$ for $i=1,2,\ldots,n+1$. Any two
      such coordinate systems $(x_1,\ldots,x_m)$ and $(y_1,\ldots,y_m)$ differ
      by a~change of variables of the form
      \begin{equation}
        (y_{\csep_i\!+1},\ldots,y_{\csep_{i+1}})
          = \mathbf{y}_i
            \mapsto
        \mathbf{x}_i(\mathbf{y}_i)
          = (x_{\csep_i\!+1}(\mathbf{y}_i),\ldots,
            x_{\csep_{i+1}}(\mathbf{y}_i))
      \end{equation}
      with constant determinant
      $\det\frac{\partial\mathbf{x}_i}{\partial\mathbf{y_i}}(\mathbf{y}_i)$
      followed by a~permutation of coordinates $\mathbf{x}_i\mapsto
      \mathbf{x}_{\sigma(i)}$ for some $\sigma\in S^n$ satisfying
      $\codim\fol_{\sigma(k)} = \codim\fol_k$, such that
      $\operatorname{sgn}\sigma|_{J_o} \cdot\prod_{i=1}^n
      \det\frac{\partial\mathbf{x}_i}{\partial\mathbf{y_i}}(\mathbf{y}_i)=1$,
      where by $J_o$ we denote the set of indices $j=1,\ldots,n$ for which
      $\codim\fol_j$ is odd.
  \end{enumerate}
\end{lem}
\begin{proof}
  Pick a $\web_\Omega$-adapted coordinate system $(y_1,\ldots,y_m)$ centered at $p$,
  then express $\Omega$ as $h(y)\, dy_1\wedge\cdots\wedge dy_m$ for some $h\in
  C^\infty(M)$ and make the substitution
  \begin{equation}
    x_{\csep_i\!+1} = \frac{1}{h(0)}\int_0^{y_{\csep_i\!+1}}
      h(0,\ldots,0,t,y_{\csep_i\!+2},\ldots,y_{\csep_{i+1}},0,\ldots,0)\,dt,
  \end{equation}
  for each $i=1,\ldots,n$, while keeping the other coordinates unaltered. If we
  name this transformation $\varphi$, then via direct calculation we obtain for
  each $i=1,\ldots,n$ and $q\in F_i$ the equality
  $\varphi^*(h(0)\,dx_1\wedge\cdots\wedge dx_m)_{|q} = \Omega_{|q}$. Further
  linear rescaling of coordinates yields the desired coordinate system.

  Any transition map $y = \varphi(x)$ between two such systems
  $(y_1,\ldots,y_m)$ and $(x_1,\ldots,x_m)$ is an equivalence of the underlying
  regular web, and as such it takes the form $\mathbf{y}_i =
  \varphi_i(\mathbf{x}_{\sigma(i)})$ for $i=1,\ldots,n$ and some permutation
  $\sigma\in S_n$ satisfying $\codim\fol_{\sigma(i)}=\codim\fol_i$ for each
  $i=1,\ldots,n$. At each point $q\in F_i$ we have
  \begin{equation}
    \label{eq:dfw-normal-coords:det}
    \begin{aligned}
      (\Omega_{0,i})_{|q} 
        &= dx_1\wedge\cdots\wedge dx_m = dy_1\wedge\cdots\wedge dy_m \\
        &= \pm \left(\prod_{j\neq i}
          \det\frac{\partial\varphi_{\sigma^{-1}(j)}}{\partial\mathbf{x}_j}(0)
        \right)\cdot
          \det\frac{\partial\varphi_{\sigma^{-1}(i)}}{\partial\mathbf{x}_i}
            (q_{\csep_i\!+1},\ldots,q_{\csep_{i+1}})
          \ dx_1\wedge\cdots\wedge dx_m.
    \end{aligned}
  \end{equation}
  This equality forces $\det\frac{\partial \varphi_i}{\partial
  \mathbf{x}_{\sigma(i)}}$ to be constant for each $i=1,\ldots,n$. Evaluating
  it at $0$ we obtain that
  \begin{equation}
    \pm \prod_{i=1}^n
      \det\frac{\partial\varphi_{\sigma^{-1}(j)}}{\partial\mathbf{x}_j} = 1,
  \end{equation}
  where the correct sign is found by noting that if we write $\sigma$ as a
  product of transpositions
  $\mathbf{x}_i\overset{\tau}{\longleftrightarrow}\mathbf{x}_j$ satisfying
  $\codim\fol_i=\codim\fol_j=c$ for some $c\in\mathbb{N}$, then each $\tau$
  corresponds to $c$ transpositions of individual variables $x_k$, contributing
  a factor $(-1)^c$ to $\det d\varphi$.
\end{proof}

The above normalization of the coordinate system is a generalization of
a~theorem of Tabachnikov \cite{2-webs} on normal forms of the volume elements
associated with divergence-free $2$-webs. Pick an arbitrary divergence-free
$n$-web $\web_\Omega$ and a coordinate system $(x_1,\ldots,x_m)$ centered at
$p\in M$ normalized by means of Lemma \ref{thm:dfw-normal-coords} with respect
to $\Omega$. Let $\mathfrak{m}=\langle x_1,x_2,\ldots,x_m\rangle$ be the
maximal ideal of the ring of smooth functions-germs at the point $p$. In these
coordinates, the density $h\in C^\infty(M,p)$ of the volume form-germ $\Omega =
h(x)\, dx_1\wedge\cdots\wedge dx_m$ can be expanded into
\begin{equation}
  h(x) = 1
    + \textstyle{\sum_{i<j,\,\csep_i<k\leq \csep_{i+1}\!,
      \,\csep_j<l\leq \csep_{j+1}}} \kappa_{kl}x_kx_l + f(x),
\end{equation}
where $\csep_i = \sum_{k=1}^{i-1}\codim\fol_k$ and $\kappa_{kl}$ are exactly
the coefficients at $p$ of the nonuniformity tensor $\kcurv(\web_\Omega)_{|p} =
\smsum{i,j}\kappa_{ij}\,dx_i\,dx_j$, while the function-germ
$f\in\mathfrak{m}^3$ vanishes on the union of leaves passing through the origin
$F_i\in\fol_i^c$ of the complementary foliations $\fol_i^c$ generated by the
tangent distribution $T\fol_i^c=\bigcap_{j\neq i}T\fol_j$. When the web
consists of two foliations, the displayed equality reduces exactly to
\cite[Theorem 0.2, (ii)]{2-webs} when taken modulo $\mathfrak{m}^3$.

For a fixed regular $n$-web $\web$, denote by $\mathfrak{G}_\web$ the group of
all coordinate transformation-germs $\maps{\varphi}{(\Rb{m},0)}{(\Rb{m},0)};
(y_1,\ldots,y_m)\mapsto (x_1,\ldots,x_m)$ preserving the web $\web$ and
satisfying $(\varphi^*\Lambda)_{|q} = \Lambda_{|q}$ for the unit volume
form $\Lambda = dx_1\wedge dx_2\wedge\cdots\wedge dx_m$ and
each $q\in\bigcup_{i=1}^n F_i$ with $0\in F_i\in\fol_i^c$, as in Lemma
\ref{thm:dfw-normal-coords}. The reduction of the group of all $\web$-adapted
coordinate transformations to $\mathfrak{G}_\web$, made possible by Lemma
\ref{thm:dfw-normal-coords}, allows us to reformulate the classification
problem of divergence-free $n$-webs as follows.

\begin{thm}
  Divergence-free $n$-web-germs with an underlying regular $m$-dimensional
  $n$-web $\web$ are classified up to equivalence by orbits of the action of
  $\mathfrak{G}_{\web}$ by pullback on germs at $0$ of covariant symmetric
  $2$-tensor fields $A$ on $\Rb{m}$ satisfying the identities of nonuniformity
  tensors listed in Theorem \ref{thm:dfw-normal-recovery} via the
  correspondence $\web_\Omega \mapsto \mathfrak{G}_\web (\psi^{-1})^*
  \kcurv(\web_\Omega)$, where $\psi(x)=(x_1,\ldots,x_m)$ is any
  $\web_\Omega$-adapted coordinate system normalized by means of Lemma
  \ref{thm:dfw-normal-coords}.
\end{thm}
\begin{proof}
  By the uniqueness claim of Theorem \ref{thm:dfw-normal-recovery}, an
  equivalence between two divergence-free web-germs $\web_{\tilde\Omega}$ and
  $\web_\Omega$ is the same as a change of variables
  $\varphi\in\mathfrak{G}_\web$ between the corresponding normalized coordinate
  systems $\tilde\psi$, $\psi$ which satisfies
  $\varphi^*(\psi^{-1})^*\kcurv(\web_\Omega) = (\tilde\psi^{-1})^*
  \kcurv(\web_{\tilde\Omega})$. Since, again by Theorem
  \ref{thm:dfw-normal-recovery}, any tensor field-germ $A$ of the form given
  above is a nonuniformity tensor of a divergence-free $n$-web-germ
  $\web_\Omega$ in some normalized $\web_\Omega$-adapted coordinates, the two
  classification problems are equivalent.
\end{proof}

Note that, in contrast to webs of higher codimension, the group
$\mathfrak{G}_\web$ is finite-dimensional for codimension-$1$ webs $\web$.
In this case, each map $\varphi\in\mathfrak{G}_\web$ has to be linear, since
the requirement $\det\frac{\partial \varphi_i}{\partial x_{\sigma(i)}}
=\operatorname{const.}$ for each $i=1,\ldots,m$ enforced during a change of
variables $\varphi(x_1,\ldots,x_m) = (\varphi_1(x_{\sigma(1)}),\ldots,
\varphi_m(x_{\sigma(m)}))$, $\sigma\in S_m$, between any two coordinate
systems normalized in the sense of Lemma \ref{thm:dfw-normal-coords}
reduces to the constancy of $d\varphi$. This opens up a possibility to classify
some generic codimension-$1$ divergence-free $n$-webs using elementary tools.

\begin{exmp}
  \label{ex:dfw-normal-2c1}
  Our goal is to determine a canonical form of a generic divergence-free
  $2$-web's volume element $\Omega$. Consider a germ at $0$ of a planar
  divergence-free $2$-web $(\Rb{2},\fol, \gol, \Omega)$ with $\Omega =
  h(x,y)\,dx\wedge dy$, $T\fol=\ker dx$ and $T\gol=\ker dy$. By Lemma
  \ref{thm:dfw-normal-coords}, we can choose the coordinate system $(x,y)$ so
  that $h(x,0)=h(0,y)=1$. Moreover, the Lemma states that any other such
  coordinate system $(\tilde{x},\tilde{y})$ differs from $(x,y)$ by a linear
  change of variables of the form
  \begin{equation}
    (\tilde{x},\tilde{y}) = (cx,c^{-1}y)\quad\text{or}
      \quad(\tilde{x},\tilde{y})=(cy,-c^{-1}x)
      \quad\text{for }c\in\mathbb{R}\setminus\set{0}.
  \end{equation}
  Since the nonuniformity tensor $\kcurv(\web_\Omega) = \kappa(x,y)\,dx\,dy =
  \frac{\partial\log h}{\partial x\,\partial y}\,dx\,dy$ is itself a
  divergence-free web-invariant, its covariant derivative with respect to the
  natural connection $\Theta$ with covariant derivative $\nabla$ is also an
  invariant. It is equal to $\nabla\kcurv(\web_\Omega)=(d\kappa - \kappa\,d\log
  h) \otimes (dx\,dy)$; hence its value at the origin is $d\kappa\otimes
  (dx\,dy)$, since $h$ is constant along the leaves $\set{x=0}$ and
  $\set{y=0}$.

  Suppose now that the genericity condition $\frac{\partial \kappa}{\partial
  x}(0) \neq 0, \frac{\partial \kappa}{\partial y}(0) \neq 0$ is satisfied.
  Using $\pi/2$-rotations and transformations of the form
  $(x,y)\mapsto(cx,\frac{1}{c}y)$ for $c>0$ we put $\nabla\kcurv(\web_\Omega)$
  into a form which satisfies $\frac{\partial \kappa}{\partial x}(0) =
  \frac{\partial \kappa}{\partial y}(0) > 0$. Now, the canonical form of the
  volume density function $h$ can be computed using Hadamard's lemma. It is
  equal to
  \begin{equation}
    \label{eq:dfw-normal-2c1}
      h(x,y) = 1 + xy(\kappa_0 + \smfrac{1}{2}a\cdot(x+y)
         + x^2\,\tilde{g}_x(x) + xy\,\tilde{g}_{xy}(x,y) +
        y^2\,\tilde{g}_y(y)\,)
  \end{equation}
  for some function-germs $\tilde{g}_x,\tilde{g}_y\in C^\infty(\Rb{1},0)$,
  $\tilde{g}_{xy}\in C^\infty(\Rb{2},0)$ and real constants
  $\kappa_0\in\Rb{1}$, $a\in\Rb{1}_+$. The two scalar invariants correspond to
  \begin{equation}
    \kappa_0 = \frac{\kappa(0)}{h(0)},\quad
    a = \abs*{\frac{
      (h(0)\smfrac{\partial{\kappa}}{\partial{x}}(0)
        - \smfrac{\partial{h}}{\partial{x}}(0)\kappa(0))
      \cdot(h(0)\smfrac{\partial{\kappa}}{\partial{y}}(0)
        - \smfrac{\partial{h}}{\partial{y}}(0)\kappa(0))}{h(0)^{5}}
      }^{1/2}
  \end{equation}
  in any
  $\web_\Omega$-adapted coordinate system $(x,y)$. Two divergence-free
  $2$-web-germs satisfying the genericity condition are locally equivalent if
  and only if their canonical forms coincide.
\end{exmp}

\section{Volume-preserving holonomy and geometric invariants}

Let $(M,\fol_1,\ldots,\fol_n,\Omega)$ be a codimension-$1$ divergence-free
$n$-web on a smooth manifold $M$. Here, we give a certain geometrical
interpretation of the nonuniformity tensor $\kcurv(\web_\Omega)$ inspired by
the works of Blaschke, Thomsen and Bol on planar $3$-webs \cite{gewebe,
pereira} and based upon the interpretation of curvature of bi-Lagrangian
structures given by Tabachnikov in dimension $2$ \cite{2-webs}. In the latter
work, the curvature of the bi-Lagrangian connection $\nabla$ at the point $p$
was approximated by certain ratios between adjacent volumes enclosed by the
leaves of the web crossing $p$ (Figure \ref{fig:dfw-taba}). Our approach goes
along similar lines.

Let $M$ be a smooth manifold and let $\web$ be a regular codimension-$1$
$n$-web on $M$. For each $x,y\in\Rb{1}$ let us treat $\ival{x,y}$ as an
ordinary closed interval $[\min(x,y), \max(x,y)]$ oriented positively if
$x\leq y$ and negatively if $x>y$. Now, given any pair of vectors
$a,b\in\Rb{m}$, we let $\ival{a,b} = \prod_{k=1}^m \ival{a_k,b_k}$ in the sense
given above with the product orientation.
\begin{defn}
  \label{def:dfw-region}
    Let $p,q\in M$. A compact set $C\subseteq M$ which takes the form
    $\ival{\varphi(p),\varphi(q)}$ for some $\web$-adapted chart $(U,\varphi)$
    will be called a \emph{region bounded by the leaves of $\web$ crossing $p$
    and $q$}. When the chart $(U,\varphi)$ is implied, we will omit $\varphi$
    and denote $C$ by $\ival{p,q}$ for convenience.
\end{defn}

Every region bounded by leaves $C$ has finite volume with respect to the volume
form $\Omega$ (henceforth abbreviated by $\Vol{\Omega}(C)$). This
$\Omega$-volume, defined as
\begin{equation}
  \Vol{\Omega}(C) = \abs*{\,\int_C \Omega\,},
\end{equation}
is an essential ingredient in the construction of the \emph{reflection holonomy
group} of a divergence-free web. The reflection holonomy group shares many
similarities with the holonomy group of a planar $3$-web, which measures the
degree of its \emph{non-hexagonality} at a given point $p\in\Rb{2}$
\cite{gewebe, pereira} or, in other words, the extent to which planar,
curvilinear figures formed by the leaves of $\web$ with vertices lying on the
leaves crossing $p$ are not closed (see Figure \ref{fig:dfw-rhol}). The $3$-web
is hexagonal precisely when a certain natural connection associated with it is
flat.

\begin{figure}
  \centering
  \def\svgwidth{\textwidth}
  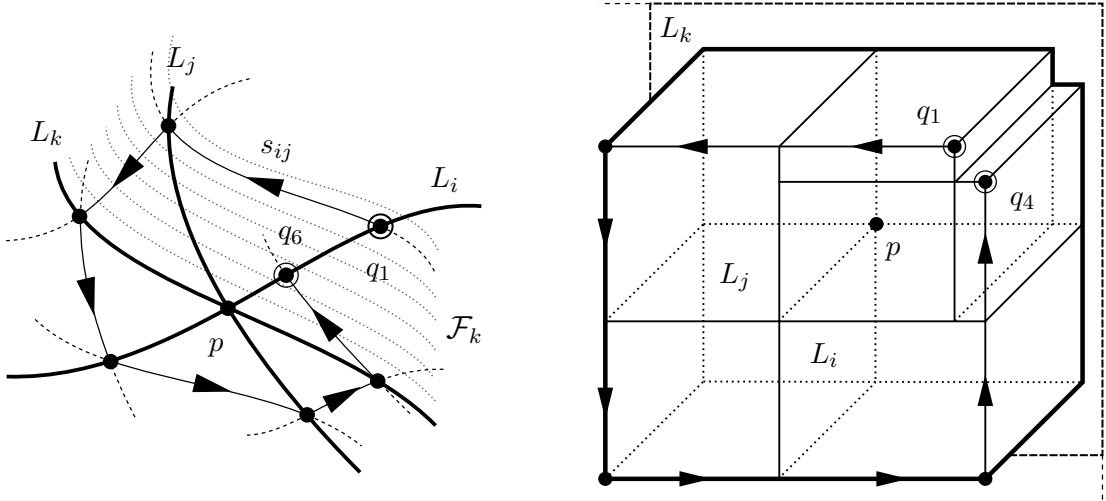
  \caption{\footnotesize Left: planar $3$-web holonomy. Right:
  volume-preserving holonomy of codimension-$1$ $3$-web. The volumes of
  adjacent cubes are the same.}
  \label{fig:dfw-rhol}
\end{figure}

The property of hexagonality is expressed in the nontriviality of the group
generated by the smooth map-germ $\maps{\ell}{(F,p)}{(F,p)}$ taking one end of
each such figure's perimeter to the other. This transformation is itself a
$6$-fold composition of map-germs $\maps{s_{ij}}{(F_i,p)}{(F_j,p)}$, each
transporting the points on $F_i\in\fol_i$ to the points on $F_j\in\fol_j$ along
the leaves of $\fol_k$ for $\web=(\Rb{2},\fol_i,\fol_j,\fol_k)$ and some leaves
$F_i, F_j$ crossing the center $p$.

We build the volume-preserving holonomy groups in a similar way, by first
introducing a way to transport the points along the leaves of a codimension-$1$
divergence-free web $\web_\Omega$ with the help of the volume form $\Omega$.
Afterwards, we measure the nonuniformity tensor $\kcurv(\web_\Omega)$ by means
of compositions of such transport mappings.

\subsection{Reflections, loops and reflection holonomy}

Let us first clarify the construction of the point-transport map utilizing the
volume form. The next lemma requires an auxiliary definition.

\begin{defn}
  \label{def:dfw-adj}
  Let $\web$ be a regular codimension-$1$ regular $n$-web. Two regions $A,B$
  bounded by leaves of $\web$ are said to be \emph{adjacent} if $C=A\cup B$ is
  also a~region bounded by leaves of $\web$ and $A\cap B$ is entirely contained
  in a leaf $F$ of some $\fol\in\Fols(\web)$ (in which case we say that $A$ and
  $B$ are adjacent \emph{along $F$}, or that $F$ \emph{subdivides} the region
  $C$ into subregions $A$ and $B$).
\end{defn}

Given a foliation $\fol$ and an open subset $U$ of its domain, we denote by
$\fol_{|U}$ the foliation of $U$ whose leaves are exactly the connected
components of $F\cap U$ for $F$ the leaves of $\fol$.

\begin{lem}
  \label{thm:dfw-refl}
  Let $\web_\Omega$ be a divergence-free codimension-$1$ $n$-web on $M$ and let
  $\fol\in\Fols(\web_\Omega)$. There exists an open neighbourhood $U$ of $p$
  such that to each point $q\in U$ not lying on the leaf $G$ of
  $\gol_{|U}$ crossing $p$ for each $\gol\in\Fols(\web_\Omega)$ there
  corresponds a unique point $q'\in U$ different from $q$ such that the regions
  $A$ and $B$ bounded by leaves of $\web_\Omega$ crossing $p,q$ and $p,q'$
  respectively are adjacent along $\fol$ and have the same $\Omega$-volume. The
  correspondence $q\mapsto q'$ extends uniquely to a smooth map-germ
  $\maps{r_{p;\fol}}{(M,p)}{(M,p)}$.
\end{lem}
\begin{proof}
  Fix a $\web_\Omega$-adapted coordinate system $(x_1,\ldots,x_n)$ centered at
  $p$ such that $T\fol=\ker dx_i$ and write $\Omega=h(x)\, dx_1\wedge
  dx_2\wedge\cdots\wedge dx_n$ for some $h\in C^\infty(M)$. Whenever two given
  points $q=(u_1,\ldots,u_n)$ and $q'=(u'_1,\ldots,u'_n)$ lie on the same leaf
  $F=\set{\forall_{j\neq i}\ x_j = u_j}$ of $\fol^c$, the difference between
  the integrals
  \begin{equation}
    \label{eq:dfw-refl-implicit1}
    g_i(q,u'_i)=g_i(u_1,u_2,\ldots,u_n,u'_i)
      = \Int_0^{u_i}\!f_i(q,s)\,ds - \Int_{u'_i}^0\!f_i(q,s)\,ds
  \end{equation}
  of the function $f_i$ defined by
  \begin{equation}
    \label{eq:dfw-refl-implicit2}
    \begin{aligned}
      &f_i(q,s) = f_i(u_1,u_2,\ldots,\widehat{u_i},\ldots,u_n,s) \\
        &\hskip 1em= \IterInt{0}{1}{0}{1} h(t_1u_1,\,\ldots,\, t_{i-1}u_{i-1},
          \,s,\,t_{i+1}u_{i+1},\,\ldots,\,t_nu_n)
          \ dt_1\cdots \widehat{dt_i}\cdots dt_n
    \end{aligned}
  \end{equation}
  is zero precisely when $\int_{\ival{p,q}}\Omega = -\int_{\ival{p,q'}}\Omega$,
  which, under the assumption $u_j\neq 0$ for $j=1,\ldots,n$, is true if
  and only if $\Vol{\Omega}(A) = \Vol{\Omega}(B)$ and $q\neq q'$.
  The uniqueness of the point $q'$ for fixed $q$ is guaranteed by the strict
  monotonicity of $g$ with respect to $u'_i$. Its local existence is
  deduced using the classical implicit function theorem applied to the function
  $g$ to obtain a smooth function $\maps{\rho_i}{\widetilde{U}}{\Rb{1}}$
  satisfying
  \begin{equation}
    \label{eq:dfw-refl-implicit3}
    g_i(q,\rho_i(q))
      = g_i(0) = \Int_0^0\! f_i(0,s)\,ds - \Int_0^0\! f_i(0,s)\,ds = 0,
  \end{equation}
  where $\widetilde{U}$ is an open neighbourhood of $p$ equal to
  $(-\varepsilon,\varepsilon)^n$ in coordinates for some $\varepsilon>0$. Put
  \begin{equation}
    r_{p;\fol_i}(u_1,\ldots,u_n) =
      (u_1,\ldots,u_{i-1},\rho_i(u_1,\ldots,u_n),u_{i+1},\ldots,u_n)
  \end{equation}
  to obtain a smooth map $\maps{r_{p;\fol_i}}{\widetilde{U}}{M}$ extending the
  correspondence $q\mapsto q'$. Any other smooth extension of this
  correspondence coincides with $r_{p;\fol_i}$ on an open and dense subset of a
  sufficiently small open neighbourhood of $p$ by uniqueness of $q'$, hence is
  equal to $r_{p;\fol_i}$ as a map-germ by continuity. Thus, the smooth
  extension of $q\mapsto q'$ to a neighbourhood of $p$ is unique. Since
  $r_{p;\fol_i}(p) = p$, we can put $U=\widetilde{U}\cap
  r_{p;\fol_i}^{-1}(\widetilde{U})$ to finish the proof.
\end{proof}

\begin{defn}
  \label{def:dfw-reflections}
  Let $\fol\in\Fols(\web_\Omega)$. The map-germ
  $\maps{r_{p;\fol}}{(M,p)}{(M,p)}$ defined in the above proposition is called
  the \emph{volume-preserving reflection through $\fol$ anchored at $p$}.
\end{defn}

These reflections are defined in an invariant way, hence they behave well with
respect to local equivalences of divergence-free webs.

\begin{prop}
  \label{thm:dfw-refl-conj}
  If $\maps{\varphi}{(M,p)}{(N,s})$ is a germ of a local equivalence between
  codimension-$1$ webs $\web_M$ and $\web_N$ mapping the leaves of foliation
  $\fol$ onto the leaves of foliation $\tau(\fol)$ for some map
  $\maps{\tau}{\Fols(\web_M)}{\Fols(\web_N)}$, then the corresponding
  reflections are conjugate:
  \begin{equation}
    \label{eq:dfw-refl-conj}
    r_{s,\tau(\fol)}
      = \varphi\circ r_{p;\fol}\circ\varphi^{-1}
        \qquad\text{for each }\fol\in\Fols(\web_\Omega).
  \end{equation}
\end{prop}
\begin{proof}
  Choose representatives of $r_{p;\fol}$, $r_{s,\tau(\fol)}$ with their
  respective domains $U$, $V$ as in the statement of Lemma
  \ref{thm:dfw-refl}. Moreover, assume without loss of generality that
  $\varphi=\varphi_{|U}$ is a divergence-free web-equivalence and
  $\varphi(U)=V$. Now, if two different points $q,q'$ span adjacent regions
  $A$, $B$ bounded by the leaves of $\web_M$ of equal
  volume crossing $p,q$ and $p,q'$ respectively, their images $\varphi(q),
  \varphi(q')$ also do so, since $\varphi$ is a volume-preserving
  web-equivalence. This immediately leads to the implication
  \begin{equation}
    q' = r_{p;\fol}(q) \ \implies\ \varphi(q') = r_{s,\tau(\fol)}(\varphi(q))
  \end{equation}
  for $q$ in a dense subset of $U$ by uniqueness, which by continuity gives
  the sought relation $(\ref{eq:dfw-refl-conj})$.
\end{proof}

In a similar vein one proves the following properties of volume-preserving
reflection mappings.

\begin{lem}
  \label{thm:dfw-refl-props}
  Let $\web_\Omega$ be a divergence-free $n$-web of codimension $1$ on $M$ with
  volume form $\Omega$. Fix $p\in M$ and $\fol\in\Fols(\web_\Omega)$ and denote
  by $r_{p;\fol}$ the volume-preserving reflection of $\web_\Omega$ along
  $\fol$. Then
  \begin{enumerate}[label=$(\arabic{enumi})$,ref=\arabic{enumi}]
    \item \label{thm:dfw-refl-props:inv}
      $r_{p;\fol}\circ r_{p;\fol} = \mathrm{id}$,
    \item \label{thm:dfw-refl-props:proj}
      if $c\in\Rb{1}\setminus\set{0}$ and $\tilde{r}_{p;\fol}$ is the
      volume-preserving reflection map-germ corresponding to the
      divergence-free $n$-web $(M,\Fols(\web_\Omega),c\Omega)$, there is an
      equality $r_{p;\fol}=\tilde{r}_{p;\fol}$.
    \item \label{thm:dfw-refl-props:sign}
      in any $\web_\Omega$-adapted coordinates chart $(U, (x_1,\ldots,x_n))$
      around $p$ in which $T\fol=\ker dx_i$, the condition $x_i(q) > 0$ implies
      $x_i(r_{p;\fol}(q)) < 0$ for any $q\in U$.
  \end{enumerate}
\end{lem}
\begin{proof}
  Properties $(\ref{thm:dfw-refl-props:inv})$ and
  $(\ref{thm:dfw-refl-props:proj})$ both follow directly from Lemma
  \ref{thm:dfw-refl} and its uniqueness statement. For the last property
  $(\ref{thm:dfw-refl-props:sign})$ use Lemma \ref{thm:dfw-refl} together
  with the fact that the function $x\mapsto \int_{\ival{p,x}}\!\Omega$ is
  strictly monotone for generic $x\in U$ with respect to $x_i$.
\end{proof}

The invariance statement of Proposition \ref{thm:dfw-refl-conj} carries over
naturally to compositions of map-germs $r_{p;\fol}$ for
$\fol\in\Fols(\web_\Omega)$. The properties of $r_{p;\fol}$ listed above
suggest that the shortest among possibly nontrivial such compositions which
are actually comparable to the identity (hence can be used to gauge
nontriviality of the web $\web_\Omega$) are the commutators of two reflections.

\begin{defn}
  \label{def:dfw-loops}
  Let $p\in M$ and $\fol,\gol\in\Fols(\web_\Omega)$. The germ of the
  composition
  \begin{equation}
    \ell_{p;\fol,\gol}
      = r_{p,\gol}\circ r_{p;\fol}\circ r_{p,\gol}\circ r_{p;\fol}
  \end{equation}
  will be called a \emph{(volume-preserving) loop through $\fol$ and $\gol$
  anchored at $p$}.
\end{defn}

\begin{defn}
  Let $p\in M$. The commutator $[\mathcal{R}_p, \mathcal{R}_p]$ of the group
  $\mathcal{R}_p$ generated by volume-preserving reflections $r_{p;\fol}$
  anchored at $p$ for $\fol\in\Fols(\web_\Omega)$ will be called the
  \emph{(volume-preserving) reflection holonomy group of
  $(M,\fol_1,\ldots,\fol_n,\Omega)$ at $p$}.
\end{defn}

The reflection holonomy groups are generated by conjugates of the loops
$\ell_{p;\fol,\gol}$ by compositions of volume-preserving reflections. Their
triviality can be rephrased as commutativity of all pairs of reflections, i.e.
all loops being equal to identities in the neighbourhood of a given point.

\begin{exmp}
  \label{ex:dfw-refl-trivial}
  Let $\web_0=(\Rb{n},\fol_1,\ldots,\fol_n,\Lambda)$ be a trivial
  divergence-free $n$-web. In the standard coordinate system
  $(x_1,x_2,\ldots,x_n)$, in which $T\fol_i=\ker dx_i$ and $\Lambda =
  dx_1\wedge dx_2\wedge\cdots\wedge dx_n$, the reflection maps
  anchored at $p\in\Rb{n}$ become ordinary linear reflections
  \begin{equation}
    r_{p;\fol_i}(q_1,\ldots,q_{i-1},q_i,q_{i+1},\ldots,q_n)
      = (q_1,\ldots,q_{i-1},-q_i+2p_i,q_{i+1},\ldots,q_n).
  \end{equation}
  The reflection holonomy
  group $[\mathcal{R}_p,\mathcal{R}_p]$ of $\web_0$ is trivial.
\end{exmp}

The reflection holonomy groups $[\mathcal{R}_p,\mathcal{R}_p]$ can be treated
as covariants of the web $\web_\Omega$ via the following construction. Every
germ of a divergence-free web-equivalence $\varphi$ acts on the reflections at
a point $p$ by conjugation: if $\varphi$ maps the leaves of
$\fol\in\Fols(\web_\Omega)$ onto the leaves of $\tau(\fol)$, then $\varphi$
transforms $r_{p;\fol}$ into $r_{\varphi(p), \tau(\fol)} = \varphi\circ
r_{p;\fol}\circ \varphi^{-1}$. Hence, $\varphi$ defines a group isomorphism
$\varphi_*:\mathcal{R}_p\to \mathcal{R}_{\varphi(p)}$, which restricts to an
isomorphism of the corresponding commutators. If $p=\varphi(p)$, then
$\varphi_*$ is simply the conjugation by $\varphi$ inside $\mathrm{Diff}(M,p)$.

\begin{defn}
  Let $(U,\varphi)$ be a $\web_\Omega$-adapted chart on some open neighbourhood
  $U$ centered at $p\in M$. Then the conjugacy class of
  $\varphi_*[\mathcal{R}_p,\mathcal{R}_p]$ inside the group
  $\mathrm{Diff}(\Rb{n},0)$ of diffeomorphism-germs will be called the
  \emph{(volume-preserving) reflection holonomy of $\web_\Omega$ at $p$}.
\end{defn}

The reflection holonomy is a local invariant of the web by construction. Its
triviality is dependent upon the nonuniformity tensor of the web, as we will
soon show. Below we give an elementary example of a divergence-free web with
nontrivial holonomy.

\begin{exmp}
  \label{ex:dfw-holonomy-nontrivial}
  Let $(x,y)$ be the standard coordinates on $\Rb{2}$ and let
  $\web_0=(I^2,\fol,\gol,\Omega)$ be a divergence-free $2$-web on
  $I=(-1,1)\subseteq\Rb{1}$ with $\Omega = (1+xy)\,dx\wedge dy$. Its
  nonuniformity tensor is $\kcurv(\web_{0;\Omega}) = (1+xy)^{-2}\,dx\,dy\neq
  0$. Elementary computations allow us to recover the closed form of the
  volume-preserving reflections anchored at $0$:
  \begin{equation}
    \begin{aligned}
      r_{0,\fol}(x,y) &= (z(x,y), y), \\
      r_{0,\gol}(x,y) &= (x, z(y,x)),
    \end{aligned}
  \end{equation}
  where the smooth function
  \begin{equation}
    z(x,y) = \begin{cases}
      \frac{\sqrt{4(1-xy) - x^2y^2}-2}{y},& \quad\text{if }y\neq 0, \\[1ex]
      -x,& \quad\text{if }y=0
    \end{cases}
  \end{equation}
  is defined on some open neighbourhood of $0$. Using these we compute the
  closed form of the loop $\ell_{0;\fol,\gol}$, which turns out to be
  \begin{equation}
    \ell_{0;\fol,\gol}(x,y) = \textstyle{\left(
      \frac{(\sqrt{4(1-xy)-x^{2}y^{2}} - 2)^{2}}
        {xy^2},
      \frac{x^2y^3}
        {(\sqrt{4(1-xy)-x^{2}y^{2}} - 2)^{2}}
    \right)}.
  \end{equation}
  Since it generates the holonomy group of $\web_\Omega$ at $0$,
  $[\mathcal{R}_0, \mathcal{R}_0]$ is nontrivial.
\end{exmp}

\begin{figure}
  \centering
  \def\svgwidth{\textwidth}
  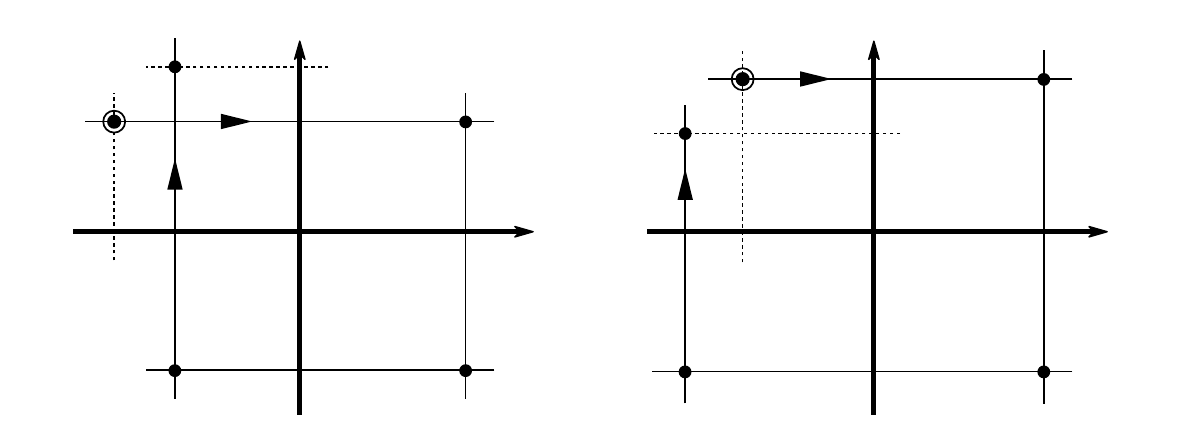
  \caption{\footnotesize The effect of the nonuniformity tensor
  $\kcurv(\web_\Omega) = \sum_{i\neq j} k_{ij}\,dx_idx_j$ on reflections.}
  \label{fig:dfw-curv}
\end{figure}

\begin{lem}
  \label{thm:dfw-refl-taylor}
  Let $\web_\Omega=(M,\fol_1,\ldots,\fol_n,\Omega)$ be a divergence-free
  $n$-web of codimension $1$. Fix a point $p\in M$ and a $\web_\Omega$-adapted
  coordinate system $(x_1,\ldots,x_n)$ centered at $p$. The volume-preserving
  reflection along the foliation $\fol_i\in\Fols(\web_\Omega)$ generated
  locally by $\ker dx_i$ satisfies
  \begin{equation}
    r_{p;\fol_i}(x) = (x_1,\ldots,x_{i-1},z(x),x_{i+1},\ldots,x_n)
  \end{equation}
  where
  \begin{equation}
    z(x) = - x_i - \alpha_{i}x_i^2
      - \alpha_{i}^2x_i^3
      - \sum_{j\neq i} \alpha_{ij}x_i^2x_j
      + o(|x|^3),
  \end{equation}
  with $\alpha_{i} = (\smfrac{\partial}{\partial{x_i}} \log h)(0)$ and
  $\alpha_{ij} = (\smfrac{\partial^2}{\partial{x_ix_j}} \log h)(0)$.
\end{lem}

\begin{lem}
  \label{thm:dfw-loop-taylor}
  Let $\web_\Omega=(M,\fol_1,\ldots,\fol_n,\Omega)$ be a divergence-free
  $n$-web of codimension $1$. Fix a point $p\in M$ and a $\web_\Omega$-adapted
  coordinate system $(x_1,\ldots,x_n)$ centered at $p$. Express the volume form
  as $\Omega=h(x)\,dx_1\wedge dx_2\wedge\cdots\wedge dx_n$ and the
  nonuniformity tensor of $\web_\Omega$ at $p$ as $\kcurv(\web_\Omega)_{|p} =
  \sum_{i\neq j} \kappa_{ij}\,dx_idx_j$, where $\kappa_{ij}=\frac{\partial\log
  h}{\partial x_i\,\partial x_j}(p)$. In this setting, the volume-preserving
  loop along the foliations $\fol_i$, $\fol_j$ with $T\fol_i=\ker dx_i$
  and $T\fol_j=\ker dx_j$ satisfies
  \begin{equation}
    \ell_{p;\fol_i,\fol_j}(x) =
      (x_1,\ldots,x_{i-1},u_i(x),
        x_{i+1},\ldots,x_{j-1},u_j(x),x_{j+1},\ldots,x_n)
  \end{equation}
  where the $i^{\text{th}}$ and $j^{\text{th}}$ coordinates of the image
  satisfy
  \begin{equation}
    \begin{aligned}
      u_i(x) &= x_i + 2\kappa_{ij} x^2_ix_j + o(|x|^3)\quad\text{and} \\
      u_j(x) &= x_j - 2\kappa_{ij} x_ix^2_j + o(|x|^3).
    \end{aligned}
  \end{equation}
  respectively (cf. Figure \ref{fig:dfw-curv}).
\end{lem}

The proof of the two propositions above can be reduced to elementary (but quite
tedious) calculations involving the implicit function $g_i$ defined by equation
(\ref{eq:dfw-refl-implicit3}).

\subsection{Geometric triviality conditions}

For planar divergence-free $2$-webs, the curvature of the natural connection
was given a vivid and intuitive geometric interpretation in \cite[Fig.
2]{2-webs}. We are now ready to give its extension to the case of an arbitrary
codimension-$1$ divergence-free $n$-web, which will be used in the next section
to characterize the curvature geometrically in full generality. Before
proceeding, recall the definition of \emph{adjacency} of regions bounded by
leaves; it can be found in Definition \ref{def:dfw-adj}.

\begin{thm}
  \label{thm:dfw-geom}
  Let $\web_\Omega=(M,\fol_1,\ldots,\fol_n,\Omega)$ be a codimension-$1$
  divergence-free $n$-web. The following conditions are equivalent.
  \begin{enumerate}[label=$(\arabic*)$, ref=(\arabic*)]
    \item\label{thm:dfw-geom:triv}
      The web $\web_\Omega$ is locally trivial.
    \item\label{thm:dfw-geom:taba}
      For each pair $\fol,\gol\in\Fols(\web_\Omega)$ of two different
      foliations of $M$, any region bounded by leaves $K$, and any two open
      subsets of leaves $F\in\fol$, $G\in\gol$ which subdivide $K$ into four
      subregions $A,B,C,D$ with $(A\cup B)\cap(C\cup D)\subseteq F$ and $(A\cup
      D)\cap(B\cup C)\subseteq G$, the respective $\Omega$-volumes $a,b,c,d$ of
      $A,B,C,D$ satisfy
      \begin{equation}
        ac=bd.
      \end{equation}
    \item\label{thm:dfw-geom:cut}
      For each pair $\fol,\gol\in\Fols(\web_\Omega)$ of two different
      foliations of $M$, any region bounded by leaves $K$, and any two open
      subsets of leaves $F\in\fol$, $G\in\gol$ which subdivide $K$ into four
      subregions $A,B,C,D$ with $(A\cup B)\cap(C\cup D)\subseteq F$ and $(A\cup
      D)\cap(B\cup C)\subseteq G$ in such a way that the $\Omega$-volumes
      $a,b,c,d$ of $A,B,C,D$ satisfy $a+b = c+d$, the equality $a=b$ implies
      $a=b=c=d$.
    \item\label{thm:dfw-geom:split}
      For any region bounded by leaves $K$ and each $k=1,2,\ldots,n$ there
      exist open subsets of leaves $F_i\in\fol_i$ for $i=1,2,\ldots,k$ which
      subdivide $K$ into $2^k$ subregions with equal $\Omega$-volumes.
    \item\label{thm:dfw-geom:hol}
      The volume-preserving reflection holonomy of $\web_\Omega$ at each
      point $p\in M$ is trivial.
    \item\label{thm:dfw-geom:curv}
      The nonuniformity tensor $\kcurv(\web_\Omega)$ vanishes identically.
  \end{enumerate}
\end{thm}
\begin{proof}
  To prove \ref{thm:dfw-geom:taba} from \ref{thm:dfw-geom:triv},
    carry out direct computations of volumes of $A,B,C,D$ inside a
    $\web_\Omega$-adapted coordinate system $(x_1,\ldots,x_n)$ in which the
    volume form $\Omega$ becomes $dx_1\wedge dx_2\wedge\cdots\wedge dx_n$.
  Now, condition \ref{thm:dfw-geom:cut}
    follows from \ref{thm:dfw-geom:taba} by elementary algebra.

  In order to deduce \ref{thm:dfw-geom:split} from \ref{thm:dfw-geom:cut}
    we proceed by induction on $k$. The base case $k=1$ is covered by the
    intermediate value theorem; if $K=\ival{a,b}$ for $a,b\in\Rb{n}$ in some
    fixed $\web_\Omega$-adapted coordinate system $(x_1,\ldots,x_n)$, then we
    apply the theorem to the continuous function $t\mapsto
    \int_{\ival{a,b_t}}\Omega$ with $b_t = (b_1,\ldots,b_{i-1},t,b_{i+1},
    \ldots,b_n)$ to find a subregion with half the $\Omega$-volume of $K$,
    where we treat $[a,b_t]$ as a nonempty oriented chain according to
    Definition \ref{def:dfw-region}.

    Assume the region $K$ is subdivided into $2^k$ subregions $K_j$
    of equal $\Omega$-volumes by the open subsets of leaves
    $F_i\in\fol_i$ for $i=1,\ldots,k$ and $j=1,\ldots,2^k$. Pick one of the
    subregions $K_{j_0}$ and a leaf $F_{k+1}$ of $\fol_{k+1}$ subdividing each
    region $K_j$ into two subregions $A_j, B_j$ with $\Omega$-volumes $a_j,b_j$
    in such a way, that $a_{j_0}=b_{j_0}$. Since for each region $K_j$ the
    equality $a_j + b_j = a_{j_0} + b_{j_0}$ holds, we obtain
    $a_j=b_j=a_{j_0}=b_{j_0}$ for every subregion $K_j$ adjacent to $K_{j_0}$
    as a consequence of $\ref{thm:dfw-geom:cut}$. Apply the above reasoning
    inductively to arrive at $a_i=b_j$ for $i,j=1,\ldots,2^k$, thereby
    finishing the induction step and concluding the proof of
    \ref{thm:dfw-geom:split}.

\begin{figure}
  \centering
  \def\svgwidth{\textwidth}
  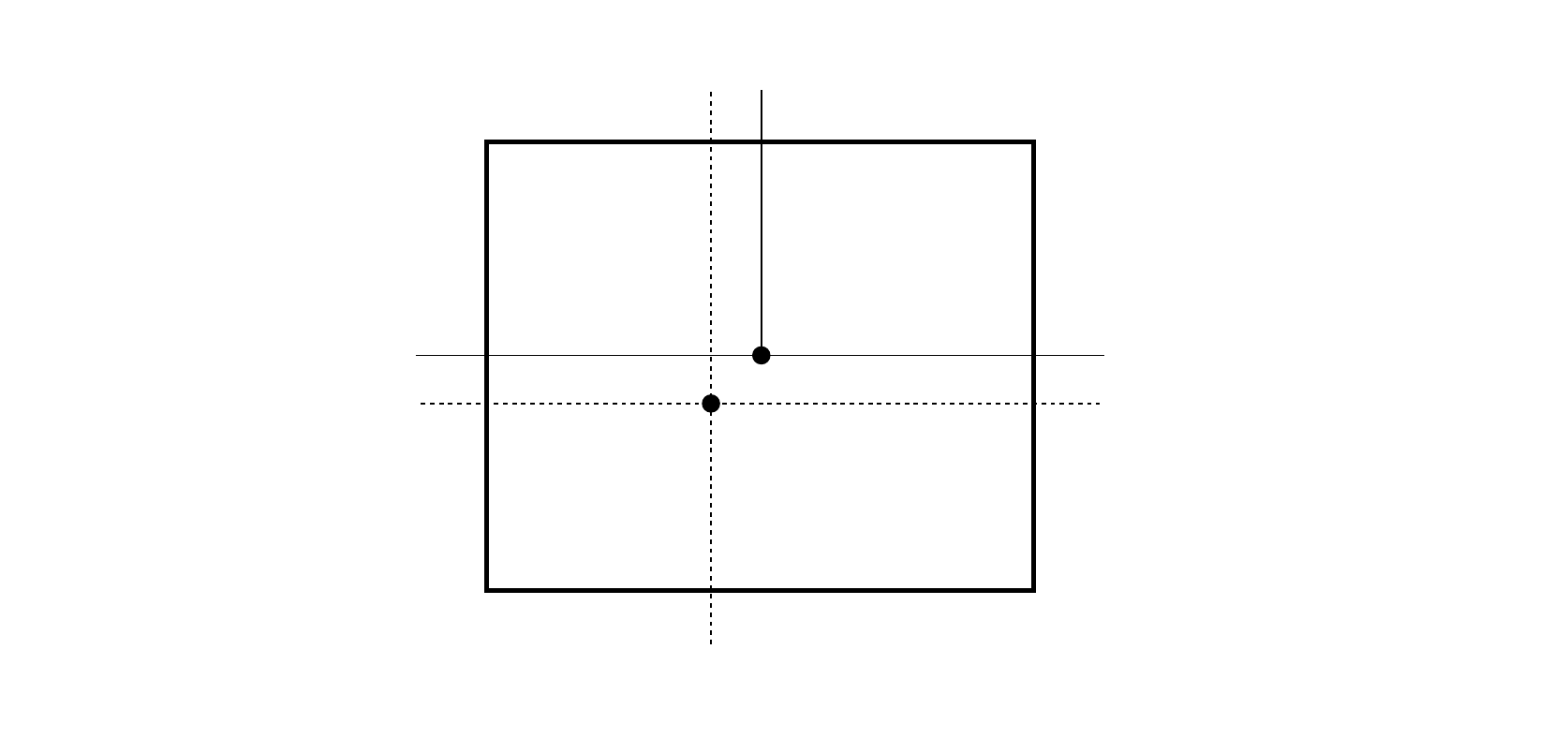
  \caption{\footnotesize If $p\neq q$, then there is a strict containment
    relation between some subregions.}
  \label{fig:dfw-geom-pq}
\end{figure}

  We will now prove the converse implication from \ref{thm:dfw-geom:split} to
    \ref{thm:dfw-geom:cut}. Fix a $\web_\Omega$-adapted coordinate system
    $(x_1,\ldots,x_n)$ on some open subset $U$ containing $K$ and assume that
    locally $\fol,\gol$ are generated by $\ker dx_i,\ker dx_j$ respectively, so
    that $F\subseteq\set{x_i=p_i}$ and $G\subseteq\set{x_j=p_j}$. By our
    assumption, there exist two other open subsets of leaves
    $F'\subseteq\set{x_i=q_i}\in\fol$, $G'\subseteq\set{x_j=q_j}\in\gol$
    subdividing $K$ into subregions $E_1,E_2,E_3,E_4$ with $\Omega$-volumes all
    equal to $e\in\mathbb{R}_+$. Note that $e=\frac{1}{4}(a+b+c+d)$. Since
    $a+b=c+d =2e$, the only possibility is that $p_i=q_i$. Since $a=b=e$, we
    similarly get $p_j=q_j$. Hence $F,F'\subseteq\set{x_i=p_i}$ and
    $G,G'\subseteq\set{x_j=p_j}$, which imply $A=E_1$, $B=E_2$, $C=E_3$ and
    $D=E_4$ after a permutation of indices (Figure \ref{fig:dfw-geom-pq}).
    The subregions $A,B,C,D$ have equal $\Omega$-volumes as intended.

  To show that the ability to cut regions as in \ref{thm:dfw-geom:cut}
    trivializes the holonomy \ref{thm:dfw-geom:hol}, pick a point $p\in M$, a
    $\web_\Omega$-adapted coordinate system centered at $p$ on an open set $W$,
    an open neighbourhood $V\subseteq W$ of $p$ of the coordinate form
    $(-\varepsilon,\varepsilon)^n$ in which every
    volume-preserving reflection is well-defined and unique in the sense
    of Lemma \ref{thm:dfw-refl}, and a neighbourhood $U\subseteq V$ of $p$ such
    that every sufficiently long composition of reflection map-germs
    $r_{p;\fol}$ for $\fol\in\Fols(\web_\Omega)$ maps $U$ into $V$. Our goal is
    to show that for fixed $\fol,\gol\in\Fols(\web_\Omega)$, $\fol\neq\gol$ the
    equality $\ell_{p;\fol,\gol}(q)=q$ holds for each $q\in U$. It is
    sufficient to prove this for points $q$ not lying on any plaque of
    $\web_\Omega$ passing through $p$; the result extends to the whole $U$ by
    the continuity of $\ell_{p;\fol,\gol}$.

    Let $q' = r_{p;\fol}(q)$. The regions bounded by leaves $A=\ival{p,q}$ and
    $B=\ival{p,q'}$ are well-defined and contained in $V$ (since $V$ is a
    coordinate cube), adjacent along the plaque $F_p\in\fol$ crossing $p$ and
    having the same $\Omega$-volumes $\varepsilon>0$. Denote by $G_p$ the
    plaque of $\gol$ crossing $p$. We will now find a plaque $\tilde{G}\in\gol$
    such that the region $R$ adjacent to $A\cup B$ along $G_p$ bounded in the
    $x_2$-direction by $\tilde{G}$ has $\Omega$-volume at least $2\varepsilon$.

    One of the plaques $G_u,G_v\in\gol$ crossing the points $u=r_{p;\gol}(q)$
    and $v=r_{p;\gol}(q')$ must satisfy this criterion. To see this, denote
    the corresponding regions by $R_u, R_v$ and suppose that either of their
    $\Omega$-volumes $r_u, r_v$ is less than $2\varepsilon$ (we can assume
    that $r_u\leq r_v$ without loss of generality). The plaque $F_p$
    subdivides $R_u$ into subregions $S_u = \ival{p,u}, T_u$ and $R_v$ into
    $S_v, T_v=\ival{p,v}$. Note that $\Omega$-volumes $s_u, t_v$ of both
    $S_u$ and $T_v$ are $\varepsilon$ by Definition \ref{def:dfw-reflections}.
    Therefore, the inequality $r_u\leq2\varepsilon$ implies that the
    $\Omega$-volume $t_u$ of $T_u$ is not greater than $\varepsilon$. Since the
    geometry of this setup forces either $T_u\subseteq T_v$ or $T_v\subseteq
    T_u$, by $t_u\leq t_v = \varepsilon$ we obtain that $T_u\subseteq T_v$,
    hence also $S_u \subseteq S_v$. This leads through $\varepsilon = s_u \leq
    s_v$ to $r_v = s_v + t_v = s_v + \varepsilon \geq 2\varepsilon$, proving
    that in this case we can put $\tilde{G} = G_v$, so that $R=R_v$.

    Now, by picking a suitable plaque $G\in\gol$ we can shrink $R$ to a
    region $E$ adjacent to $A\cup B$ along $F_p$ with $\Omega$-volume $e$ equal
    to exactly $2\varepsilon$, which we subsequently subdivide by means of
    $G_p$ into subregions $C, D$ adjacent to $B, A$ respectively with the
    corresponding $\Omega$-volumes $c, d$. Since $a=b=\varepsilon$ and
    $c+d=2\varepsilon$, by our assumption $\ref{thm:dfw-geom:cut}$ we obtain
    $c=d=\varepsilon$. Since $A\cup B\cup C\cup D\subseteq V$ and reflections
    are well-defined and unique on $V$ in the sense of Lemma
    \ref{thm:dfw-refl}, we obtain that $C=\ival{p,q''}$ for
    $q''=r_{p;\gol}(q')$, $D=\ival{p,q^{(3)}}$ for $q^{(3)}=r_{p;\fol}(q'')$
    and finally $A=\ival{p,q}=\ival{p,q^{(4)}}$ for
    $q^{(4)}=r_{p;\gol}(q^{(3)})$, hence $q=q^{(4)}=\ell_{p;\fol,\gol}(q)$ as
    claimed.

  Let us assume \ref{thm:dfw-geom:hol} that the reflection holonomy at every
  point $p\in M$ is trivial.
    Since the loops $\ell_{p;\fol,\gol}$ for each $p\in M$ and
    $\fol,\gol\in\Fols(\web_\Omega)$ are identities, all coefficients in their
    Taylor expansions of order higher than $1$ vanish. The coefficients
    $\kappa_{ij}(p)$ of $\kcurv(\web_\Omega) = \sum_{i\neq j}
    \kappa_{ij}\,dx_idx_j$ all occur as order $3$ coefficients inside the
    expansion of some loop $\ell_{p;\fol,\gol}$ by Lemma
    \ref{thm:dfw-loop-taylor}. Hence, the nonuniformity tensor vanishes
    identically, proving
    \ref{thm:dfw-geom:curv}.

  The last implication from \ref{thm:dfw-geom:curv} to \ref{thm:dfw-geom:triv}
    is exactly the statement of Theorem \ref{thm:dfw-logh}.
\end{proof}

\begin{rem}
  In the above theorem, the equivalence of web triviality with purely local
  conditions $\ref{thm:dfw-geom:hol}$ and $\ref{thm:dfw-geom:curv}$
  allows us to substitute conditions $\ref{thm:dfw-geom:taba}$,
  $\ref{thm:dfw-geom:cut}$, $\ref{thm:dfw-geom:split}$ with their
  localized versions $\ref{thm:dfw-geom:taba}'$,
  $\ref{thm:dfw-geom:cut}'$, $\ref{thm:dfw-geom:split}'$ in which the
  phrase ''\emph{any region bounded by leaves $K$}'' is replaced by
  ''\emph{any sufficiently small region bounded by leaves $K$}''. Validity of
  this assertion becomes more clear after a careful investigation of the
  proofs of the implications between $\ref{thm:dfw-geom:taba}$,
  $\ref{thm:dfw-geom:cut}$, $\ref{thm:dfw-geom:split}$ and
  $\ref{thm:dfw-geom:hol}$; they all can be carried over
  almost verbatim to their localized form.
\end{rem}

We will now give a quantitative version of condition $\ref{thm:dfw-geom:taba}$
of Theorem \ref{thm:dfw-geom}, which gives information about the nonuniformity
tensor directly from non-fulfillment of Tabachnikov's triviality condition
\cite[Fig. 2.]{2-webs}.

\begin{thm}
  \label{thm:dfw-geom-quantitative}
  Let $\web_\Omega$ be a codimension-$1$ divergence-free $n$-web on a smooth
  manifold $M$ and let $p\in M$. Given sufficiently small $u,v\in\Rb{n}$
  defining a region bounded by leaves $K=\ival{u,v}$ with
  $p\in\operatorname{Int}K$ inside a fixed $\web_\Omega$-adapted coordinate
  system $(x_1,\ldots,x_n)$, and two plaques $F=\set{x_i=p_i}\in\fol$,
  $G=\set{x_j=p_j}\in\gol$ of $\fol,\gol\in\Fols(\web_\Omega)$ which subdivide
  $K$ into four subregions $A,B,C,D$ of the form
  \begin{equation}
    \begin{aligned}
      A &= K\cap\set{x_i\leq p_i}\cap\set{x_j\geq p_j}, &
      B &= K\cap\set{x_i\geq p_i}\cap\set{x_j\geq p_j}, \\
      D &= K\cap\set{x_i\leq p_i}\cap\set{x_j\leq p_j}, &
      C &= K\cap\set{x_i\geq p_i}\cap\set{x_j\leq p_j}, \\
    \end{aligned}
  \end{equation}
  the $\Omega$-volumes $a,b,c,d$ of the respective subregions $A,B,C,D$ satisfy
  \begin{equation}
    \label{eq:dfw-geom-pointwise-stmt}
    ac<bd\quad\text{if}\quad \kappa_{ij}>0, \qquad
    ac>bd\quad\text{if}\quad \kappa_{ij}<0,
  \end{equation}
  where $\kappa_{ij}$ are the $dx_i dx_j$-coefficients of the
  nonuniformity tensor $\kcurv(\web_\Omega)_{|p} = \sum_{k\neq l}
  \kappa_{kl}\,dx_k dx_l$ at point $p$.
\end{thm}
\begin{proof}
  We first prove the theorem in the special case of planar divergence-free
  $2$-webs $\web_\Omega$ with a coordinate system $(x,y)$ centered at $p\in M$
  which is normalized in the sense of Lemma \ref{thm:dfw-normal-coords}. In
  such coordinates the volume form $\Omega = h(x,y)\,dx\wedge dy$ satisfies
  $h(x,0)=h(0,y)=1$. Denote by $\kappa$ the only nonzero coefficient of the
  nonuniformity tensor $\kcurv(\web_\Omega)_{|p} = \kappa\,dx\,dy$ at point
  $p$. In these coordinates we can express $h(x,y)$ as
  \begin{equation}
    h(x,y) = 1 + \kappa\,xy + g(x,y) xy
  \end{equation}
  for some continuous function $g$ vanishing at $0$. Hence, for fixed
  $\varepsilon>0$ and small enough $u=(u_1,u_2)\in\Rb{2}$ we have
  \begin{equation}
    \label{eq:dfw-geom-pointwise-int}
    \abs*{\int_{\ival{0,u}}\!\Omega - (u_1u_2 +\smfrac{1}{4}\kappa u_1^2u_2^2)}
      < \smfrac14\varepsilon\,u_1^2u_2^2.
  \end{equation}
  Now, choose $\varepsilon>0$ such that $\abs*{\kappa}>\varepsilon$. Suppose
  that $K=\ival{u,v}$ for some $u_1,u_2 > 0$, $v_1,v_2<0$, and that
  $u_1,u_2,v_1,v_2$ are so close to $0$ that each $(x,y)\in K$ satisfies
  $(\ref{eq:dfw-geom-pointwise-int})$ and $1+(\kappa\pm\varepsilon) xy>0$.
  Then, since $A=[v_1,0]\times[0,u_2]$, $B=[0,u_1]\times[0,u_2]$,
  $C=[0,u_1]\times[v_2,0]$ and $D=[v_1,0]\times[v_2,0]$, we obtain that the
  corresponding $\Omega$-volumes $a,b,c,d$ satisfy
  \begin{align}
    &\begin{aligned}
      &bd-ac = \smint{B}\Omega\cdot\smint{D}\Omega
      - \smint{A}\Omega\cdot\smint{C}\Omega \\
      &\hskip 1em {}>
        (u_1u_2 + \smfrac14(\kappa-\varepsilon) u_1^2u_2^2)\cdot
        (v_1v_2 + \smfrac14(\kappa-\varepsilon) v_1^2v_2^2)\\
      &\hskip 3em {}
      - (v_1u_2 + \smfrac14(\kappa-\varepsilon) v_1^2u_2^2)\cdot
        (u_1v_2 + \smfrac14(\kappa-\varepsilon) u_1^2v_2^2)\\
      &\hskip 1em {}=
        \smfrac14u_1u_2v_1v_2(u_1-v_1)(u_2-v_2)(\kappa-\varepsilon),
      \end{aligned}
    \intertext{and similarly}
      &bd-ac < \smfrac14u_1u_2v_1v_2(u_1-v_1)(u_2-v_2)(\kappa+\varepsilon).
  \end{align}
  Since $u_1u_2v_1v_2(u_1-v_1)(u_2-v_2)>0$, the sign of $bd-ac$ coincides with
  that of $\kappa$, which was to be proved in the planar case.

  The same result holds irrespective of the choice of a
  $\web_\Omega$-coordinate system, since any such set of coordinates $(x,y)$
  can be transformed into one of the above form, say $(\tilde{x},\tilde{y})$,
  by letting $\tilde{x}=\int_0^x h(t,0) dt$ and
  $\tilde{y}=h(0,0)^{-1}\cdot\int_0^y h(0,s)ds$ as in Lemma
  \ref{thm:dfw-normal-coords}. As a result of this transformation, the
  nonuniformity tensor coefficient changes from $\kappa$ to
  $\tilde{\kappa}=\kappa/h(0,0)$. It might change sign in the process, but its
  effect on the statement of the theorem is countered by the way the subregions
  $A,B,C,D$ arrange themselves in the new set of coordinates;
  note that $\smset{p:x(p)>0}=\smset{p:\tilde{x}(p)<0}$ and
  $\smset{p:y(p)>0}=\smset{p:\tilde{y}(p)>0}$ if $h(0,0)<0$ and relabel the
  regions as necessary.

  We now prove the theorem for an arbitrary codimension-$1$ divergence-free
  $n$-web $\web_\Omega$. First, express $\Omega$ in the $\web_\Omega$ adapted
  coordinate chart $(U,\varphi)$ with $\varphi=(x_1,\ldots,x_n)$ as $\Omega =
  h\,dx_1\wedge dx_2\wedge\cdots\wedge dx_n$. Then, assuming without loss of
  generality that $i=1$ and $j=2$, define a family of planar divergence-free
  $2$-webs $\web_{\alpha,\beta}$ depending continuously on parameters
  $\alpha,\beta\in\Rb{n-2}$ in the following way.
  Let $\mathcal{U}$ be the image of $U$ by the projection
  $(x_1,\ldots,x_n)\mapsto (x_1,x_2)=(x,y)$ and let $\fol,\gol$ be generated by
  $\ker dx, \ker dy$ respectively. Let $\Omega_{\alpha,\beta} =
  h_{\alpha,\beta}(x,y)\,dx\wedge dy$, where $h_{\alpha,\beta}$ denotes the
  integral of $h$ over fibers of the projection
  \begin{equation}
    \label{eq:dfw-quantitative-param}
    h_{\alpha,\beta}(x,y) =
      \IterInt{0}{1}{0}{1}
        h(\,x,y,\lambda_1(t_1),\ldots,
          \lambda_{n-2}(t_{n-2})\,)
        \ dt_{n-2}\cdots dt_1,
  \end{equation}
  with $\lambda_k(t) = \alpha_k + t(\beta_k-\alpha_k)$. The webs
  $\web_{\alpha,\beta}$ are defined as
  $\web_{\alpha,\beta}=(\mathcal{U},\fol,\gol,\Omega_{\alpha,\beta})$.

  Pick any two points $u,v\in\mathcal{U}$ and represent them as
  $u=(q^{(1)},\alpha)$ and $v=(q^{(2)},\beta)$ for $q^{(1)},q^{(2)}\in\Rb{2}$
  and $\alpha,\beta\in\Rb{n-2}$. When $\alpha_k\neq\beta_k$ for each
  $k=1,\ldots,n-2$, then a change of variables in
  $(\ref{eq:dfw-quantitative-param})$ yields
  \begin{equation}
    \label{eq:dfw-geom-pointwise-vols}
    \int_{\smash{\ival{q^{(1)},q^{(2)}}}}\Omega_{\alpha,\beta} =
    \frac1{\prod_{k=1}^{n-2}(\alpha_k-\beta_k)}
    \int_{\ival{u,v}}\Omega,
  \end{equation}
  where we treat $\ival{a,b}$ as chains in the sense given in Definition
  \ref{def:dfw-region}. On the other hand, taking $\alpha=\beta$ results in
  equality $h_{\alpha,\beta}(x,y) = h(x,y,\alpha_1,\ldots,\alpha_{n-2})$.
  Straightforward calculation confirms that the nonuniformity tensor of
  $\web_{0,0}$ at $0$ is exactly $\kappa_{ij}=\kappa_{12}$. Since
  $h_{\alpha,\beta}\in C^{\infty}(\mathcal{U})$ depends continuously on
  $\alpha,\beta\in\Rb{n-2}$, so does the coefficient $\tilde{\kappa}_{\alpha,
  \beta}$ of the nonuniformity tensor $\kcurv(\web_{\alpha,\beta})=
  \tilde{\kappa}_{\alpha,\beta}\,dx\,dy$. By continuity, for all $\alpha$,
  $\beta$ in some neighbourhood of $0$ the sign of $\kappa_{\alpha,\beta}$ is
  the same as the sign of $\kappa_{ij}$. Moreover, regions $A,B,C,D$ are all of
  the form $\ival{p,u_k}$ for points $u_k\in\Rb{n}$, $k=1,2,3,4$, differing
  only at the $i^{\text{th}}$ and $j^{\text{th}}$ coordinate. Hence, by the
  planar case and equality $(\ref{eq:dfw-geom-pointwise-vols})$, the
  inequalities $(\ref{eq:dfw-geom-pointwise-stmt})$ hold for regions
  $K=\ival{u,v}$ with sufficiently small $u,v\in\Rb{n}$.
\end{proof}

\subsection{Geometric conditions in higher codimensions}

The following definition outlines one of the possible extensions of Definitions
\ref{def:dfw-region} and \ref{def:dfw-adj} to an arbitrary codimension (Figure
\ref{fig:dfw-region-hicodim}).

\begin{defn}
  Let $\web=(M, \fol_1,\ldots,\fol_n)$ be a regular $n$-web on a smooth
  $m$-dimensional manifold $M$. A compact set $C\subseteq M$ which
  takes the form $C_1\times C_2\times\cdots\times C_n$ for connected closed
  subsets $C_i\subseteq\Rb{c_i}$, $c_i=\codim\fol_i$, in some $\web$-adapted
  chart will be called a \emph{region bounded by leaves of $\web$}. Two
  such regions $A,B$ are \emph{adjacent along $\fol\in\Fols(\web)$} if $A\cup
  B$ is also a region bounded by the leaves of $\web$ and $A\cap B$ is a
  closed fragment of a smooth hypersurface $F\subseteq M$ satisfying
  $T\fol\subseteq TF$. In this case we say that \emph{$F$ subdivides $A\cup B$}
  into subregions $A,B$.
\end{defn}

The volume-preserving loops along $\fol,\gol\in\Fols(\web_\Omega)$ of a
codimension-$1$ divergence-free $n$-web $\web_\Omega$ (Definition
\ref{def:dfw-loops}) allow us to probe the coefficients of the nonuniformity
tensor $\kcurv(\web_\Omega)$ in the directions transverse to $\fol, \gol$ by
means of their Taylor expansions (Lemma \ref{thm:dfw-loop-taylor}). If
$\web_\Omega$ has arbitrary codimension, we can still consider
volume-preserving loops coming from local webs $\web_{\varphi;\Omega}$ with
foliations formed by level-sets of individual coordinate functions
$(x_1,\ldots,x_m)$ inside a $\web_\Omega$-adapted chart $(U,\varphi)$ centered
at $p\in M$. As the natural connection $\Theta$ of $\web_{\varphi;\Omega}$ with
Ricci tensor $\Rc$ is a $\web_\Omega$-connection on $U$, Theorem
\ref{thm:dfw-ricci-general} implies that $\kcurv(\web_\Omega) =
\mathrm{pr}_O(\Rc) = \mathrm{pr}_O(\kcurv(\web_{\varphi;\Omega}))$. With that
in mind, we are able to recover $\kcurv(\web_\Omega)$ from loops through the
foliations generated by $\ker dx_i$, $\ker dx_j$ for $i\not\sim j$, where both
$i\sim j$ and $i\not\sim j$ are defined in $(\ref{eq:dfw-sim})$. This motivates
the following definition of volume preserving holonomy.

\begin{defn}
  Let $M$ be a $m$-dimensional smooth manifold and let
  $\web_\Omega=(M,\fol_1,\ldots,\fol_n,\Omega)$ be a divergence-free $n$-web.
  Consider a $\web_\Omega$-adapted chart $(U,\varphi)$ centered at $p\in M$
  with coordinates $\varphi=(x_1,\ldots,x_m)$ and the induced local
  codimension-$1$ divergence-free $m$-web $\web_{\varphi;\Omega}$ whose
  defining foliations $\gol_i$ satisfy $T\gol_i=\ker dx_i$ for $i=1,\ldots,m$.
  \begin{enumerate}[label=$(\arabic*)$, ref=(\arabic*)]
    \item The smallest normal subgroup $\mathcal{H}_{p;\varphi}$ generated by
      all volume-preserving loops $\ell_{p;\gol_i,\gol_j}$ of
      $\web_{\varphi;\Omega}$ with $i\not\sim j$ inside the group
      $\mathcal{R}_{p;\varphi}$ of all volume-preserving reflections will be
      called the \emph{(volume-preserving) reflection holonomy group of
      $\web_\Omega$ at $p$ in the chart $(U,\varphi)$}.
    \item Let $\mathcal{H}_p$ be the group generated by the union of all
      $\mathcal{H}_{p;\varphi}$, where $\varphi$ runs over all
      $\web_\Omega$-adapted coordinate charts $(U,\varphi)$ centered at some
      fixed point $p\in M$, and let $(V,\psi)$ be a fixed $\web_\Omega$-adapted
      chart. Define $\maps{\psi_*}{\mathcal{H}_p}{\mathcal{H}_{\psi(p)}}$ by
      $\psi_*(f)=\psi\circ f\circ\psi^{-1}$. The conjugacy class of
      $\psi_*\mathcal{H}_p$ inside the group of diffeomorphism-germs
      $\mathrm{Diff}(\Rb{m},0)$ is called the \emph{(volume-preserving)
      reflection holonomy of $\web_\Omega$ at $p$}.
  \end{enumerate}
\end{defn}

\begin{figure}
  \centering
  \def\svgwidth{\textwidth}
  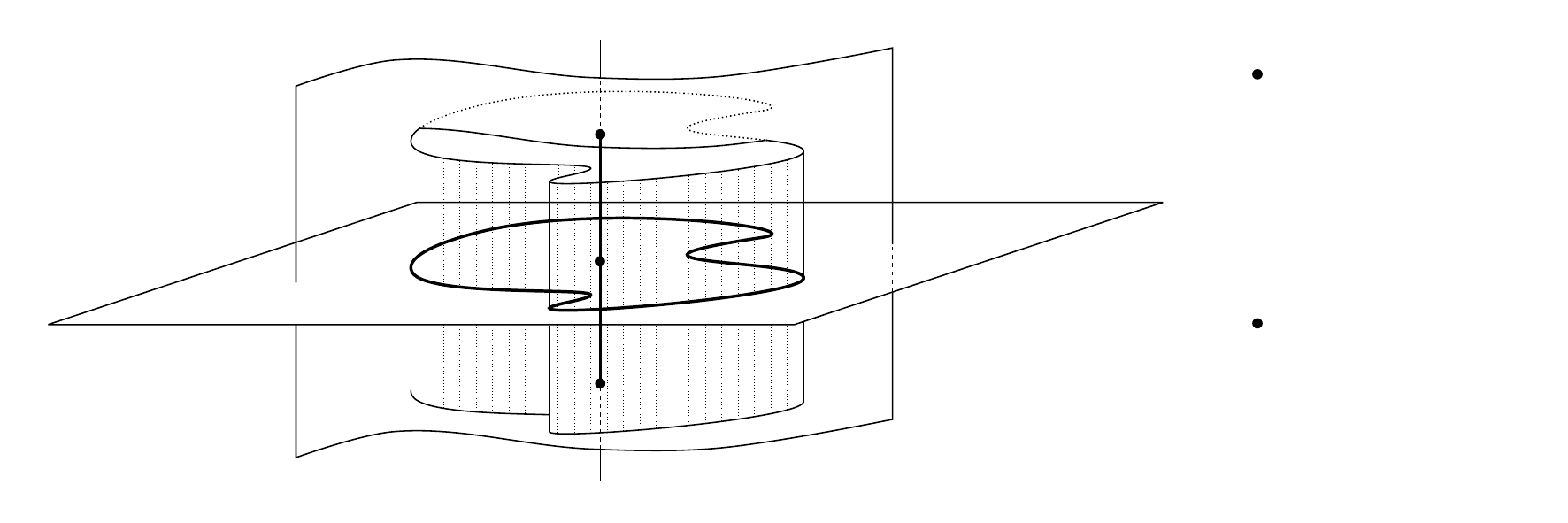
  \caption{\footnotesize A region $K$ bounded by the leaves of the web
    $\web=(\mathbb{R}^3,\fol,\gol)$, where $\codim\fol=1$ and $\codim\gol=2$.
    The set $K$ is of the form $C_1\times C_2$ in $\web$-adapted coordinates.
    Hypersurfaces $F$ and $H$ satisfying $T\fol=TF$ and
    $T\gol\subseteq TH$ subdivide this region into subregions.}
  \label{fig:dfw-region-hicodim}
\end{figure}

These straightforward generalizations allow us to carry over the statement of
Theorem \ref{thm:dfw-geom} almost verbatim to the higher-codimensional case.

\begin{thm}
  \label{thm:dfw-geom-hicodim}
  Let $M$ be a $m$-dimensional smooth manifold and let
  $\web_\Omega=(M,\fol_1,\ldots,\fol_n,\Omega)$ be a divergence-free $n$-web.
  The following conditions are equivalent.
  \begin{enumerate}[label=$(\arabic*)$, ref=(\arabic*)]
    \item\label{thm:dfw-geom-hicodim:triv}
      The web $\web_\Omega$ is locally trivial.
    \item\label{thm:dfw-geom-hicodim:taba}
      For each pair $\fol,\gol\in\Fols(\web_\Omega)$ of two different
      foliations of $M$, any region bounded by leaves $K$, and any two
      hypersurfaces $F,G$ satisfying $T\fol\subseteq TF$, $T\gol\subseteq TG$
      which subdivide $K$ into four subregions $A,B,C,D$ with $(A\cup
      B)\cap(C\cup D)\subseteq F$ and $(A\cup D)\cap(B\cup C)\subseteq G$, the
      respective $\Omega$-volumes $a,b,c,d$ of $A,B,C,D$ satisfy
      \begin{equation}
        ac=bd.
      \end{equation}
    \item\label{thm:dfw-geom-hicodim:cut}
      For each pair $\fol,\gol\in\Fols(\web_\Omega)$ of two different
      foliations of $M$, any region bounded by leaves $K$, and any two
      hypersurfaces $F,G$ satisfying $T\fol\subseteq TF$, $T\gol\subseteq TG$
      which subdivide $K$ into four subregions $A,B,C,D$ with $(A\cup
      B)\cap(C\cup D)\subseteq F$ and $(A\cup D)\cap(B\cup C)\subseteq G$ in
      such a way that the $\Omega$-volumes $a,b,c,d$ of $A,B,C,D$ satisfy $a+b
      = c+d$, the equality $a=b$ implies that $a=b=c=d$.
    \item\label{thm:dfw-geom-hicodim:hol}
      The volume-preserving reflection holonomy of $\web_\Omega$ at each
      point $p\in M$ is trivial.
    \item\label{thm:dfw-geom-hicodim:curv}
      The nonuniformity tensor $\kcurv(\web_\Omega)$ vanishes identically.
  \end{enumerate}
\end{thm}

\begin{rem}
  The naive analogue of the splitting condition $\ref{thm:dfw-geom:split}$ of
  Theorem \ref{thm:dfw-geom}, namely \emph{``for any region bounded by leaves
  $K$ and each $k = 1, 2, \ldots, n$ there exist hypersurfaces $F_i$
  satisfying $T\fol_i\subseteq TF_i$ for $i = 1, 2, \ldots, k$ which subdivide
  $K$ into $2^k$ subregions with equal $\Omega$-volumes''} is not equivalent to
  the other conditions, despite being satisfied by trivial divergence-free
  webs. For example, in codimension $(1,2)$ the above property holds for every
  divergence-free $2$-web $(M,\fol,\gol,\Omega)$ due to a variant of the
  Stone-Tukey ham sandwich theorem. More specifically, we subdivide a
  region bounded by leaves $K$ by a $2$-dimensional plaque $F\in\fol$ into two
  subregions with equal $\Omega$-volumes, and then we choose a coordinate plane
  $H$ with $T\gol\subseteq TH$ subdividing both subregions into subregions with
  equal $\Omega$-volumes by means of the classical Borsuk-Ulam theorem.
\end{rem}
\begin{proof}[Proof of Theorem \ref{thm:dfw-geom-hicodim}]
  The proof proceeds as in Theorem \ref{thm:dfw-geom} with some minor
  modifications, which will be highlighted below.

  The implications from $\ref{thm:dfw-geom-hicodim:triv}$ through
  $(\ref{thm:dfw-geom-hicodim:taba})$ to
  $\ref{thm:dfw-geom-hicodim:cut}$ and from $\ref{thm:dfw-geom-hicodim:curv}$
  to $\ref{thm:dfw-geom-hicodim:triv}$ are established in essentially the same
  way as their codimension-$1$ counterparts.

  To deduce $\ref{thm:dfw-geom-hicodim:hol}$ from
  $\ref{thm:dfw-geom-hicodim:cut}$, note that the holonomy in arbitrary
  codimension is generated by conjugates of the holonomy group
  $\mathcal{H}_{p;\varphi}$ inside some fixed coordinates; it is therefore
  sufficient to trivialize $\mathcal{H}_{p;\varphi}$ for an arbitrary
  coordinate chart $(U,\varphi)$ by following closely the proof in the
  codimension $1$-case. In order to use property
  $\ref{thm:dfw-geom-hicodim:cut}$, let the coordinate function level-sets
  assume the role of the hypersurfaces $F,G$ subdividing the region $K$.

  Finally, assuming $\ref{thm:dfw-geom-hicodim:hol}$, fix a
  $\web_\Omega$-adapted coordinate chart $(U,\varphi)$ with coordinates
  $(x_1,\ldots,x_m)$ and let $\gol_i$ be the foliations with $T\gol_i=\ker
  dx_i$ for $i=1,\ldots,m$ comprising the local web $\web_{\varphi;\Omega}$ on
  $U$. Since the holonomy group $\mathcal{H}_{p;\varphi}$ is trivial by
  assumption, the volume-preserving loops $\ell_{p;\gol_i,\gol_j}$ are all
  identities for $i\not\sim j$ (see $(\ref{eq:dfw-sim})$). The equality
  $\kcurv(\web_\Omega) = \mathrm{pr}_O(\kcurv( \web_{\varphi;\Omega})) =
  \smsum{i\not\sim j} \kappa_{ij}\,dx_i\,dx_j$, which holds by an application
  of Theorem \ref{thm:dfw-ricci-general} to the natural connection of
  $\web_{\varphi;\Omega}$, leads to $\kappa_{ij} = 0$ for each $i\not\sim j$
  and $p\in M$ via Lemma \ref{thm:dfw-loop-taylor}. This is exactly condition
  $\ref{thm:dfw-geom-hicodim:curv}$.
\end{proof}

\section{Applications in general relativity}

In this section we consider divergence-free $2$-webs $\web_{dV}$ arising
from a $4$-dimensional spacetime $(M, g)$ with metric $g$ of
signature $(\mathord{-}\mathord{+}\mathord{+}\mathord{+})$, its natural volume
element $dV$ and a~single codimension-$1$ spacelike foliation $\fol$ of $M$, as
defined below.

The foliation $\fol$ is \emph{spacelike} if the restriction of $g$ to each of
its leaves is positive-definite. This means that the leaves of $\fol$ can be
locally treated as \emph{hypersurfaces of simultaneity} in the ambient
spacetime. The assignment to each event $p\in M$ a plaque $L\ni p$ of $\fol$ in
a certain coordinate system adapted to the foliation $\fol$ can be treated as a
measurement of local time $t$ by some observer; the plaque itself
(roughly) represents the state of the world at time $t$. The orthogonal
complement $T\fol^{\bot}$ of $T\fol$ with respect to $g$ is a $1$-dimensional
distribution, hence is integrable; the corresponding foliation $\fol^\bot$ is
necessarily timelike and can be treated as trajectories of a family of
\emph{Eulerian observers}. Let $dV = \sqrt{-\det g}$ be the pseudo-Riemannian
volume form of the metric $g$. The divergence-free $2$-web $\web_{dV}$ in
question is $(M,\fol,\fol^\bot,dV)$.

The $\web_{dV}$-adapted coordinates are called \emph{normal (Eulerian)
coordinates} \cite{3p1}. In these coordinates, the metric tensor splits into a
sum
\begin{equation}
  \label{eq:dfw-split-metric}
  g = -\alpha^2\,dt^2 + \gamma,
\end{equation}
where the $t$-coordinate parametrizes the leaves of $\fol$, $\alpha$ is the
\emph{lapse function} and $\gamma$ is a positive-definite metric on $T\fol$.
Of course, according to circumstances, one can find complementary foliations
$\gol$ which are better suited to the situation at hand; the above choice is
motivated mainly by the clarity of exposition and can be altered without
difficulty.

\begin{exmp}
  Consider the event horizon exterior of the Schwarzschild spacetime
  $(M, g)$ in standard radial coordinates $(t,r,\theta,\phi)$, where $M =
  \Rb{4}\setminus\set{r\leq 2m}$ for some positive mass $m$ and
  \begin{equation}
    g = -(1-\frac{2m}{r})\,dt^2 + (1-\frac{2m}{r})^{-1}\,dr^2
      + r^2\,d\theta^2 + r^2\sin^2(\theta)\,d\phi^2.
  \end{equation}
  The hypersurfaces of constant time $\set{t=c}$ for $c\in\Rb{1}$ define a
  foliation $\fol$ of $M$. Its orthogonal complement $\fol^{\bot}$ is exactly
  the set of curves $\set{r=c_r, \theta=c_\theta, \phi=c_\phi}$ for
  $c_r,c_\theta,c_\phi\in\Rb{1}$ by virtue of the form of $g$ in these
  coordinates. We also have
  \begin{equation}
    dV = r^2\sin\theta\cdot dt\wedge dr\wedge d\theta\wedge d\phi.
  \end{equation}
  Since the coordinate system is $\web_{dV}$-adapted, the divergence-free
  $2$-web $\web_{dV}$ is trivial by Theorem \ref{thm:dfw-logh}.
\end{exmp}

\begin{exmp}
  The constant-$T$ foliation $\fol$ corresponding to Gullstrand-Painlevé
  coordinates $(T, r,\theta, \phi)$ on the Schwarzschild spacetime
  $(\Rb{1}\times(\Rb{3}\setminus\set{0}),g)$, in
  which the metric $g$ takes the form
  \begin{equation}
    g = -(1-\frac{2m}{r})\,dT^2 + 2\sqrt{\frac{2m}{r}}\,dTdr + dr^2
      + r^2\,d\theta^2 + r^2\sin^2(\theta)\,d\phi^2,
  \end{equation}
  is spacelike and genreates a divergence-free $2$-web $\web_{dV}$ which is
  not trivial. To see this, switch to a $\web_{dV}$-adapted coordinate system
  $(T,R,\theta,\phi)$ with $R=\smfrac{2}{3}r^{3/2}+\sqrt{2m}T$ (known as
  Lemaître coordinate system) to bring $g$ into the split form
  $(\ref{eq:dfw-split-metric})$
  \begin{equation}
    g = -dT^2 + r^{-1}\,dR^2 + r^{2}\,d\theta^2
      + r^{2}\sin^2(\theta)\,d\phi^2,\quad r
      = (\smfrac{3}{2}R-\smfrac{3}{2}\sqrt{2m}T)^{2/3}.
  \end{equation}
  The claim follows from Theorem \ref{thm:dfw-logh}, since the
  corresponding volume form
  \begin{equation}
    dV = \smfrac{3}{2}(R-\sqrt{2m}T)\sin\theta
      \cdot dT\wedge dR\wedge d\theta\wedge d\phi
  \end{equation}
  yields a nonvanishing nonuniformity tensor $\kcurv(\web_{dV}) =
  \sqrt{2m}(R-\sqrt{2m}T)^{-2}\,dTdR$. Its form in the initial
  Gullstrand-Painlevé coordinate system is
  \begin{equation}
    \kcurv(\web_{dV}) = \smfrac{9}{2}mr^{-3}\,dT^2
     + \smfrac{9}{4}\sqrt{2m}\,r^{-5/2}\,dTdr.
  \end{equation}
  Since $g(\basis{T}, \basis{T}) = -1$, the variable $T$ can be interpreted as
  the proper time of Eulerian observers falling radially into the Schwarzschild
  black hole. Near the singularity itself, the nonuniformity tensor grows
  without bound.
\end{exmp}

The triviality of the induced divergence-free $2$-web has some relevance in
relativistic fluid dynamics, where the laws of motion of the fluid expressed in
$\web_{dV}$-adapted coordinates include the conservation of energy
$(\ref{eq:dfw-gr-ene})$ and momentum $(\ref{eq:dfw-gr-mom})$, both of which
depend on the spacetime volume element $dV$ in a nontrivial way. These laws,
expressed using the notation introduced in \cite[Part 11]{andersson1} and the
Einstein summation convention, reduce to
\begin{equation}
  \nabla_\alpha T^{\alpha\beta} = 0,
\end{equation}
for the stress-energy tensor $T$ bound by certain constitutive relations, where
$\nabla_\alpha$ denotes the Levi-Civita covariant differentiation. The above
condition on the divergence of $T$ results in appearance of the Lorentzian
volume density in the resulting equations. In normal coordinates, assuming that
in the tensors below the $0^\text{th}$ index corresponds to time-like
components, while the indices $i,j,k=1,2,3$ refer to the spatial ones, the
equations read \cite{lefloch}
\begin{align}
  \label{eq:dfw-gr-ene}
  \begin{split}
    \basis{t}(\alpha^3\gamma^{1/2}T^{00})
      + \basis{x_j}(\alpha^3\gamma^{1/2}T^{0j})
      &= \gamma^{1/2}\alpha^2(\basis{t}\alpha)T^{00} -
        \smfrac{1}{2}\gamma^{1/2}\alpha(\basis{t}g_{jk})T^{jk}
  \end{split}\\
  \label{eq:dfw-gr-mom}
  \begin{split}
    \basis{t}(\alpha^3\gamma^{1/2}T^{i0})
      + \basis{x_j}(\alpha^3\gamma^{1/2}T^{ij})
      &= \smfrac{1}{2}\gamma^{1/2}\alpha^3g^{ik}(\basis{x_k}\alpha^2)T^{00}
        -\gamma^{1/2}\alpha^3g^{ik}(\basis{t}g_{jk})T^{j0} \\
        &{}-\gamma^{1/2}\alpha^3\Gamma^i_{jk}T^{jk}
        -T^{0i}\basis{t}(\gamma^{1/2}\alpha^3) \\
        &{}+2T^{0i}\gamma^{1/2}\alpha^2(\basis{t}\alpha)
        +2\alpha^3T^{ij}(\basis{x_j}\gamma^{1/2}),
  \end{split}
\end{align}
where $\alpha$ is the lapse function, $\Gamma_{jk}^i$ are the Christoffel
symbols and $\gamma^{1/2}$ is the coefficient of
the spacelike volume element on $T\fol$ induced from $dV$.

If the nonuniformity tensor $\kcurv(\web_{dV})$ of the web $\web_{dV}$
defined in the introduction vanishes identically, then by Theorem
\ref{thm:dfw-logh} one can pick a $\web_{dV}$-adapted coordinate system
$(t,x_1,x_2,x_3)$ in which the density of the volume element $dV =
\alpha\gamma^{1/2}\, dt\wedge dx_1\wedge dx_2\wedge dx_3$ becomes constant.
This simplifies the above equations, completely eliminating their dependence on
$\gamma^{1/2}$. If we assume that $\alpha\gamma^{1/2} = 1$, then from
$(\ref{eq:dfw-gr-ene})$ and $(\ref{eq:dfw-gr-mom})$ we obtain
\begin{align}
  \label{eq:dfw-gr1-ene}
  \begin{split}
    \basis{t}(\alpha^2T^{00}) + \basis{x_j}(\alpha^2 T^{0j})
      &= \smfrac{1}{2}(\basis{t}\alpha^2)T^{00} -
        \smfrac{1}{2}(\basis{t}g_{jk})T^{jk}
  \end{split}\\
  \label{eq:dfw-gr1-mom}
  \begin{split}
    \basis{t}(\alpha^2T^{i0}) + \basis{x_j}(\alpha^2T^{ij})
      &= \smfrac{1}{2}\alpha^2g^{ik}(\basis{x_k}\alpha^2)T^{00}
        -\alpha^2g^{ik}(\basis{t}g_{jk})T^{j0} \\
        &{}-\alpha^2\Gamma^i_{jk}T^{jk}
        -(\basis{x_j}\alpha^2)T^{ij}.
  \end{split}
\end{align}

This form of the equations is more convenient to work with, and due to its
simplicity it might improve numerical accuracy if $\alpha^2$ and $g$ are well
behaved in the new coordinates. Given $\alpha^2$ and the other components of
the metric, one can avoid taking any square roots and determinants in order to
evolve the stress-energy tensor $T$ in time. The equations can be simplified
further if we are able to choose $\fol$ such that $\fol^\bot$ is \emph{totally
geodesic}. This amounts to the equality $d\alpha_{|T\fol}=0$ by a direct
verification of the identity $g(\nabla_n n, v) = d\log\alpha(v)$ for the unit
normal $n\in T\fol^\bot$ with $\norm{n}_{g} = -1$ and each $v\in T\fol$. In
this setting, the leaves of the orthogonal foliation $T\fol^\bot$ are
geodesics; a foliation $\fol$ of this kind is called a \emph{geodesic
slicing} of the spacetime \cite[{}10.2]{3p1}.

One last thing to note is that the Levi-Civita connection $\nabla$ is
\emph{not} a $\web_{dV}$-connection in general (see Definition
\ref{def:dfw-affine}). It is so only if both $\fol$ and $\fol^\bot$ are totally
geodesic, in the sense that for all $X,Z\in\Gamma(T\fol)$ and
$Y,W\in\Gamma(T\fol^\bot)$
\begin{equation}
  \nabla_ZX\in\Gamma(T\fol),\qquad\qquad \nabla_WY\in\Gamma(T\fol^\bot).
\end{equation}
This condition is equivalent to $d\alpha_{|T\fol}=0$ and
$K_{ij}=\frac{1}{2\alpha}\basis{t}\gamma_{ij}=0$ for $i,j=1,2,3$. As a
consequence, it is usually not possible to determine local triviality of
$\web_{dV}$ using the Ricci tensor of the Levi-Civita connection alone, as the
results of section \ref{sec:dfw-nonuniformity} seem to suggest, except in quite
special geometries defined by the above two sets of equations.

  \printbibliography

\end{document}

%% file: graphics/planar.pdf_tex
\begingroup%
  \makeatletter%
  \providecommand\color[2][]{%
    \errmessage{(Inkscape) Color is used for the text in Inkscape, but the package 'color.sty' is not loaded}%
    \renewcommand\color[2][]{}%
  }%
  \providecommand\transparent[1]{%
    \errmessage{(Inkscape) Transparency is used (non-zero) for the text in Inkscape, but the package 'transparent.sty' is not loaded}%
    \renewcommand\transparent[1]{}%
  }%
  \providecommand\rotatebox[2]{#2}%
  \newcommand*\fsize{\dimexpr\f@size pt\relax}%
  \newcommand*\lineheight[1]{\fontsize{\fsize}{#1\fsize}\selectfont}%
  \ifx\svgwidth\undefined%
    \setlength{\unitlength}{510.23622047bp}%
    \ifx\svgscale\undefined%
      \relax%
    \else%
      \setlength{\unitlength}{\unitlength * \real{\svgscale}}%
    \fi%
  \else%
    \setlength{\unitlength}{\svgwidth}%
  \fi%
  \global\let\svgwidth\undefined%
  \global\let\svgscale\undefined%
  \makeatother%
  \begin{picture}(1,0.33333333)%
    \lineheight{1}%
    \setlength\tabcolsep{0pt}%
    \put(0,0){\includegraphics[width=\unitlength,page=1]{planar.pdf}}%
    \put(0.19212098,0.21098701){\makebox(0,0)[lt]{\lineheight{1.25}\smash{\begin{tabular}[t]{l}$A$\end{tabular}}}}%
    \put(0.3070373,0.219601){\makebox(0,0)[lt]{\lineheight{1.25}\smash{\begin{tabular}[t]{l}$B$\end{tabular}}}}%
    \put(0.32184826,0.12679596){\makebox(0,0)[lt]{\lineheight{1.25}\smash{\begin{tabular}[t]{l}$C$\end{tabular}}}}%
    \put(0.1832394,0.12685907){\makebox(0,0)[lt]{\lineheight{1.25}\smash{\begin{tabular}[t]{l}$D$\end{tabular}}}}%
    \put(0,0){\includegraphics[width=\unitlength,page=2]{planar.pdf}}%
    \put(0.80989816,0.23330147){\color[rgb]{0,0,0}\makebox(0,0)[lt]{\begin{minipage}{0.14487617\unitlength}\raggedright $\varepsilon$\end{minipage}}}%
    \put(0.86082434,0.25962432){\color[rgb]{0,0,0}\makebox(0,0)[lt]{\begin{minipage}{0.14487617\unitlength}\raggedright $q$\end{minipage}}}%
    \put(0.6958292,0.23330147){\color[rgb]{0,0,0}\makebox(0,0)[lt]{\begin{minipage}{0.14487617\unitlength}\raggedright $\varepsilon$\end{minipage}}}%
    \put(0.80989816,0.11920736){\color[rgb]{0,0,0}\makebox(0,0)[lt]{\begin{minipage}{0.14487617\unitlength}\raggedright $\varepsilon$\end{minipage}}}%
    \put(0.54724036,0.28982318){\color[rgb]{0,0,0}\makebox(0,0)[lt]{\begin{minipage}{0.14487617\unitlength}\raggedright $r_{p;\fol}(q)$\end{minipage}}}%
    \put(0.84975547,0.0252587){\color[rgb]{0,0,0}\makebox(0,0)[lt]{\begin{minipage}{0.14487617\unitlength}\raggedright $r_{p;\gol}(q)$\end{minipage}}}%
    \put(0.79811685,0.19598556){\color[rgb]{0,0,0}\makebox(0,0)[lt]{\begin{minipage}{0.14487617\unitlength}\raggedright $p$\end{minipage}}}%
  \end{picture}%
\endgroup%

%% file: graphics/spirals.pdf_tex
\begingroup%
  \makeatletter%
  \providecommand\color[2][]{%
    \errmessage{(Inkscape) Color is used for the text in Inkscape, but the package 'color.sty' is not loaded}%
    \renewcommand\color[2][]{}%
  }%
  \providecommand\transparent[1]{%
    \errmessage{(Inkscape) Transparency is used (non-zero) for the text in Inkscape, but the package 'transparent.sty' is not loaded}%
    \renewcommand\transparent[1]{}%
  }%
  \providecommand\rotatebox[2]{#2}%
  \newcommand*\fsize{\dimexpr\f@size pt\relax}%
  \newcommand*\lineheight[1]{\fontsize{\fsize}{#1\fsize}\selectfont}%
  \ifx\svgwidth\undefined%
    \setlength{\unitlength}{510.23622047bp}%
    \ifx\svgscale\undefined%
      \relax%
    \else%
      \setlength{\unitlength}{\unitlength * \real{\svgscale}}%
    \fi%
  \else%
    \setlength{\unitlength}{\svgwidth}%
  \fi%
  \global\let\svgwidth\undefined%
  \global\let\svgscale\undefined%
  \makeatother%
  \begin{picture}(1,0.38888889)%
    \lineheight{1}%
    \setlength\tabcolsep{0pt}%
    \put(0,0){\includegraphics[width=\unitlength,page=1]{spirals.pdf}}%
  \end{picture}%
\endgroup%

%% file: graphics/loops.pdf_tex
\begingroup%
  \makeatletter%
  \providecommand\color[2][]{%
    \errmessage{(Inkscape) Color is used for the text in Inkscape, but the package 'color.sty' is not loaded}%
    \renewcommand\color[2][]{}%
  }%
  \providecommand\transparent[1]{%
    \errmessage{(Inkscape) Transparency is used (non-zero) for the text in Inkscape, but the package 'transparent.sty' is not loaded}%
    \renewcommand\transparent[1]{}%
  }%
  \providecommand\rotatebox[2]{#2}%
  \newcommand*\fsize{\dimexpr\f@size pt\relax}%
  \newcommand*\lineheight[1]{\fontsize{\fsize}{#1\fsize}\selectfont}%
  \ifx\svgwidth\undefined%
    \setlength{\unitlength}{481.88976378bp}%
    \ifx\svgscale\undefined%
      \relax%
    \else%
      \setlength{\unitlength}{\unitlength * \real{\svgscale}}%
    \fi%
  \else%
    \setlength{\unitlength}{\svgwidth}%
  \fi%
  \global\let\svgwidth\undefined%
  \global\let\svgscale\undefined%
  \makeatother%
  \begin{picture}(1,0.47058824)%
    \lineheight{1}%
    \setlength\tabcolsep{0pt}%
    \put(0,0){\includegraphics[width=\unitlength,page=1]{loops.pdf}}%
    \put(0.77771967,0.21487979){\makebox(0,0)[lt]{\lineheight{1.25}\smash{\begin{tabular}[t]{l}$p$\end{tabular}}}}%
    \put(0.80318112,0.32823427){\makebox(0,0)[lt]{\lineheight{1.25}\smash{\begin{tabular}[t]{l}$q_1$\end{tabular}}}}%
    \put(0.87975349,0.25681834){\makebox(0,0)[lt]{\lineheight{1.25}\smash{\begin{tabular}[t]{l}$q_4$\end{tabular}}}}%
    \put(0.5917421,0.39591809){\makebox(0,0)[lt]{\lineheight{1.25}\smash{\begin{tabular}[t]{l}$L_k$\end{tabular}}}}%
    \put(0.71523583,0.12569282){\makebox(0,0)[lt]{\lineheight{1.25}\smash{\begin{tabular}[t]{l}$L_i$\end{tabular}}}}%
    \put(0.64077384,0.1921943){\makebox(0,0)[lt]{\lineheight{1.25}\smash{\begin{tabular}[t]{l}$L_j$\end{tabular}}}}%
    \put(0,0){\includegraphics[width=\unitlength,page=2]{loops.pdf}}%
    \put(0.40313255,0.27116528){\makebox(0,0)[lt]{\lineheight{1.25}\smash{\begin{tabular}[t]{l}$L_i$\end{tabular}}}}%
    \put(0.18566884,0.37122628){\makebox(0,0)[lt]{\lineheight{1.25}\smash{\begin{tabular}[t]{l}$L_j$\end{tabular}}}}%
    \put(0.07390777,0.30993853){\makebox(0,0)[lt]{\lineheight{1.25}\smash{\begin{tabular}[t]{l}$L_k$\end{tabular}}}}%
    \put(0.35018686,0.19796006){\makebox(0,0)[lt]{\lineheight{1.25}\smash{\begin{tabular}[t]{l}$q_1$\end{tabular}}}}%
    \put(0.27809583,0.22824857){\makebox(0,0)[lt]{\lineheight{1.25}\smash{\begin{tabular}[t]{l}$q_6$\end{tabular}}}}%
    \put(0.22102485,0.13573769){\makebox(0,0)[lt]{\lineheight{1.25}\smash{\begin{tabular}[t]{l}$p$\end{tabular}}}}%
    \put(0.41633592,0.14626833){\makebox(0,0)[lt]{\lineheight{1.25}\smash{\begin{tabular}[t]{l}\large$\fol_k$\end{tabular}}}}%
    \put(0.26340914,0.29956407){\makebox(0,0)[lt]{\lineheight{1.25}\smash{\begin{tabular}[t]{l}$s_{ij}$\end{tabular}}}}%
  \end{picture}%
\endgroup%

%% file: graphics/curv.pdf_tex
\begingroup%
  \makeatletter%
  \providecommand\color[2][]{%
    \errmessage{(Inkscape) Color is used for the text in Inkscape, but the package 'color.sty' is not loaded}%
    \renewcommand\color[2][]{}%
  }%
  \providecommand\transparent[1]{%
    \errmessage{(Inkscape) Transparency is used (non-zero) for the text in Inkscape, but the package 'transparent.sty' is not loaded}%
    \renewcommand\transparent[1]{}%
  }%
  \providecommand\rotatebox[2]{#2}%
  \newcommand*\fsize{\dimexpr\f@size pt\relax}%
  \newcommand*\lineheight[1]{\fontsize{\fsize}{#1\fsize}\selectfont}%
  \ifx\svgwidth\undefined%
    \setlength{\unitlength}{340.15748031bp}%
    \ifx\svgscale\undefined%
      \relax%
    \else%
      \setlength{\unitlength}{\unitlength * \real{\svgscale}}%
    \fi%
  \else%
    \setlength{\unitlength}{\svgwidth}%
  \fi%
  \global\let\svgwidth\undefined%
  \global\let\svgscale\undefined%
  \makeatother%
  \begin{picture}(1,0.375)%
    \lineheight{1}%
    \setlength\tabcolsep{0pt}%
    \put(0,0){\includegraphics[width=\unitlength,page=1]{curv.pdf}}%
    \put(0.28302536,0.04027581){\color[rgb]{0,0,0}\makebox(0,0)[lt]{\begin{minipage}{0.15628359\unitlength}\raggedright $k_{ij}>0$\end{minipage}}}%
    \put(0.76973524,0.04027581){\color[rgb]{0,0,0}\makebox(0,0)[lt]{\begin{minipage}{0.15628359\unitlength}\raggedright $k_{ij}<0$\end{minipage}}}%
    \put(0.44199332,0.21458785){\color[rgb]{0,0,0}\makebox(0,0)[lt]{\begin{minipage}{0.15628359\unitlength}\raggedright $x_i$\end{minipage}}}%
    \put(0.92530573,0.21458785){\color[rgb]{0,0,0}\makebox(0,0)[lt]{\begin{minipage}{0.15628359\unitlength}\raggedright $x_i$\end{minipage}}}%
    \put(0.27173943,0.35571653){\color[rgb]{0,0,0}\makebox(0,0)[lt]{\begin{minipage}{0.15628359\unitlength}\raggedright $x_j$\end{minipage}}}%
    \put(0.75779629,0.35571653){\color[rgb]{0,0,0}\makebox(0,0)[lt]{\begin{minipage}{0.15628359\unitlength}\raggedright $x_j$\end{minipage}}}%
  \end{picture}%
\endgroup%

%% file: graphics/3to4.pdf_tex
\begingroup%
  \makeatletter%
  \providecommand\color[2][]{%
    \errmessage{(Inkscape) Color is used for the text in Inkscape, but the package 'color.sty' is not loaded}%
    \renewcommand\color[2][]{}%
  }%
  \providecommand\transparent[1]{%
    \errmessage{(Inkscape) Transparency is used (non-zero) for the text in Inkscape, but the package 'transparent.sty' is not loaded}%
    \renewcommand\transparent[1]{}%
  }%
  \providecommand\rotatebox[2]{#2}%
  \newcommand*\fsize{\dimexpr\f@size pt\relax}%
  \newcommand*\lineheight[1]{\fontsize{\fsize}{#1\fsize}\selectfont}%
  \ifx\svgwidth\undefined%
    \setlength{\unitlength}{481.88976378bp}%
    \ifx\svgscale\undefined%
      \relax%
    \else%
      \setlength{\unitlength}{\unitlength * \real{\svgscale}}%
    \fi%
  \else%
    \setlength{\unitlength}{\svgwidth}%
  \fi%
  \global\let\svgwidth\undefined%
  \global\let\svgscale\undefined%
  \makeatother%
  \begin{picture}(1,0.47058824)%
    \lineheight{1}%
    \setlength\tabcolsep{0pt}%
    \put(0,0){\includegraphics[width=\unitlength,page=1]{3to4.pdf}}%
    \put(0.32356641,0.34445363){\makebox(0,0)[lt]{\lineheight{1.25}\smash{\begin{tabular}[t]{l}$A$\end{tabular}}}}%
    \put(0.6170574,0.34445363){\makebox(0,0)[lt]{\lineheight{1.25}\smash{\begin{tabular}[t]{l}$B$\end{tabular}}}}%
    \put(0.61672808,0.1151608){\makebox(0,0)[lt]{\lineheight{1.25}\smash{\begin{tabular}[t]{l}$C$\end{tabular}}}}%
    \put(0.67555672,0.25992528){\makebox(0,0)[lt]{\lineheight{1.25}\smash{\begin{tabular}[t]{l}$F_i$\end{tabular}}}}%
    \put(0.49609424,0.40033343){\makebox(0,0)[lt]{\lineheight{1.25}\smash{\begin{tabular}[t]{l}$F_{k+1}$\end{tabular}}}}%
    \put(0.49803866,0.2640312){\makebox(0,0)[lt]{\lineheight{1.25}\smash{\begin{tabular}[t]{l}$p$\end{tabular}}}}%
    \put(0.42337171,0.18338574){\makebox(0,0)[lt]{\lineheight{1.25}\smash{\begin{tabular}[t]{l}$q$\end{tabular}}}}%
    \put(0.26020961,0.1805454){\makebox(0,0)[lt]{\lineheight{1.25}\smash{\begin{tabular}[t]{l}$G_i$\end{tabular}}}}%
    \put(0.37032832,0.05771453){\makebox(0,0)[lt]{\lineheight{1.25}\smash{\begin{tabular}[t]{l}$G_{k+1}$\end{tabular}}}}%
  \end{picture}%
\endgroup%

%% file: graphics/region.pdf_tex
\begingroup%
  \makeatletter%
  \providecommand\color[2][]{%
    \errmessage{(Inkscape) Color is used for the text in Inkscape, but the package 'color.sty' is not loaded}%
    \renewcommand\color[2][]{}%
  }%
  \providecommand\transparent[1]{%
    \errmessage{(Inkscape) Transparency is used (non-zero) for the text in Inkscape, but the package 'transparent.sty' is not loaded}%
    \renewcommand\transparent[1]{}%
  }%
  \providecommand\rotatebox[2]{#2}%
  \newcommand*\fsize{\dimexpr\f@size pt\relax}%
  \newcommand*\lineheight[1]{\fontsize{\fsize}{#1\fsize}\selectfont}%
  \ifx\svgwidth\undefined%
    \setlength{\unitlength}{510.23622047bp}%
    \ifx\svgscale\undefined%
      \relax%
    \else%
      \setlength{\unitlength}{\unitlength * \real{\svgscale}}%
    \fi%
  \else%
    \setlength{\unitlength}{\svgwidth}%
  \fi%
  \global\let\svgwidth\undefined%
  \global\let\svgscale\undefined%
  \makeatother%
  \begin{picture}(1,0.33333333)%
    \lineheight{1}%
    \setlength\tabcolsep{0pt}%
    \put(0,0){\includegraphics[width=\unitlength,page=1]{region.pdf}}%
    \put(0.66138204,0.14728355){\makebox(0,0)[lt]{\lineheight{1.25}\smash{\begin{tabular}[t]{l}$\fol$\end{tabular}}}}%
    \put(0.39564841,0.30289409){\makebox(0,0)[lt]{\lineheight{1.25}\smash{\begin{tabular}[t]{l}$\gol$\end{tabular}}}}%
    \put(0,0){\includegraphics[width=\unitlength,page=2]{region.pdf}}%
    \put(0.89345356,0.02644077){\makebox(0,0)[lt]{\lineheight{1.25}\smash{\begin{tabular}[t]{l}$C_2$\end{tabular}}}}%
    \put(0.83472705,0.27802709){\makebox(0,0)[lt]{\lineheight{1.25}\smash{\begin{tabular}[t]{l}$C_1$\end{tabular}}}}%
    \put(0.60258276,0.17539459){\makebox(0,0)[lt]{\lineheight{1.25}\smash{\begin{tabular}[t]{l}$F$\end{tabular}}}}%
    \put(0.52984678,0.25765735){\makebox(0,0)[lt]{\lineheight{1.25}\smash{\begin{tabular}[t]{l}$H$\end{tabular}}}}%
    \put(0,0){\includegraphics[width=\unitlength,page=3]{region.pdf}}%
    \put(0.46825231,0.08530418){\makebox(0,0)[lt]{\lineheight{1.25}\smash{\begin{tabular}[t]{l}$K$\end{tabular}}}}%
  \end{picture}%
\endgroup%